\documentclass[11pt,a4paper]{article}
\usepackage[utf8]{inputenc}
\usepackage[T1]{fontenc}
\usepackage{amsmath}
\usepackage{amsfonts}
\usepackage{amssymb}
\usepackage{amsthm}
\usepackage{bm}
\usepackage{enumitem}
\setlist{itemsep=1pt,topsep=2pt,parsep=1pt}
\usepackage{upgreek}

\usepackage{authblk}

\theoremstyle{plain}
\newtheorem{theorem}{Theorem}[section]
\newtheorem{assumption}{Assumption}

\newtheorem{proposition}[theorem]{Proposition}

\newtheorem{lemma}[theorem]{Lemma}
\newtheorem{example}{Example}
\newtheorem{remark}[theorem]{Remark}
\newtheorem{corollary}[theorem]{Corollary}
\newtheorem*{remarque*}{Remark}

\newcounter{countD}

\newtheorem{assumpprime}[countD]{Assumption}

\newcommand{\sign}{\mathrm{sign}}
\newcommand{\e}{\mathrm{e}}
\newcommand{\m}{\mathfrak{m}}
\newcommand{\f}{\mathfrak{F}}
\newcommand{\hq}{\tilde{q}}
\newcommand{\he}{\tilde{e}}

\newcommand{\bbP}{\mathbb{P}}
\newcommand{\bbE}{\mathbb{E}}
\newcommand{\bbR}{\mathbb{R}}
\renewcommand{\epsilon}{\varepsilon}
\newcommand{\gep}{\epsilon}

\newcommand{\dr}{\mathrm{d}}
\newcommand{\dd}{\mathrm{d}}
\newcommand{\nn}{\mathfrak{n}}
\newcommand{\sca}{\mathfrak{s}}

\newcommand{\mm}{\mathfrak{m}}
\newcommand{\ind}{\bm{1}}

\usepackage[left=2.5cm,right=2.5cm,top=2cm,bottom=2cm]{geometry}
\usepackage{graphicx}
\usepackage{subcaption}
\usepackage{float}
\usepackage{verbatim}
\usepackage[colorlinks=true,
            linkcolor=black,
            urlcolor=black,
            citecolor=black]{hyperref}
\usepackage{cite}
\usepackage{dsfont}
\usepackage[multiple]{footmisc}

\numberwithin{equation}{section}

\author[*,$\dagger$]{Quentin Berger}
\author[*]{Lo\"ic B\'ethencourt}
\author[*]{Camille Tardif}
\affil[*]{\footnotesize Laboratoire de Probabilit\'e Statistique et Mod\'elisation, Sorbonne Universit\'e.}
\affil[$\dagger$]{\footnotesize D\'epartement de Math\'ematiques et Applications, \'Ecole Normale Sup\'erieure, PSL.}

\title{Persistence problems for additive functionals\\
of one-dimensional Markov processes}

\date{}

\begin{document}

\maketitle

\begin{abstract}
\noindent
In this article, we consider additive functionals $\zeta_t = \int_0^t f(X_s)\dd s$ of a c\`adl\`ag Markov process $(X_t)_{t\geq 0}$ on $\mathbb R$.
Under some general conditions on the process $(X_t)_{t\geq 0}$ and on the function~$f$, we show that the persistence probabilities verify $\bbP(\zeta_s < z  \text{ for all } s\leq t ) \sim \mathcal V(z) \varsigma(t)\, t^{-\theta}$ as $t\to\infty$, for some (explicit) $\mathcal V(\cdot)$, some slowly varying function $\varsigma(\cdot)$ and some $\theta\in (0,1)$.
This extends results in the literature, which mostly focused on the case of a self-similar process $(X_t)_{t\geq 0}$ (such as Brownian motion or skew-Bessel process) with a homogeneous functional~$f$ (namely a pure power, possibly asymmetric). 
In a nutshell, we are able to deal with processes which are only asymptotically self-similar and functionals which are only asymptotically homogeneous.
Our results rely on an excursion decomposition of $(X_t)_{t\geq 0}$, together with a Wiener--Hopf decomposition of an auxiliary (bivariate) L\'evy process, with a probabilistic point of view. This provides an interpretation for the asymptotic behavior of the persistence probabilities, and in particular for the exponent $\theta$, which we write as $\theta = \rho \beta$, with $\beta$ the scaling exponent of the local time of $(X_{t})_{t\geq 0}$ at level $0$ and $\rho$ the (asymptotic) positivity parameter of the auxiliary L\'evy process.\\[3pt]
\textit{MSC2020 AMS classification:} 60G51; 60J55; 60J25.
\end{abstract}

%
%

{\small
\setcounter{tocdepth}{2}
\tableofcontents
}

\section{Introduction}

The study of persistence (or survival) probabilities for stochastic processes is a widely investigated problem and an extensive literature exists on this matter.
It consists in estimating the probability for a given stochastic process to remain below some level (or more generally below some barrier), at least in some asymptotic regime.
For instance, for a general random walk $(S_n)_{n\geq 0}$ or a Lévy process $(Z_t)_{t\geq0}$ in $\mathbb{R}$, fluctuation's theory and the Wiener--Hopf factorization gives the following characterization, see e.g.~\cite{Doney07}.
Letting $T_z$, for $z\geq0$, denote the first hitting of $(S_n)_{n\geq 0}$ (resp.\ of~$(Z_t)_{t\geq0}$) above level $z$, then $t\mapsto\mathbb{P}(T_z > t)$ is regularly varying with index $-\rho$, $\rho\in(0,1)$,  if and only if $\lim_{n\to\infty} \frac1n \sum_{k=1}^n\mathbb{P}(S_k >0) =\rho$ (resp.\ if and only if $\lim_{t\to\infty}\frac1t \int_0^t\mathbb{P}(Z_s >0)\dr s =\rho$). The latter condition is known as \textit{Spitzer's condition}.

\subsection{Persistence problems for additive functionals}
In this paper, we consider a one-dimensional c\`adl\`ag strong Markov process $(X_t)_{t\geq0}$ with values in $\mathbb R$, and we assume that $0$ is recurrent for $(X_t)_{t\geq 0}$. We denote by $\bbP_x$ the law of the process starting from $X_0=x$; we also write $\bbP =\bbP_0$ for simplicity.
For a measurable function $f:\mathbb{R}\to\mathbb{R}$, we consider the following additive functional of $(X_t)_{t\geq 0}$:
\begin{equation}
\label{def:zeta}
\zeta_t = \int_0^t f(X_s) \dr s \,.
\end{equation}
We are now interested in the asymptotic behavior of the persistence (or survival) probabilities for the (non-Markovian) process~$(\zeta_t)_{t\geq 0}$, \textit{i.e.}\ probabilities that the process~$\zeta$ avoids a barrier during a long period of time.

\smallskip
For $z>0$, we denote by $T_z= \inf\{t>0,\:\zeta_t \geq z\}$ the first hitting time of $(\zeta_t)_{t\geq0}$ at the level~$z$. 
We aim at describing the asymptotic behavior of the probability $\mathbb{P}_x(T_z > t)$.
More precisely we show that, under some natural conditions on $f$ and $(X_t)_{t\geq0}$ (presented below in Section~\ref{sec:resI}), there exist a persistence exponent $\theta \in (0,1)$, a slowly varying function $\varsigma(\cdot)$ and a constant $\mathcal V(z)$ such that
\begin{equation}
\label{eq:asympPT}
 \mathbb{P}(T_z > t) \sim \mathcal V(z) \varsigma(t) \, t^{-\theta} \qquad \text{ as } t\to\infty \,.
\end{equation}

\subsection{Overview of the literature and main contribution}

The study of persistence probabilities for additive functionals of random walks or Lévy processes processes has a long history; we refer to the survey~\cite{ASsurvey} for a thourough review. Let us here present some quick (and updated) overview of the literature and outline what are the main novelties of our paper.

\paragraph*{The discrete case: integrated random walks.}
A question that has attracted a lot of attention since the seminal work of Sinai~\cite{Sinai92} is that of persistence probabilities for the integrated random walk. Let $(U_i)_{i\geq 1}$ be i.i.d.\ random variables and let $X_n:=\sum_{i=1}^n U_i$.
Then setting $\zeta_n = \sum_{k=1}^n X_k$, Sinai~\cite{Sinai92} proved that if $(X_n)_{n\geq 0}$ is the simple random walk (\textit{i.e.}\ $U_i$ is uniform in $\{-1,1\}$), then 
\[
\bbP(\zeta_k \geq 0 \text{ for all }  0\leq k \leq n) \asymp n^{-1/4} \,,
\]
where $u_n \asymp v_n$ means that there are two constants $c,c'>0$ such that $c'v_n\leq u_n \leq c v_n$.

In the case where $\bbE[X_i]=0$ and $\bbE[X_i^2]<+\infty$, this result has then been extended to include the case of more general random walks.
Vysotsky~\cite{Vysotsky10} gave the same $\asymp n^{-1/4}$ asymptotic for double-sided exponential and double-sided geometric walks (not necessarily symmetric).
Dembo, Ding and Gao~\cite{DDG13} gave a general proof for the $\asymp n^{-1/4}$ asymptotic.
Finally, Denisov and Wachtel~\cite{DW15} proved the  precise asymptotic behavior, \textit{i.e.}\ there is some constant $c_0$ such that
\[
\bbP(\zeta_k \geq 0 \text{ for all }  0\leq k \leq n) \sim c_0 n^{-1/4} \qquad \text{as }n\to\infty \,.
\]

The case where $\bbE[X_i]=0$ and $(X_i)_{i\geq 1}$ is in the domain of attraction of an $\alpha$-stable random variable with $\alpha\in (1,2)$ is still mostly open.
An example of a \textit{one-sided} random variable $X_i$ with pure power tail is given in~\cite{DDG13}, for which one has $\bbP(\zeta_k \geq 0 \text{ for all }  0\leq k \leq n) \asymp n^{-\theta}$ with $\theta =\frac{\alpha-1}{2\alpha}$.
Let us also mention~\cite{Vysotsky14}, which gives the sharp asymptotic $\sim c_0 n^{-\theta}$ in the case of a \textit{one-sided} random variables (more precisely, right-exponential or skip free), with the same exponent~$\theta$.

Let us stress that in the discrete case, the main focus in the literature has been so far on intergrated random walks rather than some more general additive functionals of a more general Markov process. Nevertheless, let us mention a series of article by I. Grama and R. Lauvergnat and E. Le Page ~\cite{GLL18AOP,GLL20PTRF} where the authors study additive functional of discrete-time Markov chain under a spectral-gap assumption. They show that the probability of persistence behave like $n^{-1/2}$ as the simple random walk and they continue the analysis further by proving local limit theorem for the additive functional conditioned to be positive.

We do not pursue further in this article the case of an additive functional $\zeta_n=\sum_{k=1}^n f(X_n)$ of a discrete-time Markov chain $(X_n)_{n\geq0}$, but we mention~\cite{BB23} which uses a generalization of Sparre Andersen's formula to treat general cases where $X$ is symmetric (and skip free) and $f$ is also symmetric; we also refer to Section~\ref{sec:examples} where continuous-time Bessel random walks are considered.
We believe that the methods of the present paper are general and could be adapted to treat some more general (\textit{i.e.}\ non-symmetric) cases.

\paragraph*{Integrated Brownian motion and $\alpha$-stable Lévy processes.}
Analogously to the discrete case, the question that first attracted a lot of attention was that of persistence probabilities for the integrated Brownian motion. Consider $(X_t)_{t\geq 0}$ a standard Brownian motion and let $\zeta_t := \int_0^t X_s \dd s$ and $T_z:= \inf\{ t, \zeta_t \geq z\}$.
It is then proven in~\cite{Goldman71,IW94} that one also has 
\[
\bbP(T_z >t) \sim c_0 z^{1/6} t^{-1/4} \qquad \text{ as } \quad  \frac{t}{z^{2/3}}\to+\infty \,,
\]
with an explicit constant $c_0=\frac{3^{4/3} \Gamma(2/3)}{\pi 2^{13/12} \Gamma(3/4)}$.
Let us also stress that there exists an explicit formula for the density of $(T_z, B_{T_z})$, given by Lachal~\cite{lachal91}. 

The case where $(X_t)_{t\geq 0}$ is a strictly $\alpha$-stable Lévy process has been treated more recently.
First, Simon~\cite{Simon07} proved that if $\alpha\in (1,2)$ and $(X_t)_{t\geq 0}$ is spectrally positive, then, as $t\to\infty$ we have
$\bbP(T_z >t) = t^{-\theta +o(1)}$ with $\theta = \frac{\alpha-1}{2\alpha}$.
(Note that this is the same exponent as in the case of random walks mentioned above.)
In the case $\alpha\in (0,2]$ and if $(X_t)_{t\geq 0}$ has a positivity parameter $\varrho := \bbP(X_t >0)$, Profeta and Simon~\cite{ps15} proved that, as $t\to\infty$,
\[
\bbP(T_z >t) = t^{-\theta +o(1)} \qquad \text{ with } \quad \theta = \frac{\varrho}{1+\alpha (1-\varrho)}\,.
\]
These results do not include the case of general Lévy processes $(X_t)_{t\geq 0}$.
A natural conjecture, which is still open, is that the above asymptotic behavior for the persistence probabilities remains valid if $(X_t)_{t\geq 0}$ is in the domain of attraction of an $\alpha$-stable Lévy process with positivity parameter $\varrho$.

\paragraph*{Homogeneous additive functional of (skew-)Bessel processes.}
In the above, we only accounted for the literature concerning integrated processes, \textit{i.e.}\ additive functionals $\zeta_t = \int_0^t f(X_s) \dd s$ with the identity function $f(x)=x$.
The case of a more general function $f$ has also been considered, starting with the work of Isozaki~\cite{Isozaki96}.

First, in the case where $(X_t)_{t\geq 0}$ is a Brownian motion,
Isozaki considered the function $f$ is homogeneous (and symmetric), given by $f(x)=\sign(x)|x|^{\gamma}$ for some $\gamma \geq 0$ and proved that
$\bbP(T_z >t) \asymp t^{-1/4}$ as $t\to+\infty$. This work was inspired by the one of Sinaï~\cite{Sinai92} and exploits the underlying idea that the fluctuations of $\zeta_t$  are related to the one of $Z_t:=\zeta_{\tau_t}$ where $\tau_t$ is the inverse local time time at $0$ of the Brownian motion. Observing  that $Z_t$ is a Lévy process for which fluctuation theory is well known, Kotani managed to solve the persistence problem for $\zeta_t$ by establishing a Wiener--Hopf factorization for the bi-dimensional Lévy process $(\tau_t, Z_t)_{t\geq0}$.
Later, Isozaki and Kotani~\cite{IK00} considered the case where $f$ is homogeneous but possibly asymmetric, \textit{i.e.}\ $f(x)=|x|^{\gamma} (c_+\ind_{\{x>0\}} - c_- \ind_{\{x<0\}})$ for some $c_+,c_->0$ and $\gamma>-1$:
they prove the precise asymptotic estimate
$\bbP(T_z >t) \sim C_z t^{-\rho/2}$ as $t\to+\infty$,
 where $\rho$ is some asymmetry parameter that depends (explicitly) on $\gamma$ and the ratio $c_+/c_-$. Note that the tools used in \cite{IK00} are somewhat analytical and do not rely on the Wiener--Hopf factorization established by Isozaki in the previous article \cite{Isozaki96}.

 Let us also mention the work of McGill \cite{mcgill} who considers the case of a Brownian motion, and a generalized function $f$ taking $f(x)\dr x = \mathbf{1}_{\{x\geq 0\}}m_{+}(\dr x) -  \mathbf{1}_{\{x <0\}}m_{-}( \dr x)$, with $m_{+}$ and $m_{-}$ being Radon measures respectively  on $\bbR_+$ and $\bbR_{-}$. The question raised in \cite{mcgill} is to give the asymptotic behavior of $\bbP(T_z >t)$ as $z\downarrow 0$, when $t$ is fixed, which is slightly different from our persistence problem. Using excursion theory of the Brownian motion, McGill proves, under some technical conditions on $m_{\pm}$, that $\bbP(T_z >t)\sim C_t \mathcal{V}(z)$ as $z\to0$ where $\mathcal{V}$ is the renewal function  of the ladder height process associated to $(Z_t)_{t\geq 0}$ (see~\eqref{eq_renewal} for a definition), generalizing results in \cite{IK00}.

More recently, Profeta~\cite{Prof21} treated the case where $(X_t)_{t\geq 0}$ is a \textit{skew-Bessel process} with dimension $\delta \in [1,2)$ and skewness parameter $\eta \in (-1,1)$, \textit{i.e.}\ roughly speaking a Bessel process of dimension $\delta \in [1,2)$ which has some asymmetry $\eta$ when it touches $0$ (see Example~\ref{ex:skewbessel} for a proper definition).
Profeta also considers the case of a homogeneous but possibly asymmetric function $f$, namely $f(x)=|x|^{\gamma} (c_+\ind_{\{x>0\}} - c_- \ind_{\{x<0\}})$ for some $c_+,c_->0$ but his work is restricted to the case $\gamma>0$.
He proves that 
$\bbP(T_z >t) \sim C_z t^{-\theta}$ as $t\to\infty$,
with some explicit expression for the constant $C_z$ and the exponent $\theta$ (see section~\ref{sec:examples} below).
Actually,~\cite{Prof21} goes further and provides some explicit expression for the law of different quantities related to this problem.

Let us also mention the work of Simon~\cite{Simon07}, which deals with the additive functional of a strictly $\alpha$-stable Lévy process $(X_t)_{t\geq 0}$ with $\alpha \in (1,2]$ and a (symmetric) homogeneous functional $f(x)=\sign(x)|x|^{\gamma}$ for some $\gamma > - \frac12 (1+\alpha)$.
Letting $\theta = \frac{\alpha-1}{2\alpha}$, the first result of~\cite{Simon07} is that one always have $\bbP(T_z >t) \leq  C t^{-\theta}$;
the second result is that if $(X_t)_{t\geq 0}$ is spectrally positive, then $\bbP(T_z >t) =   t^{-\theta+o(1)}$.

\paragraph*{Our main contribution.}
Before we briefly describe our contribution, let us make a few comments on the above-mentioned results.

First of all, all the results on persistence probabilities for additive functionals of processes are limited to:
(i) self-similar Markov processes, \textit{i.e.}\ Brownian motion, (skew-)Bessel processes and strictly stable Lévy processes;
(ii) functions $f$ that are homogeneous, \textit{i.e.}\ also enjoy some scaling property.
These two points are important in the proofs, since it immediately entails some scaling property for the additive functional $(\zeta_t)_{t\geq 0}$; which is then easily seen to be self-similar.

Second, the method of proof of Isozaki~\cite{Isozaki96} relies on an excursion decomposition of the process $(X_t)_{t\geq 0}$, together with a Wiener--Hopf decomposition for the auxiliary process $(\tau_t,Z_t)_{t\geq 0}$.
Further works, in particular~\cite{IK00,mcgill}, did not completely rely on this decomposition to obtain the sharp asymptotic behavior (and in particular the constant $C_z$), but rather on a more analytical approach.
For instance, Profeta's approach in~\cite{Prof21} uses exact calculations to derive the densities of various quantities of interest; the exact formulas available when dealing with Bessel processes then becomes crucial.

Let us keep the following example in mind (from~\cite{Prof21}), that we will use as a common thread:

\begin{example}~
\label{ex:main}

\noindent
(i)  $(X_t)_{t\geq 0}$ is a skew-Bessel process of dimension $\delta \in (0,2)$ and skewness parameter $\eta \in(-1,1)$.

\noindent
(ii) $f$ is homogeneous $f(x)=|x|^{\gamma} (c_+\ind_{\{x>0\}} - c_- \ind_{\{x<0\}})$ for some $\gamma \in \mathbb R$ and $c_+,c_->0$.
\end{example}

\noindent
With that said, here are the main contribution of our paper
\begin{itemize}
\item We treat the case of a general Markov process $(X_t)_{t\geq 0}$ (with some minimal assumption); in particular, we only need asymptotic properties on $(X_t)_{t\geq 0}$. Note that we also treat the case where $(X_t)_{t\geq 0}$ positive recurrent (e.g.\ an Ornstein-Uhlenbeck process), which seems to have been left outside of the literature so far.
\item We treat the case of a general function $f$; in particular, we only need asymptotic properties on $f$. In the case of Example~\ref{ex:main}, we also extend the range of parameters where~\eqref{eq:asympPT} holds, compared to~\cite{IK00,Prof21}, see Section~\ref{sec:examples}.
\end{itemize}
Another contribution of our paper is that it somehow unifies all the above results, by using a probabilistic approach. We push the excursion decomposition employed in~\cite{Isozaki96,IK00} (our main assumption on $(X_t)_{t\geq 0}$ ensures that it possess an excursion decomposition) and we prove on a slightly more general Wiener--Hopf factorization for the bivariate process $(\tau_t,Z_t)_{t\geq 0}$.
In particular, this provides a probabilistic interpretation of the exponent $\theta$ in~\eqref{eq:asympPT}, which we decompose into two parts: $\theta=\beta \rho$, with
(i)~$\beta \in (0,1]$ which descibes the scaling exponent of the local time of $(X_t)_{t\geq 0}$ at level $0$ (we have $\beta=1-\delta/2$ in Example~\ref{ex:main});
(ii)~$\rho \in (0,1)$ which is an asymmetry parameter, namely the asymptotic positivity parameter of $(Z_t)_{t\geq 0}$ (explicit in Example~\ref{ex:main}, see~\eqref{eq:rhoskewbessel}).
Our approach also provides a natural interpretation of the constant $\mathcal{V}(z)$ in~\eqref{eq:asympPT}.

\section{Main results}

\subsection{Main assumptions and notation of the article}

Throughout this paper, we will consider a strong Markov process $(\Omega, \mathcal{F}, (\mathcal{F}_t)_{t\geq0}, (X_t)_{t\geq0}, (\mathbb{P}_x)_{x\in \mathbb{R}})$ with c\`adl\`ag paths and valued in $\mathrm{E} \subset \bbR$. We assume that the filtration is right-continuous and complete and that the process is conservative, \textit{i.e.}\ it has an infinite lifetime $\mathbb{P}_x$-a.s.\ for every $x\in\mathbb{R}$. We also assume that $0\in\mathrm{E}$ and we define $(\zeta_t)_{t\geq 0}$ as in~\eqref{def:zeta}, with a function $f$ that verifies the following assumption.

\begin{assumption}\label{assump-f}
The function $f$ is measurable and locally bounded, except possibly around~$0$. It is such that a.s.\ $|\zeta_t| < \infty$ for any $t\geq0$. Moreover $f$ preserves the sign of $x$, in the sense that $f(x)\geq 0$ if $x> 0$ and $f(x) \leq 0$ if $x< 0$. Finally, we assume that $f(0)=0$.
\end{assumption}

As far as the Markov process $(X_t)_{t\geq 0}$ is concerned, we assume that $0$ is regular for itself, that is $\bbP_0(\eta_0 =0)=1$, where $\eta_0 = \inf\{t>0, X_t=0\}$.
We also assume that $0$ is recurrent for $(X_t)_{t\geq 0}$.
We make the following important assumption on $(X_t)_{t\geq 0}$ (the most restricting one).

\setcounter{countD}{\value{assumption}}
\begin{assumption}
\label{assump_exc}
Under $\bbP_0$, the process $(X_t)_{t\geq 0}$ is a.s.\ not of constant sign. Additionally, it cannot change sign without touching $0$.
\end{assumption}

These assumptions allow us to introduce the local time $(L_t)_{t\geq 0}$ of the process $(X_t)_{t\geq 0}$ at level~$0$. 
Its right-continuous inverse $(\tau_t)_{t\geq 0}$ is a subordinator and we denote by $\Phi$ its Laplace exponent:
\begin{equation}
\label{def:Phi}
t\Phi(q) = - \log \mathbb{E}[\e^{-q\tau_t}] 
\quad
\text{ for any } q\geq 0\,.
\end{equation}
We then introduce the process $(Z_t)_{t\geq0} = (\zeta_{\tau_t})_{t\geq0}$ which we will refer to as the L\'evy process associated to the additive functional $(\zeta_t)_{t\geq0}$. Indeed, $(Z_t)_{t\geq0}$ is a pure jump L\'evy process with finite variations and should be understood this way:
\begin{equation}
\label{def:Z}
Z_t := \zeta_{\tau_t} = \int_0^{\tau_t}f(X_s) \dr s = \sum_{s\leq t} \int_{\tau_{s-}}^{\tau_s} f(X_u) \dr u\, .
\end{equation}
We also introduce $g_t$, the last zero before $t$, and $I_t$, the contribution of the last (unfinished) excursion:
\begin{equation}
\label{def:xig}
g_t := \sup\{s < t, \: X_s = 0\}\quad \text{and} \quad I_t=\zeta_t - \zeta_{g_t}=\int_{g_t}^t f(X_r) \dr r \,.
\end{equation}
To state our theorems, we need to introduce the renewal function $\mathcal{V}(\cdot)$ of the usual \textit{ladder height process} $(H_t)_{t\geq0}$ associated with $(Z_t)_{t\geq0}$, see Section~\ref{section_WH} for a proper definition of $(H_t)_{t\geq0}$: 
\begin{equation}\label{eq_renewal}
 \mathcal{V}(z) = \int_0^{\infty}\mathbb{P}(H_t \leq z) \dr t\,,\qquad z\in \mathbb R_+ \,.
\end{equation}

\subsection{Main results I: persistence probabilities}
\label{sec:resI}

There are two main assumptions under which we are able to obtain the exact asymptotic behavior for the persistence probability.
The first one corresponds to the case where $0$ is  positive recurrent for $(X_t)_{t\geq 0}$, in which case the last part of the integral, \textit{i.e.}\ the term $I_t$ (see~\eqref{def:xig}), becomes irrelevant.
The second one is a bit more involved since the last term $I_t$ plays a role; the assumption is discussed in more detail in Section~\ref{sec:resII} below.

\begin{assumption}\label{assump_pos}
 The point $0$ is positive recurrent for $(X_t)_{t\geq0}$.
\end{assumption}

Under this hypothesis, we have a necessary and sufficient condition so that \eqref{eq:asympPT} holds. This condition is the analog of \textit{Spitzer's condition} for Lévy processes or random walks.
\begin{theorem}\label{main_thm_rec_pos}
 Suppose that Assumption \ref{assump_pos} holds and let $\rho \in(0,1)$. The two following assertions are equivalent:
 \begin{enumerate}[label=(\roman*)]
  \item $\lim_{t\to\infty}\frac{1}{t}\int_0^{t}\mathbb{P}(\zeta_{s} \geq 0) \dr s = \rho$
  \item For any $z> 0$, the map $t\mapsto \mathbb{P}(T_z > t)$ is regularly varying at $\infty$ with index $-\rho$.
 \end{enumerate}
Moreover, if (i) or (ii) holds for some $\rho\in(0,1)$, then there exists a slowly varying function $\varsigma(\cdot)$ such that for any $z > 0$
 \[
  \mathbb{P}(T_z > t) \sim \mathcal{V}(z)\varsigma(t)t^{-\rho} \quad \text{as }t\to\infty.
 \]
\end{theorem}

For instance, Theorem~\ref{main_thm_rec_pos} applies to an Ornstein-Uhlenbeck process $(X_t)_{t\geq 0}$, see Section~\ref{sec:examples} below.
Let us stress that the slowly varying function $\varsigma(\cdot)$ and the so-called \emph{renewal function} $\mathcal V(\cdot)$ can be described explicitly in some cases, see Section~\ref{sec:examples}.
In particular, if the process $(Z_t)_{t\geq 0}$ is symmetric, for instance if $(X_t)_{t\geq 0}$ is symmetric and $f$ is odd, then $\rho=\frac12$ and the slowly varying function $\varsigma(\cdot)$ is constant (the expression in~\eqref{def:varsigma} is equal to $1$).

\begin{remark}
 It is shown below (see Theorem \ref{equivalence_theorem}) that for any $\rho\in(0,1)$, condition \textit{(i)} above is equivalent to $\frac{1}{t}\int_0^t \mathbb{P}(Z_s\geq0)\dr s \to \rho$ as $t\to \infty$. Therefore, it is also equivalent to the stronger condition $\mathbb{P}(Z_t \geq0) \to \rho$ as $t\to \infty$, see for instance~\cite{bertoin1997spitzer}. It is not clear whether or not these conditions are equivalent to $\mathbb{P}(\zeta_t \geq0) \to\rho$ as $t\to \infty$.
Condition \textit{(i)} of Theorem \ref{main_thm_rec_pos} is satisfied if, for instance, there is some central limit theorem for $(\zeta_t)_{t\geq0}$; we refer to \cite[Thm 5 and 7]{bethencourt2021stable} for such instances, in the case of one-dimensional diffusions. 
\end{remark}

We now turn to the case where $0$ is null recurrent: our assumption is the following.

\begin{assumption}\label{assump_null}
The point $0$ is null recurrent for $(X_t)_{t\geq0}$.
Moreover, there exist $\alpha\in(0,2]$, $\beta\in(0,1)$, some functions $a(\cdot)$ and $b(\cdot)$ that are regularly varying around $0$ with respective indices $1/\alpha$ and $1/\beta$ such that the following convergence in distribution holds (for the Skorokhod topology):
\[
\left(\tau_{t}^h, Z_{t}^h\right)_{t\geq 0} := \left(b(h)\tau_{t/h}, a(h)Z_{t/h}\right)_{t\geq 0} \xrightarrow[]{(d)} \left(\tau_{t}^0, Z_{t}^0\right)_{t\geq 0} \quad \text{ as } h\to0,
\]
where $(\tau_t^0, Z_t^0)_{t\geq0}$ is a Lévy process.
We additionally assume that $a(b^{-1}(q))I_e$ converges in distribution as $q\to0$, where $e = e(q)$ is an independent exponential random variable of parameter $q > 0$ and $b^{-1}$ is an asymptotic inverse of $b$.
\end{assumption}

\noindent
Let us stress that Assumption~\ref{assump_null} is about the auxiliary process $(\tau_t,Z_t)_{t\geq 0}$ (and not directly about the process $(X_t)_{t\geq 0}$ and the function $f$), which may be difficult to verify.
We present below in Section~\ref{sec:resII} some conditions on $(X_t)_{t\geq 0}$ and $f$ for Assumption~\ref{assump_null} to hold; these conditions are easier to verify in practice.

\begin{remark}
The limiting process $(\tau_t^0, Z_t^0)_{t\geq0}$ necessarily satisfies the following scaling property: 
\[
(\tau^0_t, Z^0_{t})_{t\geq0} \overset{d}{=} (c^{-1/\beta}\tau^0_{ct}, c^{-1/\alpha}Z^0_{ct})_{t\geq0} \qquad \text{ for any } c>0 \,.
\]
Also, the fact that $(b(h)\tau_{t/h})_{t\geq 0}$ converges in distribution to a $\beta$-stable subordinator $(\tau_t^0)_{t\geq 0}$ is actually equivalent to the fact that the Laplace exponent $\Phi(q)$ of $(\tau_t)_{t\geq 0}$ is regularly varying with exponent $\beta \in (0,1)$ as $q\downarrow 0$.
In that case, $b(\cdot)$ is an asymptotic inverse of $\Phi(\cdot)$ (up to a constant factor), see Section~\ref{sec:proof_null}.
\end{remark}

\begin{theorem}\label{main_theorem}
Suppose that Assumption \ref{assump_null} holds and assume that $(Z^0_{t})_{t\geq0}$ has positivity parameter $\rho = \mathbb{P}(Z_t^0 \geq0) \in(0,1)$. Then there exists a slowly varying function $\varsigma(\cdot)$ such that, for any $z > 0$,
 \[
  \mathbb{P}(T_z > t) \sim \mathcal{V}(z)\varsigma(t)t^{-\beta\rho} \quad \text{as }t\to\infty,
 \]
\noindent
where $\beta\in (0,1)$ is given by Assumption~\ref{assump_null}.
\end{theorem}

Let us observe that Example~\ref{ex:main}, where $(X_t)_{t\geq0}$ is a skew-Bessel process and $f$ is homogeneous, verify our Assumption~\ref{assump_null}. Indeed, if $\gamma>-\delta$, the bivariate Lévy process $(\tau_t,Z_t)_{t\geq 0}$ directly enjoys a $(\beta,\alpha)$-scaling property (and similarly for $I_t$)  with $\beta=1-\delta/2$ and $\alpha=(2-\delta)/(2+\gamma)\in (0,1)$, so it trivially satisfies Assumption~\ref{assump_null}.
Thus, Theorem~\ref{main_theorem} applies and the parameter $\rho$ is also explicit; we refer to Section~\ref{sec:examples} for further details.
The advantage of our result is that we are able to treat a general class of Markov processes $(X_t)_{t\geq 0}$ (for instance that are ``asymptotically skew-Bessel processes'') and of function $f$ (that are ``asymptotically'' homogeneous).

\subsection{Main results II: application to one-dimensional generalized diffusions}
\label{sec:resII}

In this section we apply our result to a large class of one-dimensional Markov processes. We recall the It\^o-McKean \cite{im63, im96} construction of generalized one-dimensional diffusions, based on a Brownian motion changed of scale and time. Our main goal is to provide conditions on the function $f$, the \emph{scale function} $\sca$ and \emph{speed measure} $\mm$ that ensure that Assumption~\ref{assump_null} holds. 

\smallskip
Let $\mm: \bbR \to \bbR$ be a non-decreasing right-continuous function sucht that $\mm(0) = 0$, and $\sca: \bbR \to \bbR$ a continuous increasing function. 
We assume that $\sca(\bbR) = \bbR$, $\sca(0)=0$ and abusively, we also denote by $\mm$ the Radon measure associated to $\mm$, that is $\mm((a,b])= \mm(b)-\mm(a)$ for all $a<b$.
We introduce $\mm^\sca$ the image of $\mm$ by $\sca$, \textit{i.e.}\ the Stieltjes measure associated to the non-decreasing function $\mm\circ\sca^{-1}$, where $\sca^{-1}$ is the inverse function of $\sca$.
Then, we define $A_t^{\mm_{\sca}}$ the continuous additive functional of a Brownian motion $(B_t)_{t\geq 0}$ given by 
\[
A_t^{\mm^\sca} = \int_{\bbR} L^x_t \mm^\sca(\dr x),
\]
where $(L_t^x)_{t\geq0, x\in \bbR}$ denotes the usual family of local times of the Brownian motion, assumed to be continuous in the variables $x$ and $t$. We let $\rho_t$ the right-continuous inverse of $A_t^{\mm^\sca}$, and we set 
\[
X_t = \sca^{-1}( B_{\rho_t} ).
\]
Then it holds that $(X_t)_{t\geq0}$ is a strong Markov process valued in $\mathrm{supp}(\mm)$, where $\mathrm{supp}(\mm)$ is the support of the measure $\mm$. We refer to Section \ref{section_ito_mc_kean} for more details. We will therefore assume that $0\in\mathrm{supp}(\mm)$ and in this framework, $0$ is a recurrent point for $(X_t)_{t\geq0}$. When $\mathrm{supp}(\mm)$ is some interval $J$, $(X_t)_{t\geq0}$ is a one-dimensional diffusion living in $J$ and its generator is formally given by 
\[
\mathcal{A} = \frac{\dr}{\dr \mm}\frac{\dr}{\dr \sca}.
\]

\begin{example}
\label{ex:skewbessel}
A skew-Bessel process of dimension $\delta\in(0,2)$ and skewness parameter $\eta\in(-1,1)$, is the linear diffusion on $\bbR$ whose scale function $\sca$ and speed measure $\mm$ are defined as
\[
 \sca(x) = \mathrm{sgn}(x)\frac{1 - \mathrm{sgn}(x)\eta}{2-\delta}|x|^{2-\delta} \quad \text{and}\quad \mm(\dr x) = \frac{1}{1-\mathrm{sgn}(x)\eta}\bm{1}_{\{x\neq0\}}|x|^{\delta - 1}\dr x.
\]
Informally, this process can be constructed by concatenating independent excursions of the (usual) Bessel process, flipped to the negative half-line with probability $(1-\eta) / 2$.
\end{example}

\begin{example}
\label{ex:birthdeath}
When $\mm^\sca$ is a sum of Dirac masses, then $(X_t)_{t\geq0}$ is a birth and death process, see \cite{s63}. For instance, if $\sca =\mathrm{id}$ and $\mm = \sum_{n\in\mathbb{N}}\delta_n$, then $(X_t)_{t\geq0}$ is a continuous-time simple random walk on~$\mathbb{Z}$.
\end{example}

\noindent
For a function $f$ as in Assumption~\ref{assump-f}, we also set $\dr \mm^f := f\circ \sca^{-1} \dr \mm^\sca$.
Note that $\mm^f$ is a signed measure (recall that $f$ preserves the sign).
We suppose in addition that $f\circ \sca^{-1}$ is locally integrable with respect to $\mm^\sca$ so that $\mm^f$ is also a Radon measure. We will also denote by $\mm^f$ the associated function, \textit{i.e.}\ $\mm^f(x) = \int_0^xf\circ \sca^{-1}(u) \mm^\sca(\dr u)$, which is non-decreasing on~$\bbR_+$ and non-increasing on~$\bbR_-$.

\smallskip
We now give practical conditions on the scale function $\sca$, the speed measure $\mm$ and the function $f$ so that Assumption \ref{assump_null} holds. We will consider three different assumptions.

\begin{assumption}\label{assump_phi_beta_ito}
 There exist $\beta \in (0,1)$, a slowing variation function $\Lambda_\sca$ at $+\infty$, and two non-negative constants $ m_-, m_+$ with $m_{-} + m_+ > 0$,  such that
 \[
 \begin{cases}
\displaystyle  \mm^\sca(x) \sim m_+\Lambda_\sca(x) x^{1/\beta -1} &\quad \text{ as } x\to+\infty,\\
\displaystyle  \mm^\sca(x)\sim -m_-\Lambda_\sca(\vert x \vert ) |x|^{1/\beta -1} &\quad \text{ as } x\to-\infty.
 \end{cases}
\]
\end{assumption}

\begin{assumption}\label{assump_gaus_ito}
 The function $\sca$ is $\mathcal C^1$, and there exist a constant $\mm^f(\infty)\in (0,\infty)$ such that $\lim_{x\to\pm\infty} \mm^f(x) = \mm^f(\infty)$ and the function $\mm^f(\infty) - \mm^f$ belongs to $\mathrm{L}^2(\dr x)$.
\end{assumption}

\begin{assumption}\label{assump_m_f_alpha_ito}
 There exist $\alpha \in (0,2)$, a slowing variation function $\Lambda_f$ at $+\infty$, and two non-negative constants $f_{-}, f_{+}$ with $f_- + f_+ > 0$, such that  according to the value of $\alpha$, we have
\begin{enumerate}[label=(\roman*)]
 \item If $\alpha\in(0,1)$, then
 \begin{equation}
 \label{cond:alpha<1}
  \begin{cases}
 \displaystyle \mm^f(x) \sim f_+\Lambda_f(x) x^{1/\alpha -1} &\quad \text{ as } x\to+\infty,\\
 \displaystyle  \mm^f(x) \sim  f_-\Lambda_f(\vert x\vert ) |x|^{1/\alpha -1} &\quad \text{ as } x\to-\infty.
 \end{cases}
 \end{equation}

\item If $\alpha = 1$, then the following limit exists $\lim_{h\downarrow0}\frac{1}{\Lambda_f(1/h)}(\mm^f(1/h)-\mm^f(-1/h)) = \mathbf{c}$ and 
\begin{equation}
\label{cond:alpha=1}
\begin{cases}
 \displaystyle \lim_{h\to 0} \frac{1}{\Lambda_f(1/h)}(\mm^f(x/h)-\mm^f(1/h)) = f_+\log x, &\qquad \forall x>0,  \\ 
 \displaystyle \lim_{h\to 0}  \frac{1}{\Lambda_f(1/h)}(\mm^f(x/h)-\mm^f(-1/h))= f_-\log \vert x \vert,  &\qquad \forall x<0.
\end{cases}
\end{equation}
Note that if the limit $\mathbf{c}$ exists, this implies that $f_+=f_-$; we will assume for simplicity that $f_+=f_-=1$.

\item If $\alpha \in (1,2)$, then there is a constant $\mm^f(\infty)\in (0,\infty)$ such that $\lim_{x\to\pm\infty} \mm^f(x) = \mm^f(\infty)$ and 
\begin{equation}
\label{cond:alpha>1}
\begin{cases}
\displaystyle \mm^f(\infty) - \mm^f(x) \sim  f_+\Lambda_f(x) x^{1/\alpha -1} &\quad \text{ as } x\to+\infty,\\
\displaystyle \mm^f(\infty) - \mm^f(x) \sim  f_-\Lambda_f(\vert x \vert ) |x|^{1/\alpha -1} &\quad \text{ as } x\to-\infty.
 \end{cases}
\end{equation}
\end{enumerate}
\end{assumption}

\noindent
Note that if Assumption \ref{assump_gaus_ito} holds, then Assumption \ref{assump_m_f_alpha_ito} can not hold and conversely.
The main results of this section are the following.

\begin{proposition}[Gaussian case]
\label{prop_brow_null_rec}
Suppose that Assumptions~\ref{assump_phi_beta_ito} and~\ref{assump_gaus_ito} hold. Then, Assumption~\ref{assump_null} is verified with $\beta\in (0,1)$, $\alpha=2$, and the following choice for $a(\cdot)$, $b(\cdot)$:
\[
a(h) = h^{1/2},\qquad b(h) = h^{1/\beta} / \Lambda_\sca(1/h) \,.
\]
As a consequence, Theorem~\ref{main_theorem} holds under Assumptions~\ref{assump_phi_beta_ito} and~\ref{assump_gaus_ito}, with $\rho =\bbP(Z_t^{0} \geq 0)= 1/2$.
\end{proposition}

\begin{proposition}[$\alpha$-stable case, $\alpha\in (0,2)$]
\label{prop_assump_true}
Suppose that Assumptions~\ref{assump_phi_beta_ito} and~\ref{assump_m_f_alpha_ito} hold.
Then, Assumption~\ref{assump_null} is verified with $\alpha\in (0,2)$, $\beta\in (0,1)$ and the following choice for $a(\cdot)$, $b(\cdot)$:
\[
a(h) = h^{1/\alpha}/ \Lambda_f(1/h),\qquad b(h) = h^{1/\beta} / \Lambda_\sca(1/h) \,.
\]
As a consequence, Theorem~\ref{main_theorem} holds under Assumptions~\ref{assump_phi_beta_ito} and~\ref{assump_m_f_alpha_ito}, with the following asymmetry parameter $\rho=\bbP(Z_t^{0} \geq 0)$,
\[
 \rho = \frac{1}{2} + \frac{1}{\pi \alpha} \arctan(\vartheta) \qquad \text{where} \quad \vartheta = 
 \begin{cases}
 \frac{f_+^\alpha - f_-^\alpha}{f_+^\alpha + f_-^\alpha} \tan(\pi\alpha/2) &\quad \text{if } \alpha \neq 1 \,,\\
\mathbf{c} \  &\quad \text{if } \alpha = 1 \,.
 \end{cases}
\]
\end{proposition}

%
%
%

\subsection{Main results III: starting with a non-zero velocity}
\label{sec:resIII}

In this section, we are interested in the hitting time of zero of the additive functional $z+\zeta_t$, with some initial velocity $X_0=x$. A motivation to consider such a question is to construct the additive functional conditioned to stay negative; of course, the only reason we deal with a condition to remain negative (and not positive) is because we have treated above the asymptotics of $\bbP(T_z > t)$ for $z > 0$.

To avoid the introduction of lengthy notation, we only give an outline of our results, summarizing the content of Section \ref{sec:hitting}: the precise statement of the results are presented there.
We restrict ourselves to the case where $(X_t)_{t\geq0}$ is a regular diffusion process, valued in some open interval $J$ containing $0$. 
We consider the process $(\zeta_t, X_t)_{t\geq0}$ as a strong Markov process. For a pair $(z,x)\in\mathbb{R}\times J$, we denote by $\bm{\mathrm{P}}_{(z,x)}$ the law of $(\zeta_t, X_t)_{t\geq0}$ when started at $(z,x)$, \textit{i.e.}\ the law of $(z + \int_0^t f(X_s)\dr s, X_t)_{t\geq0}$ under $\mathbb{P}_x$. Roughly, we derive two kind of results:
\begin{enumerate}[wide,label=(\roman*)]
 \item We identify some finite function $h:\bbR\times J\to \bbR_+$ such that, under Assumption \ref{assump_pos} or \ref{assump_null}, we have for any $(z,x)\in\bbR\times J\setminus\{(0,0)\}$,
 \[
  \bm{\mathrm{P}}_{(z,x)}(T_0 > t) \sim h(z,x)\varsigma(t)\, t^{-\rho\beta} \quad \text{as }t\to\infty,
 \]
 for some slowly varying function $\varsigma$ (which does not depend on $(z,x)$) and some parameters $\beta\in (0,1]$, $\rho\in(0,1)$ (given by the assumption). We refer to Theorem~\ref{main_thm_hitting} for the precise statement.
 
 \item Secondly, we show that the function $h$ is harmonic for the killed process $(\zeta_{t\wedge T_0}, X_{t\wedge T_0})_{t\geq0}$, see Corollary \ref{coro_harmonic}. This classicaly enables us to construct the additive functional conditionned to stay negative, through Doob's $h$-transform, see Proposition~\ref{prop:conditioned}. This result generalizes the previous work \cite{GJW99} on the integrated Brownian motion conditioned to be positive and have the same flavor of some results from Grama-Lauvergnat-Le Page  \cite{GLL18AOP} in a discrete setting. It is also related to Profeta's article \cite{Prof15} where he investigated other penalizations for the integral of a Brownian motion. 
\end{enumerate}

\subsection{A series of examples}
\label{sec:examples}

In this section, we provide several examples of application to our main theorems. 
We start with examples where Assumptions~\ref{assump_pos} or~\ref{assump_null} are easy to verify; we then turn to examples where the reformulation in terms of Assumptions~\ref{assump_phi_beta_ito}, \ref{assump_gaus_ito} or~\ref{assump_m_f_alpha_ito} are useful.

\paragraph*{Ornstein-Uhlenbeck.}

Let $(X_t)_{t\geq0}$ be an Ornstein-Uhlenbeck process and $f$ be some odd function.
Then $0$ is positive recurrent for $(X_t)_{t\geq0}$ and since the Ornstein-Uhlenbeck process started at $0$ is symmetric (in the sense that the law of $(-X_t)_{t\geq0}$ equals the law of $(X_t)_{t\geq0}$) it is clear that $\mathbb{P}(\zeta_t \geq0) = 1/2$ for any $t > 0$. Therefore Theorem \ref{main_thm_rec_pos} holds with $\rho = 1/2$ and the slowly varying function~$\varsigma$ is constant (the term~\eqref{def:varsigma} is equal to $1$ since $Z_t$ is symmetric).

\paragraph*{Skew-Bessel and homogeneous functional, back to Example~\ref{ex:main}.}

Let $(X_t)_{t\geq0}$ be a skew-Bessel process of dimension $\delta\in(0,2)$ and skewness parameter $\eta\in(-1,1)$, as defined in Example~\ref{ex:skewbessel} by its scale function $\sca$ and speed measure $\mm$.
This process can be constructed by the following informal procedure: concatenate independent excursions of the (usual) Bessel process, flipped to the negative half-line with probability $(1-\eta) / 2$. Let $c_+,c_-$ be positive constants and consider the function $f$ defined as
$
 f(x) = \left( c_+ \bm{1}_{\{x > 0\}} - c_-\bm{1}_{\{x < 0\}} \right)|x|^{\gamma},
$
with $\gamma>-\delta$.
The persistence probability of $(\zeta_t)_{t\geq0}$ is studied in Profeta \cite{Prof21} (for $\delta\in[1,2)$ and $\gamma > 0$).

The condition $\gamma>-\delta$ is here to ensure that $|\zeta_t| < \infty$ a.s.\ for all $t>0$. One can verify in this case that $(\tau_t)_{t\geq0}$ is a $\beta$-stable subordinator where $\beta = 1 - \delta/2$. By the self-similarity of the skew-Bessel process, it holds that the law of $(\tau_t, X_t)_{t\geq0}$ is equal to the law of $(c^{-1/\beta}\tau_{ct}, c^{-1/2\beta}X_{c^{1/\beta}t})_{t\geq0}$ for any $c > 0$.
This entails that the law of $(\tau_t, Z_t)_{t\geq0}$ is equal to the law of $(c^{-1/\beta}\tau_{ct}, c^{-1/\alpha}Z_{ct})_{t\geq0}$ for any $c >0$, where $\alpha = (2-\delta) / (\gamma  +2)\in(0,1)$. It also entails that for any $t > 0$, the law of $t^{-\beta / \alpha}I_t$ is equal to the law of $I_1$.
These facts imply that Assumption \ref{assump_null} holds with $a(h) = h^{1/\alpha}$ and $b(h) = h^{1/\beta}$ (with an equality rather than a convergence in distribution); one can also verify Assumptions~\ref{assump_phi_beta_ito} and~\ref{assump_m_f_alpha_ito} directly with the expressions of the scale function $\sca$ and speed measure $\mm$ (which are pure powers, so the scaling properties are clear).

Therefore, Theorem \ref{main_theorem} holds with $\beta = 1 - \delta / 2$ and $\rho = \mathbb{P}(Z_t \geq 0)$. The positivity parameter~$\rho$ can be computed, see for instance Zolotarev \cite[\S2.6]{zolotarev1986one}: we have
\begin{equation}
\label{eq:rhoskewbessel}
 \rho = \frac{1}{2} + \frac{\arctan(\vartheta\tan(\pi\alpha/2))}{\pi \alpha} \qquad \text{where} \quad \vartheta = \frac{1+\eta - (1-\eta)(\frac{c_-}{c_+})^{\alpha}}{1+\eta +  (1-\eta) (\frac{c_-}{c_+})^{\alpha}}.
\end{equation}
The computation of $\vartheta$ can be done as in \cite[Lem.~11]{bethencourt2021stable}. Finally, since $(Z_t)_{t\geq0}$ is a stable process, the renewal function $\mathcal{V}$ is such that $\mathcal{V}(z) = b_+ z^{\alpha\rho}$ for some constant $b_+ > 0$. It is also clear that, by self-similarity, the slowing varying function $\varsigma$ is constant. Hence, we fully recover and extend the results of Profeta \cite{Prof21} to $\delta\in(0,1)$ and $\gamma\in(-\delta, 0]$.

\paragraph*{Kinetic Fokker-Planck.}
Let $(X_t)_{t\geq0}$ be a solution of the following stochastic differential equation
\[
 X_t = x_0 -\frac{\mu}{2}\int_0^t \frac{X_s}{1 + X_s^2} \dr s + B_t
\]
where $(B_t)_{t\geq0}$ is a Brownian motion, $\mu > -1$ and $x_0 \in\bbR$. The scaling limit of the process $(\zeta_t)_{t\geq0} = (\int_0^t X_s \dr s)_{t\geq0}$, \textit{i.e.}\ with the choice $f = \mathrm{id}$, is studied in \cite{fournier2018one, lebeau2019diffusion, nasreddine2015diffusion, cattiaux2019diffusion}. The corresponding scale function and speed measure are given by 
\[
 \sca(x) = \int_0^x(1+v^2)^{\mu / 2} \dr v \quad \text{and} \quad \mm(x) = \int_0^x(1+v^2)^{-\mu / 2}\dr v,
\]
see for instance \cite{fournier2018one}. Then one can check that:
\begin{enumerate}[label=(\roman*)]
\item If $\theta\in(-1,1)$, Assumption \ref{assump_phi_beta_ito} is satisfied with $\beta = \frac{1}{2}(\mu + 1) \in(0,1)$ and Assumption \ref{assump_m_f_alpha_ito} is satisfied with $\alpha = \frac{1}{3}(\mu + 1)  \in(0,\frac23)$.
Since $\mm$, $\sca$ and $f$ are odd functions, Theorem \ref{main_theorem} holds with $\beta = \frac12 (\mu + 1)$, $\rho = \frac12$ and a constant slowly varying function~$\varsigma(\cdot)$.

\item  When $\mu > 1$, $0$ is positive recurrent for $(X_t)_{t\geq0}$ and since $x \mapsto x / (1+x^2)$ is odd, the process $(\zeta_t, X_t)_{t\geq0}$ is symetric (when $X_0 = 0$) so that $ \mathbb{P}(\zeta_t \geq0) = 1/2$ for any $t> 0$. Therefore Theorem \ref{main_thm_rec_pos} holds with $\rho =1/2$ and a constant slowly varying function~$\varsigma(\cdot)$.
\end{enumerate}

\noindent
Note that our results would also be able to deal with $(\zeta_t)_{t\geq0} = (\int_0^t f(X_s) \dr s)_{t\geq0}$ for more general functions $f$.

\paragraph*{Non-homogeneous functionals of Bessel processes.}

The previous examples are limited to the case where $\alpha \in (0,1)$ in Assumption~\ref{assump_null} (or Assumption~\ref{assump_m_f_alpha_ito}).
Let us give here an example where one has $\alpha \in [1,2]$; we consider a simplified example for pedagogical purposes.

Consider a symmetric Bessel process $(X_t)_{t\geq 0}$ of dimension $\delta\in (0,2)$, \textit{i.e.} a diffusion with scale function $\sca(x)= \frac{\sign(x)}{2-\delta} |x|^{2-\delta}$ and speed measure $\mm(\dd x) = \ind_{\{x\neq 0\}} |x|^{\delta-1} \dd x$.
Then, one can check that Assumption~\ref{assump_phi_beta_ito} holds with $\beta =1-\delta/2$.
Now, let $f$ be some odd function such that: $\int_0^1 f(u^{1/(2-\delta)}) u^{1/\beta-1} \dd u <+\infty$, for instance if $f$ is bounded, to ensure that $f\circ \sca^{-1}$ is locally integrable with respect to $\mm^\sca$ (so that $\zeta_t<+\infty$ for all $t>0$); $f(x)\sim \sign(x) |x|^{\gamma}$ as $x\to\infty$, for some $\gamma \in \mathbb R$. 
Then, we can check that
\begin{enumerate}[label=(\roman*)]
\item if $\gamma> -(1+\delta)$ then Assumption~\ref{assump_m_f_alpha_ito} holds with $\alpha = (2-\delta)(\gamma+2) \in (0,2)$ and $f_+=f_-$ (by symmetry);
\item if $\gamma< -(1+\delta)$, then Assumption~\ref{assump_gaus_ito} holds.
\end{enumerate}
In all cases, we have the asymptotic behavior $\bbP(T_z>t) \sim c_0 \mathcal{V}(z) t^{-\beta/2}$, since $\rho=\frac12$ and $\varsigma(\cdot)$ is constant, by symmetry.
Note that our Assumption~\ref{assump_gaus_ito} does not deal with the case $\gamma = -(1+\delta)$, but the result should still hold in that case (one would need to deal with non-normal domain of attraction to the normal law, which would require further technicalities).


\paragraph*{Continuous-time birth and death chains (and Bessel-like walks).}

Let $(\tilde{X}_n)_{n\geq 0}$ be a birth and death process on $\mathbb{Z}$ with transition probabilities given by 
\[
\bbP(\tilde{X}_{n+1}=i+1\vert \tilde{X}_{n+1}=i)=p_i \in (0,1)
\quad\text{ and }\quad
\bbP(\tilde{X}_{n+1}=i-1\vert \tilde{X}_{n+1}=i)=q_i = 1-q_i,
\]
for $i\in \mathbb{Z}$.
We then define $X_t:= \tilde{X}_{N_t}$, with $(N_t)_{t\geq 0}$ an independent Poisson process of unit intensity.
Then $(X_t)_{t\geq_ 0}$ is a continuous-time birth and death chain, and can be described as a generalized diffusion associated to a scale function $\sca$ and a speed measure $\m$ as follows; we refer to Stone \cite{s63} for more details.
Let us define $\Delta_0=1$ and
\[
\Delta_i = \prod_{k=1}^i \frac{q_k}{p_k},\ \  \forall i \geq 1 \quad \mathrm{and} \quad \Delta_i = \prod_{k=i+1}^{0} \frac{p_k}{q_k}, \ \  \forall i \leq -1.
\]
The scale function $\sca:\bbR\to\bbR$ is increasing piecewise linear and such that $\sca(i)=x_i$ for $i \in \mathbb{Z}$, with $(x_i)_{i\in \mathbb Z}$ defined iteratively by $x_0=0$ and $\Delta_i=x_{i+1}-x_i$, $\forall i \in \mathbb{Z}$.
The speed measure $\mm$ is defined as
\[
\m := \sum_{i\in \mathbb{Z}} \Big ( \frac{1}{2\Delta_i} + \frac{1}{2\Delta_{i-1}} \Big ) \delta_{i}.
\]
In a companion paper~\cite{BB23}, the first two authors use an elementary approach to obtain two-sided bounds for the persistence of integrated symmetric (discrete-time) birth and death process with an odd function $f$; they apply their results to symmetric Bessel-like random walk (see \cite{K} for a recent account). 
Let us now observe that we can apply our machinery to continuous-time Bessel-like random walks and obtain sharps asymptotics for the persistence probabilities $\bbP(\zeta_s \leq 0 \text{ for all }s\leq t)$.

We define a symmetric Bessel-like random walk as a birth and death process with transition probabilities
\[
p_i:=\frac{1}{2}\Big (1- \frac{\mu+ \varepsilon_i}{2 i}  \Big ), \text{ for } i\geq 1 \,, \quad
p_i=q_{-i} \text{ for } i\leq -1 \,, \quad p_0=q_0=\frac12 \,.
\]
Here, $\mu$ is a real parameter and $\varepsilon_i$ is such that $\lim_{i\to \infty} \varepsilon_i=0$.
Then, we have that the process is recurrent if $\mu >1$ and null-recurrent if $\mu\in (-1,1)$; the case $\mu=1$ depends on~$(\gep_i)_{i\geq 0}$.

\smallskip
(i) In the case $\mu>1$, by Theorem~\ref{main_thm_rec_pos} we directly obtain the asymptotics 
\begin{equation}
\label{besselRW1}
\mathbb{P}(T_z>t )\sim c_0\mathcal{V}(z) t^{-1/2} \qquad \text{ as } t\to\infty \,.
\end{equation}
(We have $\rho=\frac12$ and $\varsigma(\cdot)$ constant thanks to the symmetry.)

\smallskip
(ii) In the case $\mu\in (-1,1)$, we use the following asymptotics: there exists a constant $C_0$ and a slowly varying function $L(i)=\exp(-\sum_{k=1}^i \frac{\varepsilon_k}{k})$ such that $\Delta_{i} \sim C_0\, i^{\mu} L(i)$ as $ i\to + \infty$ (and symmetrically for $i\to-\infty$).
From this asymptotics we obtain that $\sca(i) = x_i =\sum_{k=0}^{i-1} \Delta_k \sim C_0'\, i^{1+\mu} L(i)$ as $i\to +\infty$ and also that $\mm(x) \sim C_0'' \, x^{1-\mu} L(x)^{-1} $ as $x\to\infty$.
One can therefore show that Assumption~\ref{assump_phi_beta_ito} is satisfied with $\beta=\frac12(1+\mu) \in (0,1)$ (and $m_+=m_-$). If we consider the function $f(x)=x$,  Assumption \ref{assump_m_f_alpha_ito} is satisfied with $\alpha= \frac13 (1+\mu)$ and $\rho=\frac12$ (by symmetry). Finally, Theorem \ref{main_theorem} states that 
\begin{equation}
\label{besselRW2}
\mathbb{P}(T_z >t)\sim \mathcal{V}(z) \varsigma(t)t^{-(1+\mu)/4} \qquad \text{ as } t\to\infty \,,
\end{equation}
with $\varsigma$ some slowing variation function (depending on $(\varepsilon_i)_{i\geq 0}$).

Obviously, by a simple de-Poissonization argument, the asymptotics~\eqref{besselRW1}-\eqref{besselRW2} remain also valid for discrete-time Bessel-like random walks.
This therefore matches the results from \cite{BB23} and  additionnally gives the sharp asymptotics of the persistence probabilities; obviously one could consider a function $f(x) = \sign(x) |x|^{\gamma}$ with $\gamma \in \mathbb R$ withour affecting the conclusion (the exponent $\alpha$ of Assumption~\ref{assump_null} does not affect the persistence exponent $\theta$ in the symmetric case).

\section{Ideas of the proof and further comments}

\subsection{Ideas of the proof: path decomposition of trajectories}


Recall from \eqref{def:xig} that $I_t = \zeta_t - \zeta_{g_t}$, we introduce
\begin{equation}\label{def_xi_delta}
\xi_t = \sup_{[0,t]}\zeta_s \qquad \text{ and }\qquad  \Delta_t = I_t - (\xi_{g_t} - \zeta_{g_t})\,.
\end{equation}
We refer to Figure~\ref{fig:decomp} for an illustration of $\zeta_t, \xi_t, \xi_{g_t}, I_t$ and $\Delta_t$.
Then, to study the probability $\bbP(T_z >t) =\bbP(\xi_t < z)$, the main idea is to decompose it into two parts: for any $z>0$, we have
\begin{equation}
\label{firstdecomp}
\mathbb{P}\big(\xi_t < z\big) = \mathbb{P}\big(\xi_{g_t} < z,\Delta_t \leq 0\big) + \mathbb{P}\big(\xi_{g_t} + \Delta_t < z, 0 < \Delta_t \leq  z\big).
\end{equation}
(The second term will turn out to be negligible.)
As a first consequence of~\eqref{firstdecomp}, we see that
$\mathbb{P}(\xi_{g_t} < z,\Delta_t \leq 0)  \leq  \mathbb{P}(\xi_{g_t} < z)  \leq \mathbb{P}(\xi_{g_t} < z,\Delta_t \leq z)$, from which one easily gets that
\begin{equation*}
\label{generalbounds}
c_t\, \bbP\big(\xi_{g_t} <z \big)  \leq  \mathbb{P}\big(\xi_t < z\big)  \leq   \bbP\big(\xi_{g_t} <z \big) \,.
\end{equation*}
with $c_t = \bbP(I_t\leq 0)\in (0,1)$.
We will show that we actually have some constant $c_1 \in (0,1]$ such that
\begin{equation}
\label{eq:constantc1}
\mathbb{P}(\xi_t < z) \sim c_1 \mathbb{P}(\xi_{g_t} < z) \qquad \text{ as } t\to\infty \,.
\end{equation}
We will then control the probability $\bbP(\xi_{g_t} <z )$ by using the fact that $\xi_{\tau_t} = \sup_{[0,t]} Z_s$ (see Remark~\ref{sup_remark} below).

\begin{figure}
\begin{center}
\includegraphics[scale=1]{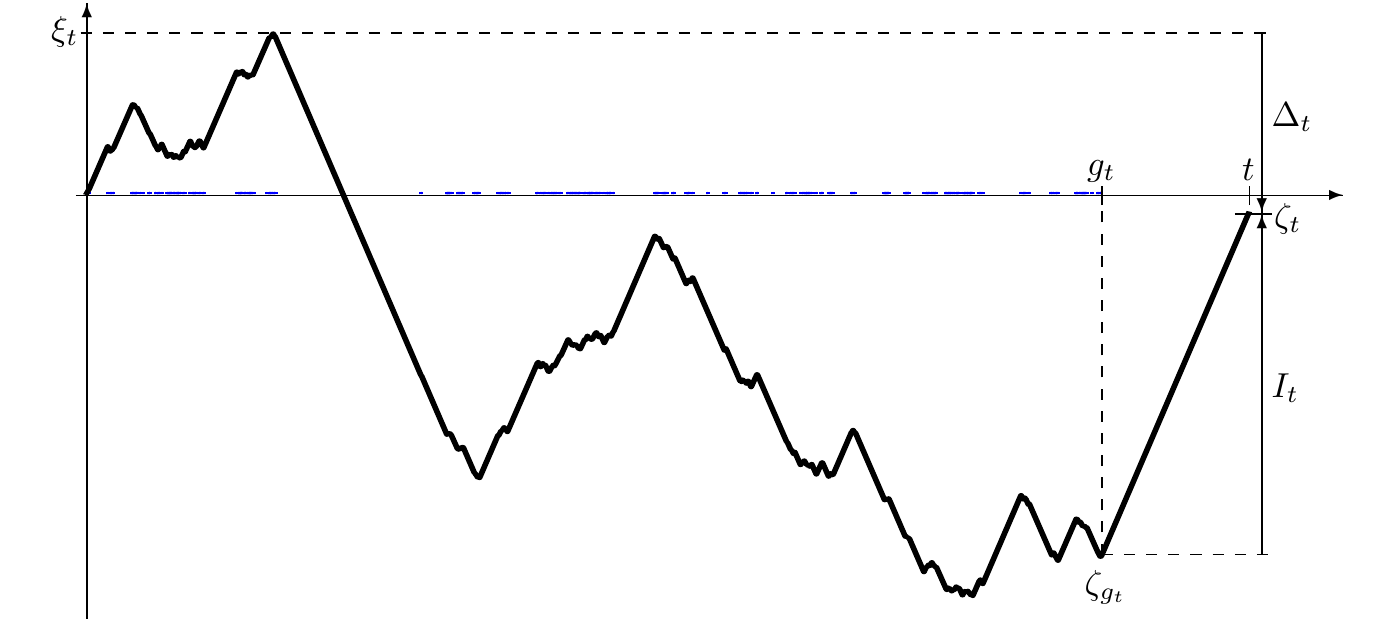}
\end{center}
\caption{\footnotesize Graphical representation of the trajectory of a realization of $(\zeta_s)_{0\leq s \leq t}$ and of its decomposition. The above is a simulation in the setting of Bessel-like random walk (see Section~\ref{sec:examples}): $(X_t)_{t\geq 0}$ is a (symmetric) Bessel-like random walk with $\mu=0.4$, \textit{i.e.}\ $\beta=0.7$, and with the function $f(x) =\sign(x)$.
The dots represent the returns to $0$ of $(X_s)_{0\leq s \leq t}$. In this particular realization, we have $I_t>0$ and $\Delta_t<0$.}
\label{fig:decomp}
\end{figure}

\medskip
\noindent
\textit{A standard trick to gain independence.\ }
To handle the quantities in~\eqref{firstdecomp}, we will use the following trick: letting $e = e(q)$ be an exponential random time $e$ with parameter $q$ independent of $(X_t)_{t\geq 0}$, we will look at the quantity $\bbP(\xi_e < z)$ instead of looking directly at $\bbP(\xi_t < z)$. 
This corresponds to taking the Laplace transform of $\bbP(\xi_t <z)$ and we loose no information by doing this: indeed, a combination of the Tauberian theom and the monotone density theorem (see \cite[Thms.~1.7.1 and~1.7.2]{bgt89}) tells us that having an asymptotic of $\bbP(\xi_e < z)$ as $q\to0$ is equivalent to having an asymptotic of $\bbP(\xi_t < z)$ as $t\to\infty$.

The first advantage of this trick is that it allows us to factorize functionals of trajectories before time $g_{e}$ and functionals of trajectories between times~$g_e$ and $e$. The precise statement is presented in Proposition~\ref{indep}, which is inspired by~\cite{svy}.
This enables us to operate a first reduction, treating $I_e$ separately from $\xi_{g_e}, \zeta_{g_e}$.
More precisely, using Proposition~\ref{indep} below, by independence, the first term in~\eqref{firstdecomp} (with $t$ replaced by an exponential random variable $e$) can be rewritten as
\begin{equation*}
\bbP(\xi_{g_e} <y, \Delta_e \leq 0 ) =\bbP(\xi_{g_e} <y , I_e \leq \xi_{g_e} -\zeta_{g_e}) = 
\mathbb E\big[ F_{I_e} (\xi_{g_e} -\zeta_{g_e}) \ind_{\{\xi_{g_e} < y\}} \big]\,,
\end{equation*}
where $F_{I_e}$ is the c.d.f.\ of $I_e$.

\medskip
\noindent
\textit{Wiener--Hopf factorization.\ }
A second key tool is a Wiener--Hopf factorization for the bivariate L\'evy Process $(\tau_t, Z_t)_{t\geq 0}$, that among other things allows us to obtain the joint distribution of~$\xi_{g_{e}}$ and $\xi_{g_e}-\zeta_{g_{e}}$,  see Corollary~\ref{laplace_transform_x_zeta} below.
In particular, it shows that $\xi_{g_e}$ and $\xi_{g_e}-\zeta_{g_e}$ are independent, so~\eqref{useindep} can further be decomposed as
\begin{equation}
\label{useindep}
\bbP(\xi_{g_e} <y, \Delta_e \leq 0 ) = \mathbb E\big[ F_{I_e} (\xi_{g_e} -\zeta_{g_e}) \big] \bbP\big( \xi_{g_e} < y \big)= \bbP(\Delta_e \leq 0) \bbP\big( \xi_{g_e} < y \big)\,.
\end{equation}
 From this, we will be able to prove that $\bbP(\xi_{e} <y)\sim c_1 \bbP(\xi_{g_e} <t)$ as $q\downarrow 0$, \textit{i.e.}~\eqref{eq:constantc1}, where the constant $c_1$ is $c_1 = \lim_{q\to0}\bbP(I_e \leq \xi_{g_e}-\zeta_{g_e})\in (0,1]$. Moreover, the Wiener--Hopf factorization will also help us obtain the asymptotic behavior of $\bbP(\xi_{g_e} <y)$ as $q\downarrow0$.

\medskip
\noindent
\textit{Conclusion.\ }
With this picture in mind, we split our results into two categories, that correspond to Assumptions~\ref{assump_pos} and~\ref{assump_null}:
\begin{itemize}[wide]
\item[(i)]   
If $X$ is positive recurrent.
Then, $\lim_{t\to \infty}\frac1t g_t = 1$ and $I_t$ will typically be much smaller than $\xi_{g_t}-\zeta_{g_t}$ as $t\to\infty$, so $\lim_{q\downarrow0}\bbP(I_e \le \xi_{g_e} -\zeta_{g_e})=1$: we will have $c_1=1$ in~\eqref{eq:constantc1}, that is $\mathbb{P}(\xi_e < y) \sim \mathbb{P}(\xi_{g_e} < y)$ as $q\downarrow 0$.
Loosely speaking, the part of the trajectory between time $g_t$ and~$t$ will have no impact on the behavior of the persistence probability.
Then, the behavior of $\bbP(\xi_{g_e}<y)$ is studied thanks to the Wiener--Hopf factorization, with the assumption that $(Z_t)_{t\geq 0}$ satisfies the so-called Spitzer's condition. 

\item[(ii)]
If $X$ is null recurrent then there are two cases, depending on whether $\alpha=2$ or $\alpha\in (0,2)$ in Assumption~\ref{assump_null} (or corresponding to Assumptions~\ref{assump_gaus_ito} and~\ref{assump_m_f_alpha_ito}).
First, if $(Z_t)_{t\geq 0}$ is in the (normal) domain of attraction of a normal law, then also in that case $I_t$ will typically be much smaller than $\xi_{g_t}-\zeta_{g_t}$ as $t\to\infty$ and we again have $\lim_{q\downarrow0}\bbP(I_e \le \xi_{g_e} -\zeta_{g_e})= 1$, that is $c_0=1$ in~\eqref{eq:constantc1}.
Second, if $(Z_t)_{t\geq 0}$ is in the domain of attraction of some $\alpha$-stable law, $\alpha\in (0,2)$, then under Assumption~\ref{assump_null} we have that $\lim_{q\downarrow0}\bbP(I_e \le \xi_{g_e} -\zeta_{g_e}) = \bbP(I\leq W)=: c_1 \in(0,1)$, where $I,W$ are independent random variables, the respective limits in law of $a(b^{-1}(q)) I_e$ and $a(b^{-1}(q))(\xi_{g_e} -\zeta_{g_e})$.
Then, the behavior of $\bbP(\xi_{g_e}<y)$ is again studied thanks to the Wiener--Hopf factorization, with the assumption that $(\tau_t)_{t\geq 0}$ is in the domain of attraction of a stable subordinator.
\end{itemize}

\subsection{Comparison with the literature}

We now discuss the novelty of our results and techniques and compare them with the existing litterature.

Let us first start with the work of Isozaki \cite{Isozaki96}, which treats the case of integrated powers of the Brownian motion.
Isozaki first uses the following Wiener--Hopf factorization of $(\tau_t, Z_t)_{t\geq0}$.
If $e=e(q)$ denotes an independent exponential random variable and $S_t = \sup_{[0,t]}Z_s$, then the product of the Laplace transforms of $(\tau_e, S_e)$ and $(\tau_e, Z_e - S_e)$ can be expressed in terms of the characteristic function of $(\tau_t, Z_t)_{t\geq0}$.
The rest of Isozaki's method is somehow more analytic and involves inversion of Fourier transforms.
Also, he exploits deeply the self-similarity of the Brownian motion. Of course, our work has been inspired by~\cite{Isozaki96}, but regarding the Wiener--Hopf factorization, we go one step further.
If we set $G_t = \sup\{s < t, Z_t = S_t\}$, then we are able to show that the law of $(G_e,\tau_{G_e}, S_e)$ and $(e-G_e,\tau_e - \tau_{G_e}, Z_e - S_e)$ are independent, infinitely divisible and can be expressed with the law of $(\tau_t, Z_t)_{t\geq0}$.
This enables us to the study the quantities of interest in the spirit of fluctuation's theory for Lévy processes.
We refer to Subsection \ref{section_WH} and Appendix~\ref{appendix_wiener} for more details.

Let us now discuss the work of McGill \cite{mcgill}, which seems to be the closest to our work.
McGill considers general additive functionals of the Brownian, \textit{i.e.}\ $\zeta_t = \int_{\bbR}L_t^x m(\dr x)$ where $(L_t^x)_{t\geq0, x\in\bbR}$ is the family of local times of the Brownian motion and $m(\dr x) = \bm{1}_{\{x \geq0\}}m_+(\dr x) - \bm{1}_{\{x < 0\}}m_-(\dr x)$ is a signed measure.
With the above notation, he caracterizes the behavior of $\bbP(\xi_e < z)$ as $z\to0$, for a fixed $q >0$; whereas we study $\bbP(\xi_e < z)$ when $q\to0$, for a fixed $z$. Let us quickly explain the content of~\cite{mcgill}.
First, if we fix $q > 0$ and set $v(z) = \bbP(\xi_e < z)$, he establishes the following identity:
\begin{equation}\label{eq_mcgill}
v(z) = \int_0^z\mathcal{V}(\dr y)\int_{\bbR_-}\hat{\mathcal{V}}(\dr x)k(z - y - x),
\end{equation}
where $\hat{\mathcal{V}}$ denotes the dual renewal function, $\mathcal{V}(\dr x)$ and $\hat{\mathcal{V}}(\dr x)$ the Stieltjes measures associated with the non-decreasing function $\mathcal{V}$ and $\hat{\mathcal{V}}$ and the function $k$ depends on $v$ and on $q$. 
This identity is derived using martingale techniques and tools from the excursion theory of $(Z_t)_{t\geq0}$ below its supremum.
With this key identity at hand, McGill proves under some technical conditions his main result: $v(z) \sim C(q)\mathcal{V}(z)$ as $z\to0$, where $C(q)$ is some unknown constant.
We emphasize that we can (more or less) recover \eqref{eq_mcgill} and his main results from our work.
Our interpretation of \eqref{eq_mcgill} is the following: the random variable $\xi_e$ can be decomposed as
\[
 \xi_e = \xi_{g_e} + \Delta_e\vee0,
\]
where $\xi_{g_e}$ and $\Delta_e := I_e- (\xi_{g_e}-\zeta_{g_e})$ are independent, see \eqref{def_xi_delta}. 
Our bivariate Wiener--Hopf factorization yields the following result (see Section~\ref{section_laplace}): $\bbP(\xi_{g_e} \in \dr z) = C(q) \mathcal{V}_q(\dr z)$, where the constant $C(q)$ is known and $\mathcal{V}_q$ is a non-decreasing function, which is close to $\mathcal{V}$ in the sense that for any fixed $z \geq0$, $\mathcal{V}_q(z)$ increases to $\mathcal{V}(z)$ as $q\to0$. 
Then, we claim that we are able to show that $v(z)\sim C(q)\mathcal{V}_q(z)$ as $z\to0$.
To match the result of Mcgill, it remains to show that $\mathcal{V}_q(z) \sim \mathcal{V}(z)$ as $z\to0$ and we believe this should hold, at least if $\mathcal{V}$ is regularly varying at $0$: we can show that their Laplace transforms are equivalent at infinity.

We insist on the fact that all of the above approach more or less relies on the fact that the integrated process is a Brownian motion, whereas we are able to consider a very large class of Markov process and of functions $f$, leading to a wide range of possible behaviors.

\subsection{Related problems and open questions}

Let us now give an overview of questions that we have chosen not to develop in the present paper, that either fall in the scope of our method or represent important challenges.

\paragraph*{Further examples in our framework.}
Let us give a couple of additional examples that we are able to treat with our method (but that we have chosen not to develop), since they are defined via their excursions (and satisfy Assumption~\ref{assump_exc}).

\medskip
\noindent
\textit{Jumping-in diffusions.\ }
A jumping-in diffusion in $\bbR$ is a strong Markov process $(X_t)_{t\geq0}$ which has continuous paths up until its first hitting time of $0$. 
When the process touches $0$, it immediately jumps back into $\bbR\setminus\{0\}$ and starts afresh. 
These processes are somehow again a generalization of diffusion processes and are typically constructed via excursion theory; we refer to the book of Itô \cite{ito2015poisson} for more details. 
Such processes obviously satisfy Assumption~\ref{assump_exc} and we can apply our results.

\medskip
\noindent
\textit{Stable processes reflected on their infimum.\ }
The following example is not related to a diffusion process and shows that our method are not only concerned with generalized diffusions.
Let $(S_t^\alpha)_{t\geq0}$ be an $\alpha$-stable process with some $\alpha \in(0,2)$ such that $(|S_t^\alpha|)_{t\geq0}$ is not a subordinator. 
We consider $(R_t^\alpha)_{t\geq0} = (S_t^\alpha - \inf_{[0,t]}S_s^\alpha)_{t\geq0}$ the process reflected on its infimum. 
It is a positive strong Markov process and $0$ is a regular and recurrent point for $(R_t^\alpha)_{t\geq0}$. 
Informally, we construct the strong Markov process $(X_t)_{t\geq0}$ by concatenating independent excursions of $(R_t^\alpha)_{t\geq0}$ flipped to the negative half-line with probability $1/2$. 
Then it also clear that $(X_t)_{t\geq0}$ satisfies Assumption~\ref{assump_exc}.
It is also clear that this process is self-similar with index $\alpha$, and as for (skew-)Bessel processes (see Section~\ref{sec:examples}), it entails that Assumption~\ref{assump_null} is satisfied if we consider $f$ to be homogeneous, for instance $f(x) = \mathrm{sgn}(x)|x|^\gamma$ for some $\gamma \in \mathbb R$ (with some restriction to ensure that $\zeta_t<+\infty$ a.s., \textit{e.g.}\ $\gamma\geq 0$).

\paragraph*{Asymptotics uniform in $t,z$.}
A natural question that we have chosen not to investigate further is the case when the barrier $z$ is ``far away''. In other words, we are interested in knowing for which regimes of $t,z$ the asymptotics~\eqref{eq:asympPT} remains valid. We would like to find some function $\varphi(\cdot)$ with $\varphi(t)\to\infty$ as $t\to\infty$ so that the following statement holds:
\begin{equation}
\label{eq:asymp++}
\bbP(T_z>t) \sim \mathcal{V}(z) \varsigma(t)\, t^{-\theta}
\quad \text{ uniformly for $t,z>0$ with } \varphi(t)/z \to \infty \,.
\end{equation}
In particular, this would allow to consider both the case when $z$ is fixed and $t\to\infty$ and the case when $t$ fixed and $z\downarrow 0$.

We have not pursued this issue further to avoid lengthening further the paper, but we believe that our method should work to obtain such a result.
Indeed, thanks to~\eqref{eq:constantc1}, we are reduced to estimating $\bbP(\xi_{g_e}<z)$ in terms of $t,z$, which is standard in fluctuation theory for Lévy processes, see~\cite{Kyprianou14}. 
Since we have identified that the correct scaling for $\xi_{g_e}$ (and $\xi_{g_e}-\zeta_{g_e}$) is $a(b^{-1}(q))$ under Assumption~\ref{assump_null}, it is natural to expect that~\eqref{eq:asymp++} holds with $\varphi(t)=a(b^{-1}(t))$.

In view of our proof, the main (and only?) step where an improvement is needed is in a control of the convergence $\mathcal{V}_q(z)\uparrow \mathcal{V}(z)$ as $q\downarrow 0$, where $\mathcal{V}_q$ is defined in~\eqref{def:Vq}, see Corollary~\ref{coro_exp}.

\paragraph*{Additive functional conditioned on being negative.}
In Section~\ref{sec:hitting}, we construct the additive functional $(\zeta_t,X_t)$ (under $\bbP_{(z,x)}$ with $(z,x) \neq 0$) conditioned to remain negative, see Proposition~\ref{prop:conditioned}. But there are several questions that we have left open.

\smallskip
\noindent
\textit{Starting from $(0,0)$.\ }
First of all, a natural question would be to construct the additive functional conditioned to stay negative, but with starting point $(0,0)$, that is starting from $0$ with a zero speed.
One should take the limit $(z,x)\to 0$ inside Proposition~\ref{prop:conditioned}, but this brings several technical difficulties and requires further work.

\smallskip
\noindent
\textit{Scaling limit of the conditioned process.\ }
Another natural question is that of obtaining the long-term behavior of the additive functional (conditioned to be negative or not) and in particular scaling limits.
Our approach should yield all necessary tools to obtain such results, and let us give an outline of what one could expect:

\begin{enumerate}[wide,label=(\roman*)]
\item  If $(X_t)_{t\geq 0}$ is positive recurrent, \textit{i.e.}\ under Assumption~\ref{assump_pos}, then one should have that the conditioned process converges to either a Brownian motion or an $\alpha$-stable Lévy process conditioned to be negative (depending on whether the process $(Z_t)_{t\geq 0}$ is in the domain of attraction of an $\alpha$-stable law with $\alpha=2$ or $\alpha\in (0,2)$).

\item  If $(X_t)_{t\geq 0}$ is null-recurrent and Assumption~\ref{assump_null} holds, then one should have that the conditioned process converges to:
\begin{itemize}
\item If $\alpha=2$, a time-changed Brownian motion conditioned to stay negative, namely $(B^0_{L_t^0})_{t\geq 0}$, where $B^0$ is a Brownian motion $Z^0$ conditioned to stay negative and $L_t^0$ is the inverse of the $\beta$-stable subordinator $\tau^0$; note that necessarily $\tau^0$ and $Z^0$ are independent.
\item If $\alpha \in (0,2)$, a ``squeleton'' given by a time-changed $\alpha$-stable Lévy process conditioned to stay negative, namely $B^0_{L_t^0}$ as above (with here $Z^0$ an $\alpha$-stable Lévy process), then ``filled'' with the integrals of excursions conditioned to bridge the squeleton (\textit{i.e.}\ excursions of lengths $\tau^0_t-\tau^0_{t-}$ with  the integral $\int_0^{\ell} (\gep_s)^{\gamma} \dd s$ conditioned to be equal to $Z^0_{t}-Z^0_{t-}$). 
\end{itemize}
\end{enumerate}
We refer to Figure~\ref{fig:conditioned} for illustrative simluations.
Let us also mention that one might also be interested in some local limit theorems for the conditioned process, in the spirit of~\cite{Car05} for random walk and \cite{GLL18AOP} for additive functional of Markov chain under a spectral gap assumption. 
\begin{figure}
\begin{center}
\includegraphics[scale=0.26]{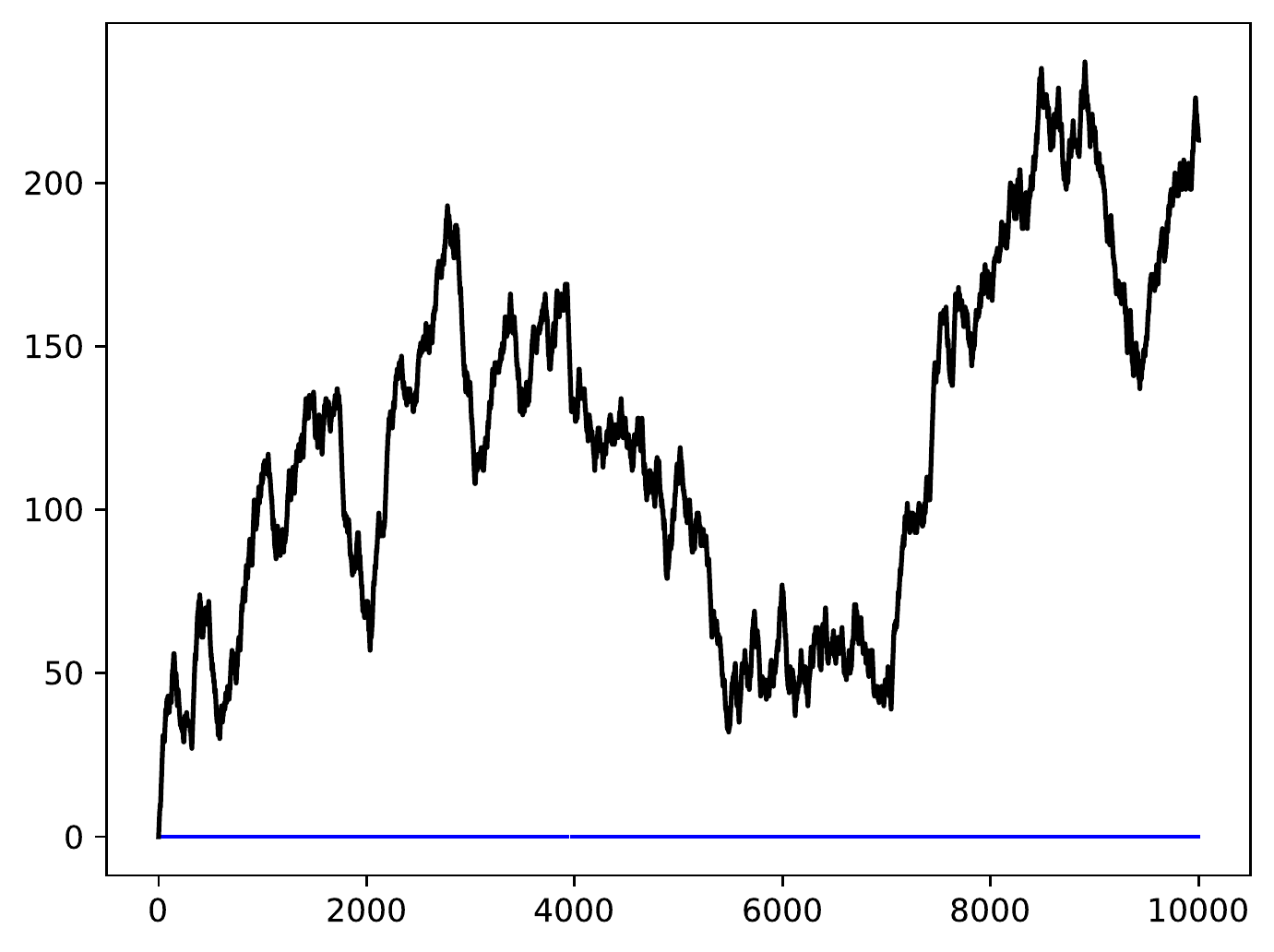}\quad
\includegraphics[scale=0.26]{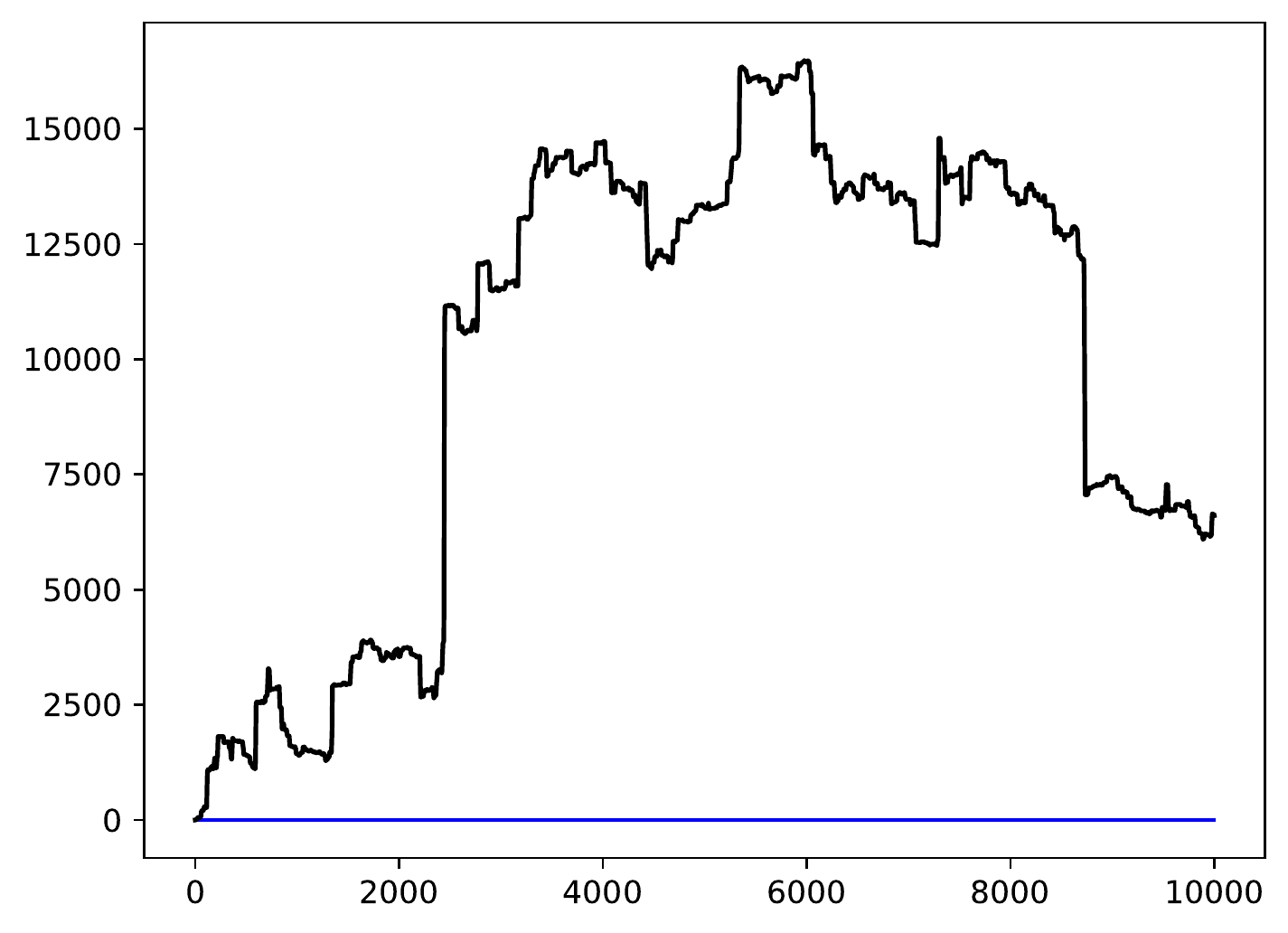}\quad
\includegraphics[scale=0.26]{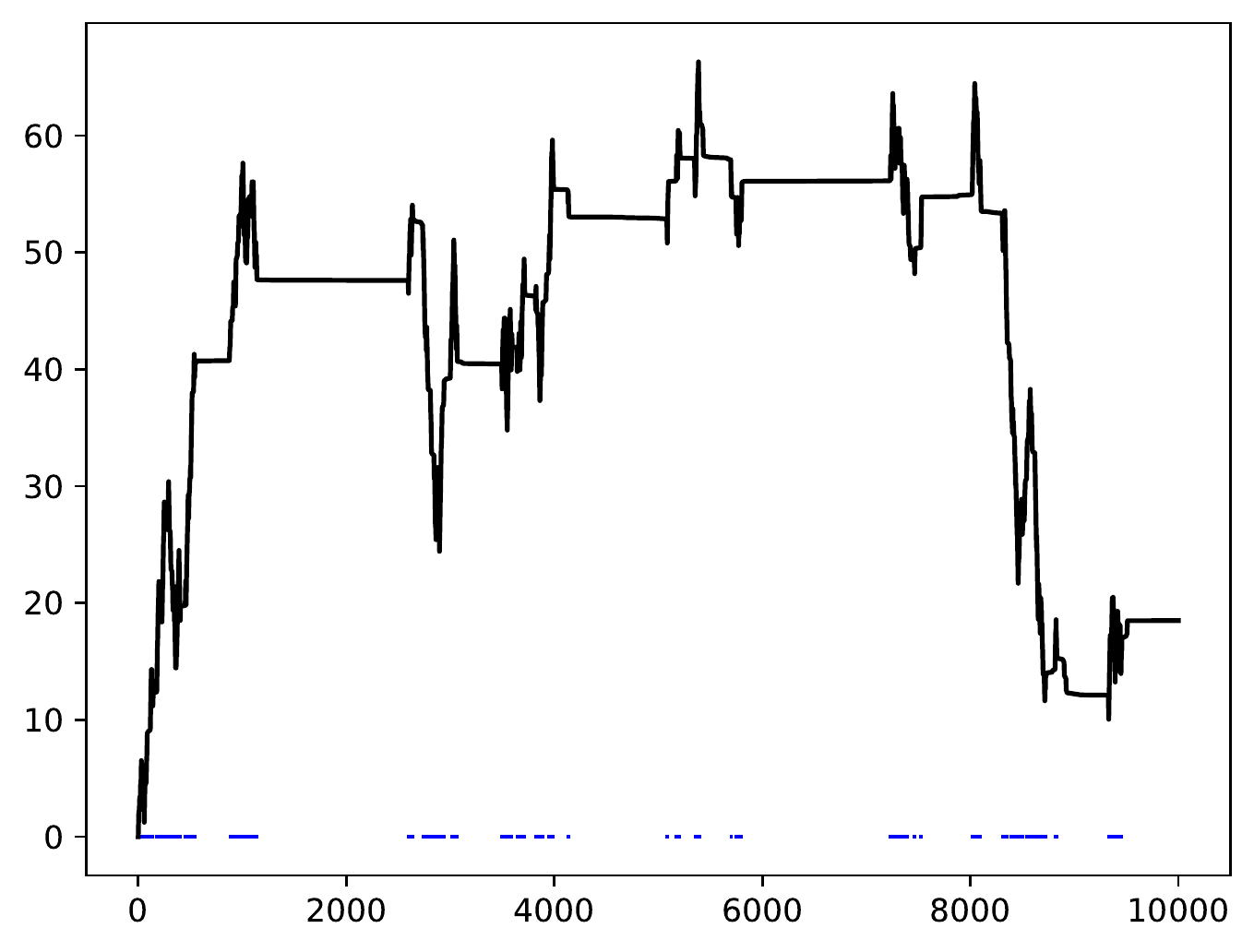}\quad
\includegraphics[scale=0.26]{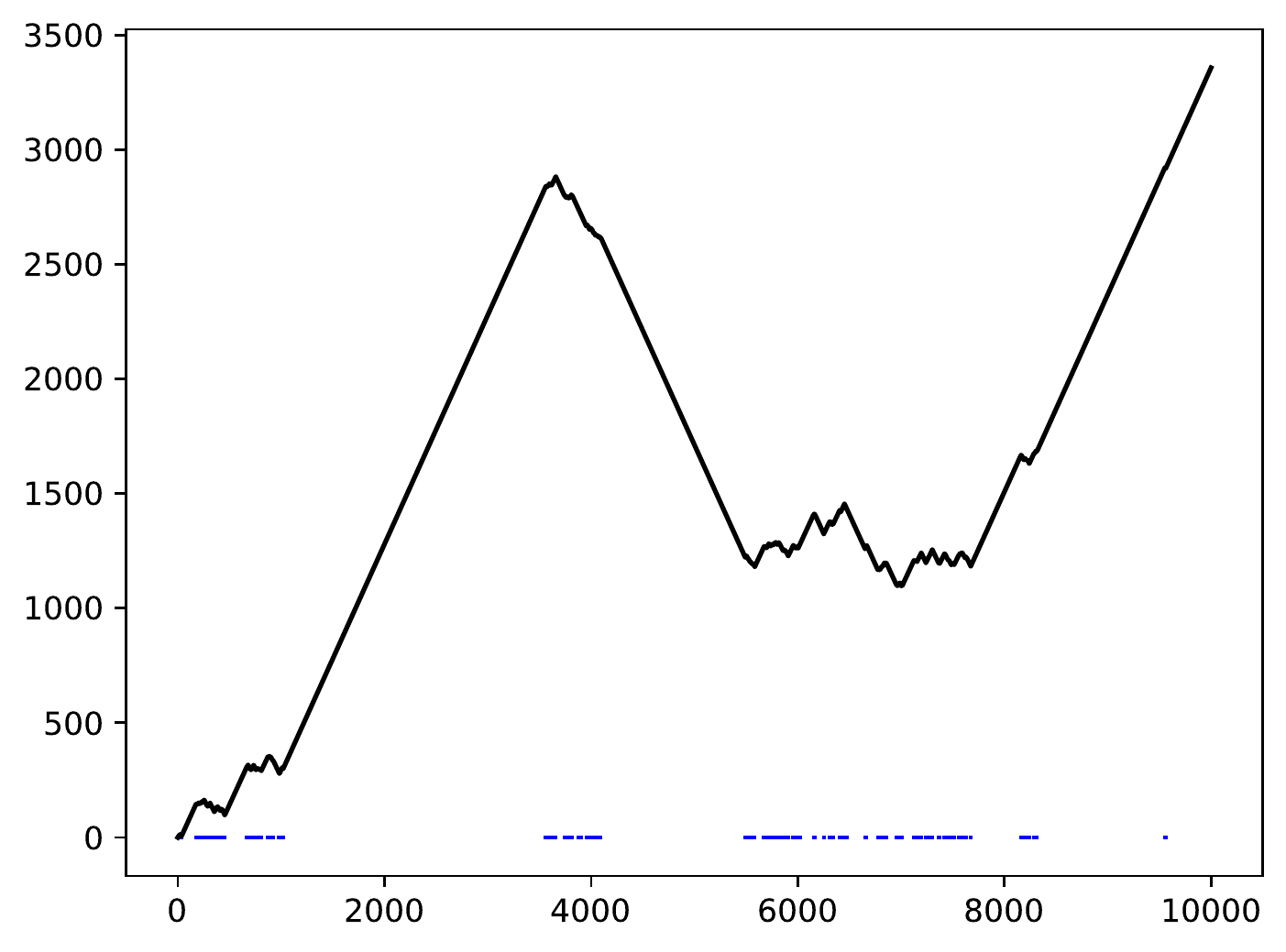}
\end{center}
\caption{\footnotesize Graphical representation of the trajectory of a realization of $(\zeta_s)_{0\leq s \leq t}$ conditioned to be positive. These are simulations in the setting of Bessel-like random walks: $(X_t)_{t\geq 0}$ is a (symmetric) Bessel-like random walk and $f(x) =\sign(x) |x|^{\gamma}$ for some $\gamma \in \mathbb R$; the dots represent the returns to $0$ of $(X_s)_{0\leq s \leq t}$.
The first two simulations are when $(X_t)_{t\geq 0}$ is positive recurrent and $(Z_t)_{t\geq 0}$ converges either to a Brownian motion (first) or an $\alpha$-stable Lévy process (second).
The last two simulations are when $(X_t)_{t\geq 0}$ is null recurrent and Assumption~\ref{assump_null} holds either with $\alpha=2$ (third, corresponding to Assumption~\ref{assump_gaus_ito}) or with $\alpha\in (0,2)$ (fourth, corresponding to Assumption~\ref{assump_m_f_alpha_ito}).}
\label{fig:conditioned}
\end{figure}

\paragraph*{Additive functionals of Markov processes with jumps.}
Finally we stress that our main assumption on the process $(X_t)_{t\geq0}$, \textit{i.e.}\ Assumption~\ref{assump_exc}, is not satisfied for a large class of processes, including Lévy processes with jumps (except for the difference of two independent Poisson processes).
When Assumption \ref{assump_exc} is not satisfied, it is not clear how to attack the problem and the only result so far in this direction seems to focus only on (symmetric) \textit{homogeneous} functionals of \textit{strictly $\alpha$-stable} Lévy processes, see~\cite{Simon07}.
Generalizing this result, for instance to aymptotically $\alpha$-stable processes, remains an important open problem.

\subsection{Overview of the rest of the paper}

The rest of the paper is organized as follows.
\begin{itemize}[wide]
 \item  In Section \ref{sec:keytools}, we develop several key tools that are crucial in our proof. We decompose the paths of the process $(X_t)_{t\geq0}$ around the last excursion and we establish a Wiener--Hopf factorization for the bi-dimensional Lévy process $(\tau_t, Z_t)_{t\geq0}$, using the excursions of $Z_t$ below its supremum, in the spirit of Greenwood-Pitman \cite{gp}. With these tools at hand we are able to compute the Laplace transform of $(\xi_{g_e}, \xi_{g_e}-\zeta_{g_e})$ which is the key identity of our paper, see Corollary~\ref{laplace_transform_x_zeta}, and seems to be new to the best of our knowledge.
 
 \item Using the tools of Section~\ref{sec:keytools}, we are able to show our persistence result in two different settings. In Section~\ref{sec_main_thm_rec_pos} we prove Theorem \ref{main_thm_rec_pos} under the assumption that $0$ is positive recurrent; in Section~\ref{sec:proof_null} we prove Theorem \ref{main_theorem} under Assumption \ref{assump_null} (corresponding to the case where $0$ is null recurrent).
 
 \item In Section~\ref{section_ito_mc_kean} we consider the case of generalized diffusions and we give tangible conditions which ensure that Assumption  \ref{assump_null} holds.
 
 \item Finally in section \ref{sec:hitting} we treat the case where $(X_t,\zeta_t)_{t\geq0}$ does not start at $(0,0)$. To handle the proof we assume that $(X_t)_{t\geq0}$ has continuous paths.
 
 \item We also collect some technical results in Appendix: in Appendix~\ref{appendix_wiener} we prove our Wiener--Hopf factorization are related results (based on well-established techniques), in Appendix~\ref{appendix_levy} we present results on convergences of Lévy processes and of their Laplace exponents that serve in Section~\ref{section_ito_mc_kean}, in Appendix~\ref{app:generaldiffusion} we give technical results on generalized one-dimensional diffusions that are used in Section~\ref{sec:hitting}.
\end{itemize}

\section{Path decomposition and Wiener--Hopf factorization}\label{sec:keytools}

\subsection{Preliminaries on excursion theory}
\label{sec:exc_theory}

Recall we assumed that $0$ is regular for itself, that is $\mathbb{P}(\eta_0 = 0) = 1$ where $\eta_0:= \inf\{t > 0, \: X_t = 0\}$ (and $\mathbb{P} = \mathbb{P}_0$), and that $0$ is a recurrent point.
In this setting, we have a theory of excursions of $(X_t)_{t\geq0}$ away from $0$, see for instance Bertoin \cite[Chapter IV]{b} or Getoor \cite[Section 7]{getoor_exc}.
Setting $S=\inf\{t\geq0, \: X_t \neq0\}$, then Blumenthal 0-1 law gives that either $\mathbb{P}(S=0) = 1$ or $\mathbb{P}(S=0)=0$:
in the first case, $0$ is called an \emph{instantaneous} point;
in the second case, $0$ is called a \emph{holding} point.

\smallskip
The process $(X_t)_{t\geq0}$ possesses a local time $(L_t)_{t\geq0}$ at the level $0$, in the sense that there is a non-decreasing and continuous additive functional whose support is the closure of the zero set of $(X_t)_{t\geq0}$. Then there exists some $\mathrm{m}\geq0$ such that a.s.\ for any $t\geq0$
\[
 \mathrm{m} L_t = \int_0^t\bm{1}_{\{X_s = 0\}}\dr s.
\]

\noindent
The right-continuous inverse $(\tau_t)_{t\geq0}$ of the local time  is a (c\`adl\`ag) subordinator and we will denote by $\Phi$ its Laplace exponent, see~\eqref{def:Phi}.
The subordinator $(\tau_t)_{t\geq0}$ has the following representation: for any $t\geq0$
\[
 \tau_t = \mathrm{m}t + \sum_{s\leq t}\Delta\tau_s,
 \quad 
 \text{ with } \Delta\tau_s := \tau_s- \tau_{s-}\,.
\]

\noindent
Let us also recall that there are two cases:
\begin{itemize}[leftmargin=*]
 \item If $0$ is an instantaneous point, then the Lévy measure of $(\tau_t)_{t\geq0}$ has infinite mass.
 \item If $0$ is a holding point, then the Lévy measure of $(\tau_t)_{t\geq0}$ has finite mass and $\mathrm{m}>0$. In other words, $(\tau_t)_{t\geq0}$ is a drifted compound Poisson process.
\end{itemize}

\medskip
We denote by $\mathcal{D}$ the usual space of c\`adl\`ag functions from $\bbR_+$ to $\bbR$, and for $\varepsilon\in\mathcal{D}$, let us introduce the length of $\gep$
\[
 \ell(\varepsilon) = \inf\{t > 0,\: \varepsilon_t = 0\}\,,
\]
and we will often write for simplicity $\ell$ instead of $\ell(\varepsilon)$.

Then the set of excursions $\mathcal{E}$ is the set of functions $\varepsilon\in\mathcal{D}$ such that:
(i) $0 < \ell(\varepsilon) < \infty$;
(ii) $\varepsilon_t = 0$ for every $t\geq\ell(\varepsilon)$;
(iii) $\varepsilon_t \neq 0$ for every $0 < t < \ell(\varepsilon)$.
This space is endowed with the usual Skorokhod's topology and the associated Borel $\sigma$-algebra. 

\smallskip
We now introduce the excursion processes of $(X_t)_{t\geq0}$, denoted by $(e_t)_{t\geq0}$, which take values in $\mathcal{E}_0=\mathcal{E}\cup\{\Upsilon\}$, where $\Upsilon$ is an isolated cemetery point. The excursion process is given by
\begin{equation*}
  e_t = 
  \begin{cases}
   (X_{\tau_{t-} + s})_{s\in[0,\Delta\tau_t]}& \text{if } \Delta\tau_t >0,  \\
   \Upsilon &\text{otherwise},. 
  \end{cases}
\end{equation*}
A famous result, essentially due to Itô \cite{i72}, states that $(e_t)_{t\geq0}$ is a Poisson point process and we denote by $\nn$ its characteristic measure, which is defined for any measurable set $\Gamma$ by
\begin{equation*}\label{def_exc_measure}
 \nn(\Gamma) = \frac{1}{t}\mathbb{E}\left[N_{t,\Gamma}\right]\qquad\text{where}\qquad N_{t,\Gamma} = \sum_{s\leq t}\bm{1}_\Gamma(e_s) \,.
\end{equation*}

\noindent
Here again, let us distinguish two cases:
\begin{itemize}[leftmargin=*]
 \item If $0$ is an instantaneous point, then $\nn$ has infinite mass
 \item If $0$ is a holding point, then $\nn$ has finite mass and is proportional to the law of $e_1$ under $\mathbb{P}$.
\end{itemize}

For a non-negative measurable function $F: \mathcal{E}_0 \to \mathbb R$, we also write $\nn(F)=\int_{\mathcal{E}} F(\gep) \nn (\dd \gep)$.
Let us stress that, from the exponential formula for Poisson point processes, we can express the Laplace exponent $\Phi$ of $(\tau_t)_{t\geq 0}$ in terms of $\mathrm{m}$ and $\nn$:
\[
 \Phi(q) = \mathrm{m}q + \nn(1-\e^{-q\ell}).
\]
%
%

Recall the definition~\eqref{def:Z} of the Lévy process $(Z_t)_{t\geq0}$, and recall Assumption~\ref{assump-f} on the function~$f$.
Let us stress that if $0$ is an instantaneous point, then the Lévy measure $\nu$ has infinite mass, whereas if $0$ is a holding point, then $(Z_t)_{t\geq0}$ is a compound Poisson process.

Recalling the last expression in~\eqref{def:Z}, the Lévy measure $\nu$ of $(Z_t)_{t\geq 0}$ can be expressed through the excursion measure $\nn$.
For $\epsilon\in\mathcal{E}$, we define 
\begin{equation}
\label{def:fepsilon}
\f(\varepsilon) = \int_0^{\ell}f(\varepsilon_s)\dr s \,.
\end{equation}
Then the exponential formula for Poisson point processes ensures that $\nu(\dr z) = \nn(\f \in \dr z)$.
Then, Assumption~\ref{assump_exc} that the process $(X_t)_{t\geq 0}$ cannot change sign without touching $0$ (and is not of constant sign) can be reformulated as follows:

\begin{assumpprime}
\label{assump_2_prime}
The measure $\nn$ is supported by the set of excursions of constant sign. Moreover, $\nn$ is neither supported by the set of positive excursions, nor by the set of negative excursions.
\end{assumpprime}

This assumption, combined with Assumption \ref{assump-f} on the function $f$, leads to the following remark, which is crucial in our study.

\begin{remark}\label{sup_remark}
Under Assumptions \ref{assump-f} and \ref{assump_2_prime}, $(|Z_t|)_{t\geq0}$ is not a subordinator. Moreover we have almost surely $\sup_{[0, \tau_t]}\zeta_s = \sup_{[0, t]}Z_s$, for every $t\geq0$. Indeed, for every $s\geq0$, $(X_t)_{t\geq0}$ is of constant sign on the time-interval $[\tau_{s-}, \tau_s]$ and consequently $t\mapsto \zeta_t$ is monotone on every such interval. As a consequence, the supremum is necessarily reached at the extremities, i.e.\ we have $\sup_{u\in[\tau_{s-},\tau_s]}\zeta_u = \zeta_{\tau_{s-}}\vee\zeta_{\tau_{s}}$.
\end{remark}

\subsection{Decomposing paths around the last excursion}

We now show a path decomposition at an exponential random time. Recall the definition~\eqref{def:xig} of $g_t=\sup\{s \leq t\,, X_s=0\}$, the last return to $0$ of $X$ before time $t$.
The following result, inspired by Salminem, Vallois and Yor \cite[Thm 9]{svy} allows us to decouple the path of $X$ at an exponential random time into two independent parts, before and after the last return to $0$.
An important application of this independence has already been presented in~\eqref{useindep}.
Additionally, Proposition~\ref{indep} provides useful formulas for computing functionals of the path before and after the last zero, respectively.

To make the statement more precise, let us introduce some notation: for $t> 0$ we let $\mathcal{D}_t$ denote the space of c\`adl\`ag funtions from $[0,t)$ to $\mathbb R$, and for $t\geq 0$ we let $\widetilde{\mathcal{D}}_t$ denote the space of c\`adl\`ag funtions from $[0,t]$ to $\mathbb R$.

\begin{proposition}\label{indep}
Let $e = e(q)$ be an exponential random variable of parameter~$q$, independent of~$(X_t)_{t\geq 0}$. Then the processes $(X_u)_{0\leq u < g_e}$ and $(X_{g_e + v})_{0 \leq v \leq e-g_e}$ are independent.
Moreover, for all non negative functionals $F_1 : \bigcup_{t> 0} \mathcal D_t \to \mathbb R_+$ and $F_2:  \bigcup_{t\geq 0}\widetilde{\mathcal{D}}_t \to \mathbb R_+$, we have
\begin{equation}
\label{eq:F1}
\mathbb{E}\left[F_1((X_u)_{0\leq u < g_e})\right] = \Phi(q)\int_0^{\infty}\mathbb{E}\left[F_1((X_u)_{0\leq u < \tau_t}) \e^{-q\tau_t}\right] \dr t
\end{equation}
and
\begin{equation}
\label{eq:F2}
\mathbb{E}\left[F_2((X_{g_e + v})_{0 \leq v \leq e-g_e})\right] = \frac{q}{\Phi(q)}\bigg(\mathrm{m}F_2(\tilde{0}) + \nn\Big( \int_0^{\ell} e^{-q u} F_2\big((\gep_{v})_{0\leq v\leq u} \big) \dd u \Big)  \bigg) \,,
\end{equation}
where $\tilde{0}$ denotes the null function of $\widetilde{\mathcal{D}}_0$.
\end{proposition}

Let us stress that the strict inequality in considering $(X_u)_{0\leq u<g_e}$ is crucial in the proof.

\begin{proof}
Let $F_1$ and $F_2$ be as in the statement. We have
\begin{align*}
  \mathbb{E}\big[F_1((X_u)_{0\leq u< g_e}) & F_2((X_{g_e + v})_{0 \leq v \leq e-g_e})\big] \\
      & \qquad =  \mathbb{E}\bigg[q\int_0^{\infty}\e^{-qt}F_1((X_u)_{0\leq u< g_t})F_2((X_{g_t + v})_{0 \leq v \leq t-g_t}) \dr t  \bigg] \,.
\end{align*}
We decompose the above integral in two parts, according to whether $X_t=0$ or not. 
If $g$ and $d$ denotes the left and right endpoints of an excursion, then the excursion straddling time $t$ is the only excursion such that $g\leq t \leq d$. 
Indexing the excursions by their left endpoint $g$, \textit{i.e.}\ by $\mathcal G = \{\tau_{t-}; \Delta \tau_t >0, t\in \mathbb R_+\}$,  the previous display is then equal to
\begin{multline*}
 \mathbb{E}\bigg[q\int_0^{\infty}\e^{-qt}F_1((X_u)_{0\leq u<t})F_2(\tilde{0})\bm{1}_{\{X_t = 0\}} \dr t \bigg] \\
   + \mathbb{E}\bigg[q\int_0^{\infty}\e^{-qt}\sum_{g\in \mathcal G} F_1((X_u)_{0\leq u< g})F_2((X_{g + v})_{0 \leq v \leq t-g})\ind_{\{g< t < d\}} \dr t \bigg] \,.
\end{multline*}
Regarding the first term, it is equal to
\[
 \mathrm{m}F_2(\tilde{0})q\mathbb{E}\bigg[\int_0^{\infty}\e^{-qt}F_1((X_u)_{0\leq u<t})\dr L_t \bigg] = \mathrm{m}F_2(\tilde{0})q\int_0^{\infty}\mathbb{E}\left[F_1((X_u)_{0\leq u < \tau_t}) \e^{-q\tau_t}\right] \dr t.
\]
For the second term, it is equal to
\[
 \mathbb{E}\bigg[\sum_{g\in \mathcal G} F_1((X_u)_{0\leq u< g}) \e^{-qg}\int_0^{\ell} q \e^{-qt} F_2((X_{g + v})_{0 \leq v \leq t}) \dr t\bigg] \,,
\]
where we simply used that $(\tau_t)_{t\geq 0}$ is the right-continuous inverse of $L_t$.

\noindent
Using the compensation formula for Poisson point processes, we end up with
\begin{align*}
\mathbb{E} \bigg[\int_0^{\infty}F_1((X_u)_{0\leq u< \tau_t})&\e^{-q\tau_t}\nn\Big( \int_0^{\ell} e^{-q u} F_2\big((\gep_{v})_{0\leq v\leq u} \big) \dd u \Big) \dr t \bigg] \\
&  = \int_0^{\infty}\mathbb{E}\left[F_1((X_u)_{0\leq u< \tau_t}) \e^{-q\tau_t}\right] \dr t \times\nn\Big( \int_0^{\ell} e^{-q u} F_2\big((\gep_{v})_{0\leq v\leq u} \big) \dd u \Big) \, .
\end{align*}

\noindent
To summarize, we have showed that $\mathbb{E}[F_1((X_u)_{0\leq u< g_e})F_2((X_{g_e + v})_{0 \leq v \leq e-g_e})]$ is equal to
\[
 \int_0^{\infty}\mathbb{E}\left[F_1((X_u)_{0\leq u < \tau_t}) \e^{-q\tau_t}\right] \dr t \times q\bigg(\mathrm{m} F_2(\tilde 0) + \nn\Big( \int_0^{\ell} e^{-q u} F_2\big((\gep_{v})_{0\leq v\leq u} \big) \dd u \Big) \bigg),
\]
and the result follows since we have $\Phi(q) = (\int_0^{\infty} \bbE[e^{-q\tau_t}] \dd t)^{-1}= \mathrm{m}q + \nn(1- \e^{-q\ell})$.
\end{proof}
%
%
%
%
%

\subsection{Wiener--Hopf factorization}\label{section_WH}

In this section, we derive a Wiener--Hopf factorization for the bivariate Lévy process $(\tau_t, Z_t)_{t\geq0}$. This factorization is similar to the one in Isozaki \cite{Isozaki96} although the factorization therein is somehow incomplete and our method is a bit different: we follow the approach of Greenwood and Pitman \cite{gp} which can also be found in Bertoin \cite[Chapter VI]{b}.
We will only display in this section the results needed as it is not the main purpose of the paper: the proofs are postponed to Appendix~\ref{appendix_wiener}.
In fact, our results hold for any bivariate Lévy process $(\tau_t, Z_t)_{t\geq0}$ for which $(\tau_t)_{t\geq 0}$ is a subordinator.

\smallskip
We define $S_t= \sup_{[0,t]} Z_t$ the running supremum of $Z_t$. Then the reflected process $(R_t)_{t\geq0} = (S_t - Z_t)_{t\geq0}$ is a strong Markov process (see \cite[Proposition VI.1]{b}) and possesses a local time $(L_t^{R})_{t\geq 0}$ at $0$ and we denote by $(\sigma_t)_{t\geq0}$ its right-continuous inverse, called the \textit{ladder time} process. 
Next we define $\theta_t := \tau_{\sigma_t}$ and $H_t:=S_{\sigma_t}$ the \textit{ladder heights processes}.
Then
$(\sigma_t, \theta_t, H_t)_{t\geq0} := (\sigma_t, \tau_{\sigma_t}, S_{\sigma_t})_{t\geq0}$ is a trivariate subordinator possibly killed at some exponential random time, according to wether or not $0$ is recurrent for $(R_t)_{t\geq0}$ (see Lemma~\ref{lem:trivariate} in  Appendix~\ref{appendix_wiener}). We denote by $\kappa$ its Laplace exponent: for any non-negative $\alpha,\beta,\gamma$:
\begin{equation}
\label{def:kappa}
\kappa(\alpha, \beta,\gamma) = -\log \bbE\left[e^{-\alpha \sigma_1 -\beta \theta_1 -\gamma H_1}\bm{1}_{\{1 < L_{\infty}^R\}}\right] \,.
\end{equation}
If $0$ is transient for $(R_t)_{t\geq0}$, then $\kappa(0,0,0) > 0$ and $L_{\infty}^R$ has an exponential distribution of parameter $\kappa(0,0,0)$ whereas if $0$ is recurrent, then $\kappa(0,0,0) = 0$ and $L_{\infty}^R = \infty$ a.s. In any case, using the convention $\e^{-\infty} = 0$, we have for any $\alpha,\beta,\gamma\geq0$ such that $\alpha + \beta + \gamma > 0$,
\begin{equation}\label{laplace_kappa}
 \bbE\left[\e^{-\alpha \sigma_t -\beta \theta_t -\gamma H_t}\right] = \exp\left(-\kappa(\alpha,\beta,\gamma)t\right).
\end{equation}
Let us also set $G_t = \sup\{s<t, Z_s = S_s\}$ be the last time before $t$ where $(Z_t)_{t\geq 0}$ attains its supremum, or equivalently the last return to $0$ before $t$ of $(R_t)_{t\geq0}$.

We now state a Wiener--Hopf factorization for the process $(\tau_t, Z_t)_{t\geq 0}$, which will allow us below to obtain (among other things), the joint Laplace transform of $\xi_{g_e}$ and $\xi_{g_e}-\zeta_{g_e}$, see Corollary~\ref{laplace_transform_x_zeta} below. Let us stress that we need to separate two cases, according to whether $0$ is regular or irregular for the reflected process $(R_t)_{t\geq 0}$.

\begin{theorem}[\bf{Wiener--Hopf factorization}]\label{WH}
Let $e = e(q)$ be an exponential random variable of parameter $q$, independent of $(\tau_t,Z_t)_{t\geq 0}$. 

\noindent
(A) If $0$ is irregular for the reflected process $(R_t)_{t\geq0}$, then we have the following:
\begin{enumerate}[label=(\roman*)]
    \item The triplets $(G_{e}, \tau_{G_{e}}, S_{e})$ and $(e-G_{e}, \tau_{e} - \tau_{G_{e}}, Z_{e} - S_{e})$ are independent.
    
    \item The law of $(G_{e}, \tau_{G_{e}}, S_{e})$ is infinitely divisible and its L\'evy measure is
\[
\mu_{+}(\dr t, \dr r, \dr x) = \frac{\e^{-q t}}{t}\mathbb{P}(\tau_t\in \dr r, Z_t\in \dr x)\dr t,\quad \text{ for } t > 0,\: r \geq 0,\: x \geq 0.
\]
    
    \item The law of $(e-G_{e}, \tau_{e} - \tau_{G_{e}}, Z_{e} - S_{e})$ is infinitely divisible and its Levy measure is
\[
\mu_{-}(\dr t, \dr r, \dr x) = \frac{\e^{-q t}}{t}\mathbb{P}(\tau_t\in \dr r, Z_t\in \dr x)\dr t,,\quad \text{ for }  t > 0,\: r \geq 0,\: x < 0.
\]
\end{enumerate}

\noindent
(B) If $0$ is regular for the reflected process, then the same statements hold with $\tau_{G_e}$ replaced by $\tau_{G_e-}$ and $\tau_e - \tau_{G_e}$ replaced by $\tau_e - \tau_{G_e -}$.
\end{theorem}

Let us give a corollary,  which is key in our analysis, that expresses the Laplace transform of $(S_e, S_e-Z_e, \tau_e)$; let us stress that a more general formula is available, see~\eqref{eq2} in Appendix~\ref{appendix_wiener}.

\begin{proposition}\label{wiener_prop_corps}
Let $e = e(q)$ be an exponential random variable of parameter $q$, independent of $(\tau_t,Z_t)_{t\geq 0}$.
There exists a constant $c>0$ such that
\begin{equation*}
\label{eq:kappa}
\kappa(\alpha, \beta, \gamma) = c \exp\bigg(\int_0^{\infty}\int_{[0,\infty)\times\mathbb{R}}\frac{\e^{-t} - \e^{-\alpha t- \beta r - \gamma x}}{t}\bm{1}_{\{x\geq0\}}\mathbb{P}(\tau_t\in \dr r, Z_t\in \dr x)\dr t\bigg).
\end{equation*}
Also, for any positive $\alpha, \beta, \gamma$, we have 
\begin{equation} \label{eq1}
\mathbb{E} \left[ \e^{-\alpha S_{e} -\beta(Z_{e} -S_{e}) - \gamma \tau_e} \right] = \frac{\kappa(q,0,0)}{\kappa(q, \gamma, \alpha)}   \frac{\bar{\kappa}(q,0,0)}{\bar{\kappa}(q, \gamma, \beta)} \,,
\end{equation}
where we have defined
 \[
  \bar{\kappa}(\alpha, \beta, \gamma) = \exp\bigg(\int_0^{\infty}\int_{[0,\infty)\times\mathbb{R}}\frac{\e^{-t} - \e^{-\alpha t- \beta r + \gamma x}}{t}\bm{1}_{\{x<0\}}\mathbb{P}(\tau_t\in \dr r, Z_t\in \dr x)\dr t\bigg).
 \]
\end{proposition}

\begin{remark}
\label{rem_dual}
We can also introduce the same objects as above for the dual process $(\hat{Z}_t)_{t\geq0} = (-Z_t)_{t\geq0}$. If we set $(\hat S_t)_{t\geq 0} = (\sup_{[0,t]} \hat Z_t)_{t\geq 0}$, then the dual reflected process $(\hat{R}_t)_{t \geq 0}= (\hat{S}_t-\hat{Z}_t)_{t\geq 0}$ also posesses a local time at $0$ denoted by $(L_t^{\hat{R}})_{t\geq 0}$, with right-continuous inverse $(\hat{\sigma}_t)_{t\geq0}$.
Finally, we set $(\hat{\sigma}_t, \hat{\theta}_t, \hat{H}_t)_{t\geq0} = (\hat{\sigma}_t, \tau_{\hat{\sigma}_t}, \hat{S}_{\hat{\sigma}_t})_{t\geq0}$, which is again a trivariate subordinator possibly killed at some exponential random time with Laplace exponent that we denote~$\hat{\kappa}$, \textit{i.e.}\ such that
$\hat \kappa(\alpha, \beta,\gamma) = -\log \bbE\left[\e^{-\alpha \sigma_1 -\beta \theta_1 -\gamma H_1}\bm{1}_{\{1 < L_{\infty}^{\hat{R}}\}}\right] $ for any non-negative $\alpha, \beta,\gamma$.
Then, we prove in Appendix~\ref{appendix_wiener} that if $(Z_t)_{t\geq0}$ is not a compound Poisson process, there exists a constant $\hat{c}>0$ such that $\hat{\kappa} = \hat{c}\,\bar{\kappa}$. 
\end{remark}
 
%

\subsection{Consequences of the Wiener--Hopf factorization}
\label{section_laplace}

From Proposition~\ref{wiener_prop_corps}, using also Proposition \ref{indep}, we can compute the Laplace transform of $(\xi_{g_e},\xi_{g_e}-\zeta_{g_e})$ in terms of $\kappa$ and $\bar{\kappa}$, where $e = e(q)$ is an exponential random variable of parameter $q >0$ independent of the rest.

\begin{corollary}\label{laplace_transform_x_zeta}
Let $e = e(q)$ be an exponential random variable of parameter $q$, independent of~$(X_t)_{t\geq 0}$. Then for every $\lambda,\mu \geq 0$, we have
\begin{align}
\mathbb{E}\left[\e^{-\lambda \xi_{g_e} - \mu(\xi_{g_e} - \zeta_{g_e})}\right]= \frac{\kappa(0,q,0)}{\kappa(0, q,\lambda)} \frac{\bar{\kappa}(0,q,0)}{\bar{\kappa}(0, q,\mu)}. \label{fluctuidentity}
\end{align}
As a consequence, $\xi_{g_e}$ and $\xi_{g_e} -\zeta_{g_e}$ are independent, the law of $\xi_{g_e}$, $\xi_{g_e} -\zeta_{g_e}$ and $\zeta_{g_e}$ are infinitely divisible.
\end{corollary}


\begin{proof}
Since $(\zeta_t)_{t\geq0}$ and $(\xi_t)_{t\geq0}$ are continuous, we have $\zeta_{g_e} = \zeta_{g_e-}$ and $\xi_{g_e}=\xi_{g_e-}$ and we can apply Proposition~\ref{indep}-\eqref{eq:F1}:
\begin{align*}
 \mathbb{E}\left[\e^{-\lambda \xi_{g_e} - \mu(\xi_{g_e} - \zeta_{g_e}) }\right] & = \Phi(q)\int_0^{\infty} \mathbb{E} \left [ \e^{-\lambda \xi_{\tau_t} - \mu (\xi_{\tau_t}-\zeta_{\tau_t})-q \tau_t }   \right ] \dr t \\
  & = \Phi(q)\int_0^{\infty} \mathbb{E} \left [ \e^{-\lambda S_t - \mu (S_t-Z_t)-q \tau_t }   \right ] \dr t.
\end{align*}
In the last equality, we have used Remark \ref{sup_remark} which tells us that $\xi_{\tau_t} = S_t$ a.s. We now introduce some additional Laplace variable $\hq>0$ so that the above integral becomes a quantity evaluated at an exponential random  variable $\he$ of parameter $\hq$. We write: 
\begin{equation*}
\label{eqhq}
\begin{split}
\mathbb{E}\left[\e^{-\lambda \xi_{g_e} - \mu(\xi_{g_e} - \zeta_{g_e}) }\right] &= \lim_{\hq \to 0} \Phi(q)\int_0^{\infty} \mathbb{E} \left [ \e^{-\lambda S_t - \mu (S_t-Z_t)-q \tau_t -\hq t}   \right ] \dr t  \\
&=  \lim_{\hq \to 0} \frac{\Phi(q)}{\hq} \mathbb{E} \left [ \e^{-\lambda S_{\he} - \mu (S_{\he}-Z_{\he}) -q \tau_{\he}} \right ] \,,
\end{split}
\end{equation*}
where $\he = \he(\hq)$ is an independent exponential random variable of parameter $\hq$.
Therefore, using~\eqref{eq1}, we obtain 
\begin{align}
\mathbb{E}\left[\e^{-\lambda \xi_{g_e} - \mu(\xi_{g_e} - \zeta_{g_e}) }\right]= \lim_{\hq \to 0} \frac{\Phi(q)}{\hq} \frac{\kappa(\hq, 0,0)}{\kappa(\hq,q,\lambda)} \frac{\bar{\kappa}(\hq, 0,0)}{\bar{\kappa}(\hq, q , \mu)}. \label{a}
\end{align}
Using Frullani's formula $\int_0^{\infty} \frac{\e^{- u}-\e^{- bu}}{u} \dd u = \ln (b)$, we can write
\[
\Phi(q) =  \exp\bigg( \int_0^{\infty} \frac{\e^{-t} - \e^{-\Phi(q) t}}{t} \dd t\bigg)\,, \qquad
\frac{1}{\hq}  = \exp\bigg( -\int_0^{\infty} \frac{\e^{-t}-\e^{-\hq t}}{t} \dd t \bigg) \,,
\]
so that writing $\e^{-\Phi(q) t} = \bbE[\e^{-q\tau_t}]$, we obtain
\begin{align*}
\frac{\Phi(q)}{\hq}&= \exp \bigg (\int_0^{+\infty} \int_{0}^{+\infty}\frac{\e^{-\hq t} - \e^{-qr}}{t}\mathbb{P}(\tau_t \in \dr r) \dr t\bigg ) =  \left (\frac{\Phi(q)}{\hq} \right)^{+}  \left (\frac{\Phi(q)}{\hq} \right)^{-},
\end{align*}
 where we have set
 \begin{align*}
\left (\frac{\Phi(q)}{\hq} \right)^{+} &:=\exp \bigg (\int_0^{+\infty} \int_{[0,+\infty)\times \mathbb{R}} \frac{\e^{-\hq t} - \e^{-qr}}{t} \mathbf{1}_{ \{x\geq 0 \}}\mathbb{P}(\tau_t \in \dr r, Z_t \in \dr  x ) \dr t\bigg ) ,\\ 
\left (\frac{\Phi(q)}{\hq} \right)^{-} &:=\exp \bigg (\int_0^{+\infty} \int_{[0,+\infty)\times \mathbb{R}}\frac{\e^{-\hq t} - \e^{-qr}}{t} \mathbf{1}_{\{x< 0 \}}\mathbb{P}(\tau_t \in \dr r, Z_t \in \dr  x ) \dr t\bigg ).
 \end{align*}
 Then, by Proposition \ref{wiener_prop_corps}, we observe that for all $\hq >0$,
 \begin{align*}
\left  (\frac{\Phi(q)}{\hq} \right)^{+} \kappa(\hq, 0,0)&= c \exp \bigg (\int_0^{+\infty} \int_{[0,+\infty)\times \mathbb{R}} \frac{\e^{-t} - \e^{-qr}}{t} \mathbf{1}_{ \{x\geq 0 \}}\mathbb{P}(\tau_t \in \dr r, Z_t \in \dr  x ) \dr t\bigg ) \\
&= \kappa(0,q,0)\,,
 \end{align*}
 and 
 \begin{align*}
\left  (\frac{\Phi(q)}{\hq} \right)^{-} \bar{\kappa}(\hq, 0,0)&= \exp \bigg (\int_0^{+\infty} \int_{[0,+\infty)\times \mathbb{R}} \frac{\e^{-t} - \e^{-qr}}{t} \mathbf{1}_{ \{x< 0 \}}\mathbb{P}(\tau_t \in \dr r, Z_t \in \dr  x ) \dr t\bigg ) \\
&= \bar{\kappa}(0,q,0).
 \end{align*}
Finally equation \eqref{a} gives that
\[
\mathbb{E}\left[\e^{-\lambda \xi_{g_e} - \mu(\xi_{g_e} - \zeta_{g_e}) }\right]= \lim_{\hq \to 0}\frac{\kappa(0,q,0)}{\kappa(\hq, q,\lambda)} \frac{\bar{\kappa}(0,q,0)}{\bar{\kappa}(\hq, q,\mu)} =\frac{\kappa(0,q,0)}{\kappa(0, q,\lambda)} \frac{\bar{\kappa}(0,q,0)}{\bar{\kappa}(0, q,\mu)},
\]
which is the desired result.

Finally, we show that the law of $\xi_{g_e}$ and $\xi_{g_e} - \zeta_{g_e}$ are infinitely divisible. Let us start with~$\xi_{g_e}$. By Proposition \ref{wiener_prop_corps} and the expression of $\kappa(\alpha,\beta,\gamma)$, we have for any $\lambda > 0$,
\begin{align*}
 \mathbb{E}\left[\e^{-\lambda \xi_{g_e} }\right]= & \exp\bigg(\int_0^{\infty}\int_{[0,+\infty)\times \mathbb{R}}t^{-1}\e^{-qr}(\e^{-\lambda x} - 1)\mathbf{1}_{ \{x\geq 0 \}}\mathbb{P}(\tau_t \in \dr r, Z_t \in \dr  x ) \dr t\bigg) \\
  &  = \exp\bigg(\int_{(0,\infty)}(\e^{-\lambda x} - 1) \mu(\dr x)\bigg),
\end{align*}
where we have set
\[
 \mu(\dr x) = \bm{1}_{\{x > 0\}}\int_{[0,\infty)^2}t^{-1}\e^{-qr}\mathbb{P}(\tau_t \in \dr r, Z_t \in \dr  x ) \dr t.
\]
It only remains to show that $\int_{(0,\infty)}1\wedge x \mu(\dr x) < \infty$. We have
\[
 \int_{(0,\infty)}1\wedge x \mu(\dr x) = \int_0^{\infty}t^{-1}\mathbb{E}\left[\e^{-q\tau_t}(1\wedge Z_t) \bm{1}_{\{Z_t > 0\}}\right].
\]
It can be shown, see for instance Lemma \ref{tempspetit}, that there exists a constant $C > 0$ such that for any $t\geq0$, $\mathbb{E}[1\wedge |Z_t|] \leq C \sqrt{t}$ which is enough to check that $\int_{(0,\infty)}1\wedge x \mu(\dr x) < \infty$. This shows that $\xi_{g_e}$ is infinitely divisible. The same proof carries on for $\xi_{g_e} - \zeta_{g_e}$. Since these variables are independent, it comes that $\zeta_{g_e} = \xi_{g_e} + (\zeta_{g_e} - \xi_{g_e})$ is also infinitely divisible.
\end{proof}

From this result, we can deduce a formula for $\mathbb{P}(\xi_{g_e} < z)$ with $z > 0$ by inverting Laplace transforms. For $q >0$, define the function $\mathcal{V}_q$ on $\mathbb{R}_+$ as
\begin{equation}
\label{def:Vq}
 \mathcal{V}_q(z) = \mathbb{E}\left[\int_0^\infty\e^{-q\theta_t}\bm{1}_{\{H_t \leq z\}}\dr t\right]\,,
\end{equation}
where we recall that $(\theta_t, H_t)_{t\geq 0}$ are the ladder heights, defined in Section~\ref{section_WH}. Recalling from \eqref{eq_renewal} that for any $z\geq0$, $\mathcal{V}(z) = \int_0^\infty \mathbb{P}(H_t \leq z) \dr t$, we observe that $\mathcal{V}_q(z)$ increases to $\mathcal{V}(z)$ as $q\to0$.

\begin{corollary}\label{coro_exp}
 For any $z > 0$, $\mathbb{P}(\xi_{g_e} < z) = \kappa(0,q,0)\mathcal{V}_q(z)$.
\end{corollary}

\begin{proof}
By Corollary~\ref{laplace_transform_x_zeta}, we have for any $q>0$ and any $\lambda > 0$
\[
\lambda\int_0^{\infty}\e^{-\lambda z}\mathbb{P}(\xi_{g_e} < z) \dr z = \mathbb{E}\left[\e^{-\lambda \xi_{g_e}}\right] = \frac{\kappa(0,q,0)}{\kappa(0,q,\lambda)}.
\]

\noindent
By the definition of $\mathcal{V}_q$, we also have
\[
\lambda\int_0^{\infty}\e^{-\lambda z}\mathcal{V}_q(z) \dr z = \mathbb{E}\left[\int_0^{\infty}\e^{-q \theta_t - \lambda H_t}\right] \dr t = \int_0^\infty \e^{-t\kappa(0,q,\lambda)} \dr t =  \frac{1}{\kappa(0,q,\lambda)},
\]

\noindent
and the results holds by injectivity of the Laplace transform.
\end{proof}

We finally end this section with the following proposition, which allows us to decompose $\mathbb{P}(\xi_e < z)$, our quantity of interest.

\begin{proposition}\label{prop_indep}
 The random variables $\xi_{g_e}$, $\xi_{g_e} - \zeta_{g_e}$ and $I_e$ are mutually independent. Moreover, we have for any $z > 0$,
 \begin{equation}\label{eq_decomp_prob}
  \mathbb{P}(\xi_e < z) = \mathbb{P}(\xi_{g_e} < z)\mathbb{P}(\Delta_e \leq 0) + \mathbb{P}(\xi_{g_e} + \Delta_e < z, \Delta_e \in (0,z)),
 \end{equation}
\noindent
where $\Delta_e = I_e + \zeta_{g_e} -\xi_{g_e} = \zeta_e - \xi_{g_e}$.
\end{proposition}

\begin{proof}
 By Proposition \ref{indep} and since $\xi_{g_e} = \xi_{g_e-}$ and $\zeta_{g_e} = \zeta_{g_e-}$, we get that $I_e$ is independent of $(\xi_{g_e}, \xi_{g_e} - \zeta_{g_e})$. Then by Corollary \ref{laplace_transform_x_zeta}, $\xi_{g_e}$ and $\xi_{g_e} - \zeta_{g_e}$ are independent. This implies that the three random variables are mutually independent.
 
Now observe that thanks to Assumptions \ref{assump-f} and \ref{assump_exc}, $s\mapsto \zeta_s$ is monotone on the interval $[g_t, t]$ for every $t > 0$. This implies that $\sup_{s\in[g_t, t]}\zeta_s = \zeta_{g_t} \vee \zeta_t$ which in turn implies that $\xi_t = \xi_{g_t}\vee \zeta_t$. We get the identity
\begin{equation}\label{split}
 \xi_t = \xi_{g_t}\bm{1}_{\{\Delta_t \leq 0\}} + (\xi_{g_t} + \Delta_t)\bm{1}_{\{\Delta_t > 0\}}= \xi_{g_t} + \max(\Delta_t,0).
\end{equation}
\noindent
This allows to derive the desired identity.
\end{proof}

\section{The positive recurrent case: proof of Theorem~\ref{main_thm_rec_pos}} \label{sec_main_thm_rec_pos}

In this section, we focus on the case $\nn(\ell) < \infty$, which corresponds to the positive recurrent case, \textit{i.e.}\ Assumption~\ref{assump_pos}. 
In this case, we have that $\tau_t \approx t$ as $t\to\infty$ and thus we should expect $\xi_t$ to behave as $S_t$, as we shall see. We recall that the Laplace exponent $\Phi$ is expressed as $\Phi(q) = \mathrm{m}q + \nn(1-\e^{-q\ell})$. Therefore, if $\nn(\ell) < \infty$, we have $\Phi(q) \sim (\mathrm{m}+\nn(\ell))q$ as $q\to0$ which shows that $\mathbb{E}[\tau_1] = \mathrm{m} + \nn(\ell) = \mathrm{m}_1 < \infty$. Then by the strong law of large numbers for subordinators, it holds that a.s.
\begin{equation}\label{SLLN}
 \lim_{t\to\infty} \frac1t\tau_t = \mathrm{m}_1 \,.
\end{equation}
Our main goal is to prove Theorem~\ref{main_theorem} under Assumption~\ref{assump_pos}. We procede in three steps: we deal with the contribution of the unfinished excursion (this will show that $I_e$ can be neglected); we study the behavior of $\kappa(0,q,0)$ as $q\downarrow 0$; we conclude the proof by combining the above with Corollary~\ref{coro_exp} and Proposition~\eqref{prop_indep}.

\subsection{Estimate of the last excursion}
Let us first estimate $I_e$ as $q\downarrow0$. It turns out that in the positive recurrent case, it converges in law.

\begin{lemma}\label{lemma_conv_ie}
 If $\nn(\ell) < \infty$, then $I_e$ converges in law as $q\to0$ to some random variable $I_0$ which is a.s. finite and such that $\mathbb{P}(I_0 \leq 0) > 0$.
\end{lemma}

\begin{proof}
 Let $F:\mathbb{R}\to\mathbb{R}_+$ be a bounded continuous function. Using Proposition \ref{indep}, we have
\[
\mathbb{E}[F(I_e)]= \frac{q}{\Phi(q)} \bigg(\mathrm{m}F(0) + \nn\bigg( \int_0^\ell \e^{-qt}F\Big(\int_0^t f(\epsilon_r)\dr r\Big) \dr t \bigg)\bigg).
\]

\noindent
Again, $\Phi(q) \sim \mathrm{m}_1q$ as $q\to0$, and we get by the dominated convergence theorem that
\[
 \mathbb{E}[F(I_e)] \longrightarrow \frac{\mathrm{m}F(0) + \nn\Big(\int_0^\ell F\Big(\int_0^t f(\epsilon_r)\dr r\Big) \dr t\Big)}{\mathrm{m}_1} \quad \text{as }q\to0,
\]
which shows the result. Since $0$ is recurrent for $(X_t)_{t\geq0}$, for $\nn$-almost every excursions $\varepsilon$, $\varepsilon$ has a finite lifetime $\ell$, which shows that $I_0$ is a.s. finite. By Assumption~\ref{assump_2_prime}, $\nn$ charges the set of negative excursions so that we necessarily have $\mathbb{P}(I_0 \leq 0) > 0$.
\end{proof}

\subsection{Asymptotic behavior of the Laplace exponent}

Let us now study the asymptotic behavior of $\kappa(0,q,0)$ as $q\downarrow0$. Since by Corollary \ref{coro_exp} we have that $\mathbb{P}(\xi_{g_e} < z) = \kappa(0,q,0)\mathcal{V}_q(z)$ with $\mathcal{V}_q(z) \uparrow \mathcal{V}(z)$ as $q\downarrow0$, the behavior of $\kappa(0,q,0)$ characterizes the behavior of $\mathbb{P}(\xi_{g_e} < z)$.

\begin{theorem}\label{equivalence_theorem}
 Assume that $\nn(\ell) < \infty$ and let $e$ be an exponential variable with parameter $q$, independent of $(X_t)_{t\geq 0}$. Then the following assertions are equivalent for any $\rho\in(0,1)$:
\begin{enumerate}[label=(\roman*)]
 \item $\lim_{t\to\infty}\frac{1}{t}\int_0^{t}\mathbb{P}(Z_s \geq 0) \dr s = \rho$;
 
 \item $\lim_{t\to\infty}\frac{1}{t}\int_0^{t}\mathbb{P}(\zeta_{s} \geq 0) \dr s = \rho$;
 
 \item The function $q\mapsto\kappa(0,q,0)$ is regularly varying as $q\downarrow0$, with index $\rho$.
\end{enumerate}
\end{theorem}

\begin{proof}
 \textit{Step 1:} We show that \textit{(i)} is equivalent to \textit{(iii)}. By Proposition \ref{wiener_prop_corps}, we have for any $q > 0$,
\begin{align*}
 \log\kappa(0,q,0) & = \log c + \int_0^{\infty} \frac{\mathbb{E}\left[(\e^{-t} - \e^{-q\tau_t})\bm{1}_{\{Z_t \geq0\}}\right]}{t} \dr t \\
  & = \log c + \rho\int_0^{\infty}\frac{\e^{-t} - \e^{-\Phi(q)t}}{t}\dr t + \int_0^{\infty} \frac{\mathbb{E}\left[(\e^{-t} - \e^{-q\tau_t})(\bm{1}_{\{Z_t \geq0\}} - \rho)\right]}{t} \dr t \\
  & = \log c + \rho\log\Phi(q) + \int_0^{\infty} \frac{\mathbb{E}\left[(\e^{-t} - \e^{-q\tau_t})(\bm{1}_{\{Z_t \geq0\}} - \rho)\right]}{t} \dr t,
\end{align*}
where we used Frullani's identity in the third equality. Let us define the function $\varsigma$ for $q>0$ as
\begin{equation}
\label{def:varsigma}
 \varsigma(q) = \exp\left(\int_0^{\infty} \frac{\mathbb{E}\left[(\e^{-t} - \e^{-q\tau_t})(\bm{1}_{\{Z_t \geq0\}} - \rho)\right]}{t} \dr t\right).
\end{equation}

\noindent
Since $\Phi(q) \sim \mathrm{m}_1 q$ as $q\to 0$, it is clear that $q\mapsto\kappa(0,q,0)$ is regularly varying as $q\to0$ with index $\rho$ if and only if $\varsigma$ is slowly varying at $0$. We have for $\lambda > 0$ and $q > 0$
\begin{equation}
\label{eq:slowlyvar}
\begin{split}
 \log\varsigma(\lambda q) - \log\varsigma(q) = & \int_0^{\infty} \frac{\mathbb{E}\left[(\e^{-q\tau_t} - \e^{-\lambda q\tau_t})(\bm{1}_{\{Z_t \geq0\}} - \rho)\right]}{t} \dr t \\
  = &  \int_0^{\infty} \frac{\mathbb{E}\left[(\e^{-q\tau_{t/q}} - \e^{-\lambda q\tau_{t/q}})(\bm{1}_{\{Z_{t/q} \geq0\}} - \rho)\right]}{t} \dr t 
  = \mathbf{I} + \mathbf{II}\,, 
\end{split}
\end{equation}
where we have set
\[
\begin{split}
 \mathbf{I} & :=  \int_0^{\infty} \frac{\e^{-\mathrm{m}_1 t} - \e^{-\mathrm{m}_1\lambda t}}{t}\left(\mathbb{P}(Z_{t/q} \geq0)-\rho\right) \dr t \,,\\
 \mathbf{II} & := \int_0^{\infty} \frac1t  \mathbb{E}\left[(\e^{-q\tau_{t/q}} - \e^{-\mathrm{m}_1t} + \e^{-\mathrm{m}_1\lambda t} - \e^{-\lambda q\tau_{t/q}})(\bm{1}_{\{Z_{t/q} \geq0\}} - \rho)\right] \dr t \,.
 \end{split}
\]

We first show that the term $\mathbf{II}$ always converges to $0$ under the assumption $\nn(\ell) < \infty$. Let us set $F(t,q) = F_1(t,q) + F_2(t,q)$ where $F_1(t,q) = \mathbb{E}[(\e^{-q\tau_{t/q}} - \e^{-\mathrm{m}_1t})(\bm{1}_{\{Z_{t/q} \geq0\}} - \rho)]$ and $F_2(t,q) = \mathbb{E}[(\e^{-\mathrm{m}_1\lambda t} - \e^{-\lambda q\tau_{t/q}})(\bm{1}_{\{Z_{t/q} \geq0\}} - \rho)]$. By \eqref{SLLN}, we have that for any $t \geq0$, $q\tau_{t/q}\to \mathrm{m}_1t$ a.s.\ as $q\to0$. Therefore by the dominated convergence theorem, we have for any $t\geq0$, $F(t,q) \to 0$ as $q\to0$, and it only remains to dominate $F(t,q)$ to get that $\mathbf{II} = \int_0^{\infty} \frac 1t F(t,q) \dd t$ goes to zero as $q\downarrow 0$. We only show that we can dominate $F_1$; the same domination can be done on $F_2$. We write
\begin{align*}
 |F_1(t,q)| \leq 2\mathbb{E}\left[|\e^{-q\tau_{t/q}} - \e^{-\mathrm{m}_1t}|\right] & \leq 2\mathbb{E}\left[|\e^{-q\tau_{t/q}} - \e^{-\mathrm{m}_1t}|^2\right]^{1/2} \\
  & \leq 2 \left(\e^{-2t\Phi(2q)/ 2q} - 2\e^{-(\mathrm{m}_1 + \Phi(q)/q)t} + \e^{-2\mathrm{m}_1t} \right)^{1/2}.
\end{align*}
We set $b_+ = \sup_{q\in(0,1)}\Phi(q) / q$ and $b_- = \inf_{q\in(0,1)}\Phi(q) / q$ which are two finite and strictly positive real numbers. Then it is clear that for any $q\in(0,\frac12)$, for any $t > 0$,
\[
 \frac{1}{t} |F_1(t,q)| \leq \frac{2}{t} \big(\e^{-2tb_-} - 2\e^{-(\mathrm{m}_1 + b_+) t} + \e^{-2\mathrm{m}_1t} \big)^{1/2}.
\]
The term on the right-hand side is integrable: indeed, 
for $t\in [0,1]$ it is bounded by a constant times $t^{-1/2}$, whereas for $t>1$ it is bounded by $2 (\e^{-2tb_-}+ \e^{-2\mathrm{m}_1t})^{1/2}$.
We can conclude by the dominated convergence theorem that $\mathbf{II} = \int_0^{\infty} \frac 1t F(t,q) \dd t$ goes to zero as $q\downarrow 0$.

\smallskip
At this point, we know that \textit{(iii)} holds if and only if for any $\lambda > 0$, the term $\mathbf{I}$ goes to $0$ as $q\downarrow 0$.
We now use well-known results about Lévy processes. The usual Laplace exponent of the ladder time process of $(Z_t)_{t\geq0}$ is $q\mapsto\kappa(q,0,0)$ and it is well-known, see for instance \cite[Prop.~VI.18 and its proof]{b} that \textit{(i)} holds if and only if $q\mapsto\kappa(q,0,0)$ is regularly varying at $0$ with index~$\rho$. Recalling Proposition~\ref{wiener_prop_corps}, we have for any $q>0$,
\begin{align*}
 \log\kappa(q,0,0) = & \log c + \int_0^\infty\frac{\e^{-t} - \e^{-qt}}{t}\mathbb{P}(Z_{t} \geq0)\dr t \\
  & = \log c + \rho\log q + \int_0^\infty\frac{\e^{-t} - \e^{-qt}}{t}\left(\mathbb{P}(Z_{t} \geq0) - \rho\right)\dr t.
\end{align*}
We therefore obtain that \text{(i)} holds if and only if, for any $\lambda >0$
\[
\log\kappa(\lambda q,0,0)- \log \kappa(q,0,0) - \rho\log \lambda
= \int_0^{\infty} \frac{\e^{-q\lambda t} - \e^{-qt}}{t}\left(\mathbb{P}(Z_{t} \geq0) - \rho\right)\dr t 
 \longrightarrow 0
\]
as $q\downarrow 0$.
But this is clearly equivalent to having the term $\mathbf{I}$ going to $0$ as $q\downarrow 0$, for any $\lambda >0$.
This concludes the proof  that \text{(i)} holds if and only if \text{(iii)} holds.

\medskip\noindent
\textit{Step 2:} Next, we show that \textit{(i)} holds if and only if we have \textit{(ii')}
$\lim_{t\to\infty}\frac{1}{t}\int_0^{t}\mathbb{P}(\zeta_{g_s} \geq 0) \dr s = \rho$.
Let $e=e(q)$ be an independent exponential random variable of parameter $q>0$ and $e'=e(\mathrm{m}_1q)$ be an independent exponential random variable of parameter $\mathrm{m}_1q>0$. We have 
\[
 \mathbb{P}(\zeta_{g_e} \geq 0) = q\int_0^{\infty}\e^{-qt}\mathbb{P}(\zeta_{g_t} \geq 0)\dr t \quad \text{and} \quad \mathbb{P}(Z_{e'} \geq 0) = \mathrm{m}_1q\int_0^{\infty}\e^{-\mathrm{m_1} qt}\mathbb{P}(Z_t\geq 0)\dr t.
\]
Then by the Tauberian theorem, see \cite[Thm.~1.7.1]{bgt89}, we have that $\lim_{q\downarrow 0}\mathbb{P}(Z_{e'} \geq 0)= \rho$ if and only if \textit{(i)} holds, and $\lim_{q\downarrow 0}\mathbb{P}(\zeta_{g_e} \geq 0) =\rho$ if and only if $\lim_{t\to\infty}\frac{1}{t}\int_0^{t}\mathbb{P}(\zeta_{g_s} \geq 0) \dr s = \rho$. Therefore, to conclude, it only remains to show that $\lim_{q\to 0} |\mathbb{P}(\zeta_{g_e} \geq 0) - \mathbb{P}(Z_{e'}\geq 0)|= 0$.

Applying Proposition~\ref{indep}-\eqref{eq:F1} with the functional $F_1 = \bm{1}_{\{\zeta_{g_t} \geq0\}}$, we get
\[
 \mathbb{P}(\zeta_{g_e} \geq 0) = \Phi(q)\int_0^{\infty}\mathbb{E}\left[\e^{-q\tau_t}\bm{1}_{\{Z_t \geq0\}}\right] \dr t.
\]
Then we obtain that $|\mathbb{P}(\zeta_{g_e} \geq 0) - \mathbb{P}(Z_{e'}\geq 0)|$ is bounded by
\begin{align*}
 \bigg|\int_0^{\infty} &\bbE\left[ (\Phi(q)\e^{-q \tau_t} -\mathrm{m}_1 q\e^{-\mathrm{m}_1 qt} ) \bm{1}_{\{Z_t\geq 0\}}\right] \dd t \bigg| \\
& \leq  \left|\Phi(q) -\mathrm{m}_1q\right|\int_0^{\infty}\mathbb{E}\left[\e^{-q\tau_t}\bm{1}_{\{Z_t \geq0\}}\right] \dr t  +  \mathrm{m}_1 \int_0^{\infty}\mathbb{E}\left[ | q \e^{-q\tau_t} -q\e^{-qt}|\bm{1}_{\{Z_t \geq0\}}\right] \dr t \\
  &  \leq \left|\Phi(q) -\mathrm{m}_1q\right| \frac{1}{\Phi(q)}
   + \mathrm{m}_1\int_0^{\infty}\mathbb{E}\left[| \e^{-q\tau_{t/q}} - \e^{-\mathrm{m}_1t}|\right] \dr t.
\end{align*}

\noindent
By assumption, we have $\lim_{q\to0} \Phi(q)/q =\mathrm{m}_1$, so the first term goes to $0$.
By the law of large numbers \eqref{SLLN} and dominated convergence, we have that for any $t\geq0$, $\mathbb{E}[|\e^{-q\tau_{t/q}} - \e^{-\mathrm{m}_1t}|]$ converges to $0$ as $q\downarrow 0$.
We conclude again by dominated convergence that the second term goes to~$0$, since $\mathbb{E}[|\e^{-q\tau_{t/q}} - \e^{-\mathrm{m}_1t}|] \leq \mathbb{E}[\e^{-q\tau_{t/q}}] + \e^{-\mathrm{m}_1t} = \e^{-t \Phi(q)/q} +  \e^{-\mathrm{m}_1t} \leq \e^{- b_- t} +  \e^{-\mathrm{m}_1t}$.

\medskip\noindent
\textit{Step 3:} Finally, we show that \textit{(ii)} is equivalent to \textit{(ii')} $\lim_{t\to\infty}\frac{1}{t}\int_0^{t}\mathbb{P}(\zeta_{g_s} \geq 0) \dr s = \rho$.
Again, by the Tauberian theorem, \textit{(ii)} holds if and only if $\lim_{q\downarrow 0}\mathbb{P}(\zeta_e \geq0) = \rho$ and \textit{(ii')} holds if and only if $\lim_{q\to 0}\mathbb{P}(\zeta_{g_e} \geq0) = \rho$ (with $e=e(q)$ an independent exponential random variable of parameter $q>0$).
Therefore it is enough to show that $\lim_{q\to0} |\mathbb{P}(\zeta_{g_e} \geq0) - \mathbb{P}(\zeta_{e} \geq0)| = 0$ in the case $\nn(\ell) < \infty$.
To prove this, we write
 \[
  \mathbb{P}(\zeta_{g_e} \geq0) - \mathbb{P}(\zeta_{e} \geq0) = \mathbb{P}(\zeta_{g_e} \geq0, \zeta_e < 0) - \mathbb{P}(\zeta_{g_e} <0, \zeta_e \geq 0).
 \]

\noindent
We will only show that $\lim_{q\to 0}\mathbb{P}(\zeta_{g_e} \geq0, \zeta_e < 0) = 0$ as the proof for the other term is similar.
By Lemma~\ref{lemma_conv_ie} above, the term $I_e = \zeta_e - \zeta_{g_e}$ converges in law as $q\to0$ to some real random variable~$I_0$.
We let $A > 0$ and we decompose the probability as
\[
 \mathbb{P}(\zeta_{g_e} \geq0, \zeta_e < 0) = \mathbb{P}(\zeta_{g_e} \geq0, \zeta_{g_e} + I_e < 0, I_e < - A) + \mathbb{P}(\zeta_{g_e} \geq0, \zeta_{g_e} + I_e < 0, I_e \geq - A),
\]
which yields the following inequality:
\[
 \mathbb{P}(\zeta_{g_e} \geq0, \zeta_e < 0) \leq \mathbb{P}(I_e < -A) + \mathbb{P}(\zeta_{g_e} \in[0,A]).
\]

\noindent
By Proposition \ref{indep}, we have
\[
 \mathbb{P}(\zeta_{g_e} \in [0, A]) = \Phi(q)\int_0^{\infty}\mathbb{E}\left[\e^{-q\tau_t}\bm{1}_{\{Z_t \in[0, A]\}}\right] \dr t = \frac{\Phi(q)}{q}\int_0^{\infty}\mathbb{E}\left[\e^{-q\tau_{t/q}}\bm{1}_{\{Z_{t/q} \in[0, A]\}}\right]\dr t.
\]

\noindent
For every $A>0$, we have $\lim_{t\to\infty}\mathbb{P}(Z_t\in[0, A]) = 0$, see e.g.\ Sato \cite[Ch.~9 Lem.~48.3]{ken1999levy}.
Since $\lim_{q\to 0}\Phi(q)/q = \mathrm{m}_1$  and $\mathbb{E}[\e^{-q\tau_{t/q}}\bm{1}_{\{Z_{t/q} \in[0, A]\}}] \leq \mathbb{E}[\e^{-q\tau_{t/q}}] = \e^{-t\Phi(q) / q} \leq \e^{- b_-t}$, it follows from the dominated convergence theorem that $\lim_{q\to 0} \mathbb{P}(\zeta_{g_e} \in [0, A)) = 0$.
We deduce that
\[
 \limsup_{q\to0}\mathbb{P}(\zeta_{g_e} \geq0, \zeta_e < 0) \leq \limsup_{q\to0}\mathbb{P}(I_e < -A) \leq \mathbb{P}(I_0 \leq A).
\]

\noindent
Since $I_0$ is a.s. finite, the result follows by letting $A\to\infty$.
\end{proof}

\subsection{Conclusion of the proof of Theorem \ref{main_thm_rec_pos} under Assumption \ref{assump_pos}}
We are finally able to prove our main theorem in the positive recurrent case.
\begin{proof}
We assume that $\nn(\ell) < \infty$, and will first show that, by setting $c_0 = \liminf_{q\to0}\bbP(\Delta_e \leq 0)$, we have $c_0\in(0,1]$ and for any $z > 0$,
\begin{equation}\label{encadre_recu_pos}
 c_0\mathcal{V}(z) \leq \liminf_{q\to0}\frac{\mathbb{P}(\xi_e < z)}{\kappa(0,q,0)} \leq \limsup_{q\to0}\frac{\mathbb{P}(\xi_e < z)}{\kappa(0,q,0)} \leq \mathcal{V}(z).
\end{equation}

\noindent
Then we will see that if $q\mapsto \mathbb{P}(\xi_e < z)$ is regularly varying at $0$ with index $\rho \in(0,1)$, or if $t^{-1}\int_0^t \mathbb{P}(\zeta_s \geq0) \dr s \to \rho$ as $t\to\infty$, then necessarily $c_0 = 1$. Using \eqref{eq_decomp_prob}, we deduce the following inequalies
\[
 \mathbb{P}(\xi_{g_e} < z) \mathbb{P}(\Delta_e \leq 0) \leq \mathbb{P}(\xi_e < z) \leq \mathbb{P}(\xi_{g_e} < z).
\]
By Corollary \ref{coro_exp}, we have $\mathbb{P}(\xi_{g_e} < z) = \kappa(0,q,0) \mathcal{V}_q(z)$ for any $z > 0$, and since $\mathcal{V}_q(z) \to \mathcal{V}(z)$, \eqref{encadre_recu_pos} holds if we can show that $\mathbb{P}(\Delta_e \leq 0)$ converges to some positive constant as $q\to0$.

\smallskip
Remember that $I_e$ and $\zeta_{g_e} - \xi_{g_e}$ are independent, that $\Delta_e = I_e + \zeta_{g_e} - \xi_{g_e}$, and that by Lemma~\ref{lemma_conv_ie}, $I_e$ converges in law as $q\to0$.
We will show that, either $\zeta_{g_e} - \xi_{g_e}$ converges to $-\infty$ as $q\to0$ in probability, or $\zeta_{g_e} - \xi_{g_e}$ converges in law to some (finite) random variable as $q\to0$. By Proposition \ref{indep}, we have for any $\mu > 0$,
\[
 \mathbb{E}\left[\e^{-\mu(\xi_{g_e} - \zeta_{g_e})}\right] = \Phi(q)\int_0^{\infty}\mathbb{E}\left[\e^{-q\tau_t -\mu(S_t - Z_t)}\right]\dr t = \frac{\Phi(q)}{q}\int_0^{\infty}\mathbb{E}\left[\e^{-q\tau_{t/q} -\mu(S_{t/q} - Z_{t/q})}\right]\dr t.
\]

\noindent
Recall that $\Phi(q)/q \to \mathrm{m}_1$ as $q\to0$ and that $q\tau_{t/q} \to \mathrm{m}_1 t$ as $q\to0$. Duality entails that for every $t\geq0$, $S_t - Z_t$ is equal in law to $\hat{S}_t$ where $\hat{S}_t = \sup_{s\in[0,t]}\hat{Z}_s$ and $\hat{Z}_t = - Z_t$, see \cite[Ch.~VI Prop.~3]{b}. Since $(\hat{S}_t)_{t\geq0}$ is increasing, $\hat{S}_t \to \hat{S}_{\infty}$ as $t\to\infty$. By the $0$-$1$ law, $\mathbb{P}(\liminf_{t\to\infty} Z_t = \infty)$ is equal to $0$ or $1$, two cases are to be considered:
\begin{itemize}
 \item If $\liminf_{t\to\infty} Z_t = \infty$ a.s., then $\hat{S}_{\infty} = \infty$ a.s. and thus, $S_t - Z_t$ converges in probability to $\infty$ as $t\to\infty$. Since $\mathbb{E}[\e^{-q\tau_{t/q} -\mu(S_{t/q} - Z_{t/q})}] \leq \mathbb{E}[\e^{-q\tau_{t/q}}] = \e^{-t\Phi(q) / q}$, we can apply the dominated convergence theorem which shows that $\xi_{g_e} - \zeta_{g_e}$ converges in probability to $\infty$ as $q\to0$. It comes that $\mathbb{P}(\Delta_e \leq 0) \to 1$ as $q\to0$.
 
 \item If $\liminf_{t\to\infty} Z_t < \infty$ a.s., then $\hat{S}_{\infty} <\infty$ a.s. and thus $S_t - Z_t$ converges in law to $\hat{S}_{\infty}$ as $t\to\infty$. By Slutsky's lemma, for any $t\geq0$, $(q\tau_{t/q}, S_{t/q} - Z_{t/q})$ converges in law to $(\mathrm{m}_1 t, \hat{S}_{\infty})$ and we can again apply the dominated convergence theorem, which shows that $\xi_{g_e} - \zeta_{g_e}$ converges in law to some non-negative finite random variable that we name $\xi_0 - \zeta_0$. Then we have $\mathbb{P}(\Delta_e \leq 0) \to \mathbb{P}(\zeta_0 - \xi_0 + I_0 \leq 0)$ where $\zeta_0 - \xi_0$ and $I_0$ are indepedent. By Lemma \ref{lemma_conv_ie}, $\mathbb{P}(I_0 \leq 0) > 0$ which shows that $\mathbb{P}(\zeta_0 - \xi_0 + I_0 \leq 0) > 0$.
\end{itemize}

\noindent
We showed that in any case \eqref{encadre_recu_pos} holds. Let us now assume first that $t^{-1}\int_0^t \mathbb{P}(\zeta_s \geq0) \dr s \to \rho$ as $t\to\infty$ for some $\rho\in(0,1)$. Then by Theorem \ref{equivalence_theorem}, it also holds that $t^{-1}\int_0^t \mathbb{P}(Z_s \geq0) \dr s \to \rho$ as $t\to\infty$, which implies that $\liminf_{t\to\infty} Z_t = \infty$ a.s., see for instance \cite[Theorem VI.12]{b}, and therefore the constant $c_0$ in \eqref{encadre_recu_pos} is equal to $1$. We then have $\mathbb{P}(\xi_e < z)\sim \kappa(0,q,0)\mathcal{V}(z)$ which shows by Theorem \ref{equivalence_theorem} that $q\mapsto\mathbb{P}(\xi_e < z)$ is regularly varying with index $\rho$.

Let us now assume that $q\mapsto\mathbb{P}(\xi_e < z)$ is regularly varying at $0$ with index $\rho\in(0,1)$. Then by \eqref{encadre_recu_pos}, for any $\delta > 0$, $q^{\rho + \delta} / \kappa(0,q,0) \to 0$ as $q\to0$. Now remember from Proposition \ref{wiener_prop_corps} that
\[
 \kappa(0,q,0) = c\exp\bigg(\int_0^{\infty}\frac{\mathbb{E}[(\e^{-t} - \e^{-q\tau_t})\bm{1}_{\{Z_t \geq0\}}]}{t}\dr t\bigg)
\]

\noindent
and 
\[
 \bar{\kappa}(0,q,0) = \exp\bigg(\int_0^{\infty}\frac{\mathbb{E}[(\e^{-t} - \e^{-q\tau_t})\bm{1}_{\{Z_t < 0\}}]}{t}\dr t\bigg).
\]

\noindent
We see by Frullani's identity that $\kappa(0,q,0)\bar{\kappa}(0,q,0) = c \Phi(q)$ and since $\Phi(q) \sim \mathrm{m}_1 q$ as $q\to0$ it comes that $\bar{\kappa}(0,q,0) \to 0$ as $q\to0$. Recall that by Corollary \ref{laplace_transform_x_zeta}, for any $\mu > 0$, we have $\mathbb{E}[\e^{-\mu(\xi_{g_e} - \zeta_{g_e})}] = \bar{\kappa}(0,q,0) / \bar{\kappa}(0,q,\mu)$, and since $\bar{\kappa}(0,q, \mu) \to \bar{\kappa}(0, 0, \mu) > 0$ as $q\to0$, we see that $\mathbb{E}[\e^{-\mu(\xi_{g_e} - \zeta_{g_e})}] \to 0$ as $q\to0$ which shows that, in this case $\xi_{g_e} - \zeta_{g_e}$ converges in probability to $\infty$. From the previous analysis, we see that we are necessarily in the case $\liminf_{t\to\infty} Z_t = \infty$ a.s. and therefore the constant $c_0$ from \eqref{encadre_recu_pos} is equal to $1$ so that $\mathbb{P}(\xi_e < z) \sim \kappa(0,q,0)\mathcal{V}(z)$ as $q\to0$. This shows that $q\mapsto\kappa(0,q,0)$ is also regularly varying at $0$ with index $\rho$ and thus, by Theorem \ref{equivalence_theorem}, $t^{-1}\int_0^t \mathbb{P}(\zeta_s \geq0) \dr s \to \rho$ as $t\to\infty$.
\end{proof}

\section{The null recurrent case: proof of Theorem~\ref{main_theorem}}
\label{sec:proof_null}

In this section, we assume that $0$ is null recurrent, i.e. that $\nn(\ell) = \infty$ and suppose that Assumption~\ref{assump_null} holds:
in a nutshell,
\[
 \left(b(h)\tau_{t/h}, a(h)Z_{t/h}\right)_{t\geq 0} \longrightarrow \left(\tau_{t}^0, Z_{t}^0\right)_{t\geq 0} \quad \text{ as } h\to0,
\]
with $(\tau_{t}^0, Z_{t}^0)_{t\geq 0}$ a bivariate Lévy process, $(\tau_t^0)_{t\geq0}$ being a $\beta$-stable subordinator and $(Z_t^0)_{t\geq0}$ an $\alpha$-stable process.
Assumption~\ref{assump_null} also entails that  $a(b^{-1}(h))\, I_{1/h}$ converges in law as $q\to0$ to some (possibly degenerate) random variable $I$.

\begin{remark}
\label{rem:bPhi}
Let us stress that since $(b(h)\tau_{t/h})_{t\geq0}$ converges in law to $(\tau_t^0)_{t\geq0}$ as $h\to0$, the Laplace exponent $\Phi(q)$ of $(\tau_t)_{t\geq0}$ is regularly varying at $0$ with index $\beta$. More precisely, $\Phi(\cdot)$ is an asymptotic inverse (up to a constant) of $b(\cdot)$ near $0$. Indeed, $(b(q)\tau_{t/q})$ converges in law to a $\beta$-stable subordinator, and therefore
\[
 -\frac{1}{t}\log\mathbb{E}\left[\e^{-b(q)\tau_{t/q}}\right] = \frac{\Phi(b(q))}{q} \xrightarrow[]{q\to 0} -\frac{1}{t}\log\mathbb{E}\left[\e^{-\tau_{t}^0}\right].
\]
This also implies, see \cite[Thm.~1.5.12]{bgt89}, that $q/b(\Phi(q))$ converges to some some positive constant $\bar{b}$ as $q\to0$.
For simplicity and without loss of generality, we assume in the following that $\bar{b}=1$
Since $(b(\Phi(q))\tau_{t/\Phi(q)}, a(\Phi(q))Z_{t/\Phi(q)})_{t\geq0}$ converges in law to $(\tau_t^0, Z_t^0)_{t\geq0}$ as $q\to0$, it comes that
\begin{equation}\label{conv_truqué}
 (q\tau_{t/\Phi(q)}, a(\Phi(q))Z_{t/\Phi(q)})_{t\geq0} \longrightarrow (\tau_t^0, Z_t^0)_{t\geq0} \quad \text{as }q\to0.
\end{equation}
\end{remark}

Below, we use some convergence results for Lévy processes, whose proofs are collected in Appendix~\ref{appendix_levy}.

\subsection{Laplace exponent and convergence of scaled $\xi_{g_e}-\zeta_{g_e}$}
We start with the following result.
\begin{proposition}
 Suppose that Assumption \ref{assump_null} holds. Then the function $q\mapsto \kappa(0,q,0)$ is regularly varying as $q\to0$ with index $\beta \rho$ with  $\rho:=\mathbb{P}(Z_t^0 \geq0)$.
\end{proposition}

\begin{proof}\label{prop_rec_null_kappa}
Recalling Proposition~\ref{wiener_prop_corps}, we start by writing that
\begin{align*}
 &\log\kappa(0,q,0) =  \log c + \int_0^\infty\frac{\mathbb{E}\left[(\e^{-t} - \e^{-q\tau_t})\bm{1}_{\{Z_t \geq0\}}\right]}{t} \dr t \\
&\  =  \log c + \int_0^{\infty}\frac{\e^{-t/\Phi(q)} - \e^{-t}}{t}\mathbb{P}(Z_{t/\Phi(q)} \geq0)\dr t  + \int_0^\infty\frac{\mathbb{E}\big[(\e^{-t} - \e^{-q\tau_{t/\Phi(q)}})\bm{1}_{\{a(\Phi(q))Z_{t/\Phi(q)} \geq0\}}\big]}{t} \dr t.
\end{align*}
Using the convergence \eqref{conv_truqué} and Proposition \ref{conv_laplace_exponent} in Appendix~\ref{appendix_levy}, the last term converges as $q\to0$. Then, with $\rho = \mathbb{P}(Z_t^0 \geq0)$, the second term becomes, by Frullani's identity
\[
 \rho\log \Phi(q) + \int_0^{\infty}\frac{\e^{-t} - \e^{-\Phi(q)t}}{t}(\mathbb{P}(Z_{t} \geq0) -\rho)\dr t.
\]

\noindent
As we argued in the proof of Theorem \ref{equivalence_theorem}, the function
\[
 q\mapsto\exp\left(\int_0^{\infty}\frac{\e^{-t} - \e^{-\Phi(q)t}}{t}(\mathbb{P}(Z_{t} \geq0) -\rho)\dr t\right)
\]
is slowly varying as $q\to0$. Since $\Phi$ is regularly varying with index $\beta$, it comes that $q\mapsto\kappa(0,q,0)$ is regularly varying with index $\beta\rho$, which establishes the results.
\end{proof}

\begin{proposition}
\label{prop_conv_xizeta}
Suppose that Assumption~\ref{assump_null} holds and let $e$ be an exponential variable with parameter $q > 0$, independent of $(X_t)_{t\geq0}$. Then $a(\Phi(q))(\xi_{g_e} - \zeta_{g_e})$ converges in law as $q\to0$: more precisely, we have that for any $\lambda>0$
\[
\lim_{q\to 0} \mathbb{E}\left[\e^{-\lambda a(\Phi(q))(\xi_{g_e} - \zeta_{g_e})}\right] = \frac{\bar{\kappa}^0(0,1,0)}{\bar{\kappa}^0(0,1,\lambda)} \,,
\]
where $\bar{\kappa}^0$ is defined, analogously to $\bar \kappa$, as
\[
 \bar{\kappa}^0(\alpha, \beta, \gamma) = \exp\bigg(\int_0^{\infty}\int_{[0,\infty)\times\mathbb{R}}\frac{\e^{-t} - \e^{-\alpha t- \beta r + \gamma x}}{t}\bm{1}_{\{x<0\}}\mathbb{P}(\tau_t^0\in \dr r, Z_t^0\in \dr x)\dr t\bigg).
\]
\end{proposition}

\begin{proof}
By Corollary \ref{laplace_transform_x_zeta}, we have that the Laplace transform of $a(\Phi(q))(\xi_{g_e} - \zeta_{g_e})$ is
\[
  \mathbb{E}\left[\e^{-\lambda a(\Phi(q))(\xi_{g_e} - \zeta_{g_e})}\right] = \frac{\bar{\kappa}(0, q, 0)}{\bar{\kappa}(0, q, \lambda a(\Phi(q)) )}.
\]

\noindent
Now, by the definition of $\bar{\kappa}$ (in Proposition~\ref{wiener_prop_corps}), we  have
\[
 \log\bar{\kappa}(0,q,0) - \log \bar{\kappa}(0,q,\lambda a(\Phi(q))) =  \int_0^{\infty}\frac{\mathbb{E}\left[(\e^{-q\tau_t + \lambda a(\Phi(q)) Z_t} - \e^{-q\tau_t})\bm{1}_{\{Z_t < 0\}}\right]}{t} \dr t.
\]
Making the change of variables $t = u / \Phi(q)$ and splitting the integral in two parts, we get that the above quantity is equal to
\begin{multline*}
 \int_0^{\infty}t^{-1}\mathbb{E}\Big[(\e^{-t} - \e^{-q\tau_{t/\Phi(q)}})\bm{1}_{\{Z_{t/\Phi(q)} < 0\}}\Big] \dr t \\
   - \int_0^{\infty}t^{-1}\mathbb{E}\left[(\e^{-t} - \e^{-q\tau_{t/\Phi(q)} + \lambda a(\Phi(q)) Z_{t/\Phi(q)}})\bm{1}_{\{Z_{t/\Phi(q)} < 0\}}\right] \dr t.
\end{multline*}
By \eqref{conv_truqué} and Proposition \ref{conv_laplace_exponent}, $\log\bar{\kappa}(0,q,0) - \log \bar{\kappa}(0,q,\lambda a(\Phi(q)))$ is a difference of two converging terms. More precisely, Proposition \ref{conv_laplace_exponent} entails that it converges to $\log\bar{\kappa}^0(0,1,0)- \log\bar{\kappa}^0(0,1,\lambda)$, which completes the proof.
\end{proof}

\subsection{Conclusion of the proof of Theorem~\ref{main_theorem} under Assumption~\ref{assump_null}}

Finally, we are able to prove our main theorem in the null recurrent case.

\begin{proof}
 We start by recalling that, from \eqref{eq_decomp_prob}, for any $z >0$ we have
 \[
  \mathbb{P}(\xi_e < z) = \mathbb{P}(\xi_{g_e} < z)\mathbb{P}(\Delta_e \leq 0) + \mathbb{P}(\xi_{g_e} + \Delta_e < z, \Delta_e \in (0,z))\,.
 \]
We also remind that by Proposition \ref{prop_indep}, the random variable $I_e$, $\xi_{g_e}$ and $\xi_{g_e} - \zeta_{g_e}$ are mutually independent. Since $\Delta_e = I_e + \zeta_{g_e} - \xi_{g_e}$, it is independent from $\xi_{g_e}$ and we get that
\[
 \mathbb{P}(\xi_{g_e} < z)\mathbb{P}(\Delta_e \leq 0) \leq \mathbb{P}(\xi_e < z) \leq \mathbb{P}(\xi_{g_e} < z)\mathbb{P}(\Delta_e < z).
\]

\noindent
By Corollary \ref{coro_exp}, we have $\mathbb{P}(\xi_{g_e} < z) = \kappa(0,q,0) \mathcal{V}_q(z)$ and since $\kappa(0,q,0)$ is regularly varying at $0$ with index $\beta\rho$ and $\mathcal{V}_q(z)$ increases to $\mathcal{V}(z)$, it only remains to control $\mathbb{P}(\Delta_e \leq 0)$ and $\mathbb{P}(\Delta_e < z)$. By Assumption~\ref{assump_null}, and recalling that $\Phi(\cdot)$ is an asymptotic inverse of $b(\cdot)$, we have that $a(\Phi(h))I_{1/h}$ converges in law as $h\to0$ to some random variable $I$.
This easily implies that $a(\Phi(q))I_{e}$ converges in law as $q\downarrow 0$ and thanks to Proposition~\ref{prop_conv_xizeta} so does $a(\Phi(q))(\zeta_{g_e} - \xi_{g_e})$.
Since $I_e$ and $\zeta_{g_e} - \xi_{g_e}$ are independent, it gives that $a(\Phi(q))\Delta_e$ converges in law as $q\to0$ to some random variable, that we denote $\Delta_0$. Finally, we end up with
\[
 \lim_{q\downarrow 0}\mathbb{P}(\Delta_e \leq 0) = \lim_{q\downarrow0} \mathbb{P}(\Delta_e < z) = \mathbb{P}(\Delta_0 \leq 0) \,,
\]
which completes the proof.
\end{proof}

\subsection{The case of Gaussian fluctuations}

In this section, we treat the special case where  Assumption \ref{assump_null} is satisfied with $\alpha = 2$, \textit{i.e.}\ when the limiting process $(Z_t^0)_{t\geq0}$ is a Brownian motion. This case is somehow simpler since $(\tau_t^0)_{t\geq0}$ and $(Z_t^0)_{t\geq0}$ are then necessarily independent (this can be seen directly from the Lévy-Khintchine formula), and the convergence of $(b(h)\tau_{t/h})_{t\geq0}$ and $(a(h)Z_{t/h})_{t\geq0}$ alone implies the convergence of the bi-dimensional process, see Lemma~\ref{indep_limite}.
We are going to prove that in that case, we can somehow weaken Assumption~\ref{assump_null} (relaxing the assumption on the convergence of the last part $a(b^{-1}(h))\,I_{1/h}$).

Recall that for a generic excursion $\varepsilon$, we denote $\f = \f(\varepsilon) = \int_0^\ell f(\varepsilon_s)\dr s$, see~\eqref{def:fepsilon}.
\begin{assumption}
\label{assum_gauss}
There exists some $\beta\in(0,1)$ such that $\Phi$ is regularly varying at $0$ with index $\beta$, and $\nn(\f)=0$, $\nn(\f^2) < \infty$.
\end{assumption}

\noindent
We then have the following result.
\begin{theorem}
Suppose that Assumption \ref{assum_gauss} holds. Then there exists a slowly varying function $\varsigma$ such that for any $z > 0$,
 \[
  \mathbb{P}(T_z > t) = \mathcal{V}(z)\varsigma(t)t^{-\beta/2} \quad \text{as }t\to\infty,
\]
where $\beta$ is given by Assumption \ref{assum_gauss}.
\end{theorem}

\begin{proof}
We simply check that Assumption~\ref{assum_gauss} implies that Assumption \ref{assump_null} is satisfied.
Let us denote by $\nu(\dr z) = \nn(\f \in \dr z)$ the Lévy measure of $(Z_t)_{t\geq0}$. Then Assumption \ref{assum_gauss} entails that $\int_{\mathbb{R}\setminus\{0\}}z^2\nu(\dr z) < \infty$, which implies that $(h^{1/2}Z_{t/h})_{t\geq0}$ converges in law towards a Brownian motion $(Z_t^0)_{t\geq0}$ as $h\to0$ (this can be easily checked via the characteristic function).
 
Moreover, if $b$ is an asymptotic inverse of $\Phi$ at $0$, we have that $(b(h)\tau_{t/h})_{t\geq0}$ converges in law to a $\beta$-stable subordinator $(\tau_t^0)_{t\geq0}$ as $h\to0$.
 By Lemma \ref{indep_limite}, this shows that the bivariate process $(b(h)\tau_{t/h}, h^{1/2}Z_{t/h})_{t\geq0}$ converges in law to $(\tau_t^0, Z_t^{0})_{t\geq0}$ as $h\to0$.
 
Finally, we show that $\Phi(q)^{1/2}I_e$ converges to $0$ in $\mathrm{L}^2$ as $q\to0$. By Proposition \ref{indep}-\eqref{eq:F2}, 
setting $\f_u(\varepsilon)=\int_0^u f(\epsilon_v) \dr v$ for any $u\leq \ell(\varepsilon)$,
we have
 \[
\mathbb{E}\left[\Phi(q)I_e^2\right] = q\int_{\mathcal{E}}\int_0^{\ell(\epsilon)}\e^{-qu}\f_u(\varepsilon)^2 \dr u \, \nn(\dr\epsilon) \leq \int_{\mathcal{E}}\f(\varepsilon)^2(1 - \e^{-q\ell(\epsilon)}) \nn(\dr\epsilon) \,.
\]
Indeed Assumptions \ref{assump-f} and \ref{assump_2_prime} entail that the map $u \mapsto \f_u(\varepsilon)$ is monotonic on $[0,\ell(\epsilon)]$ and thus $\f_u(\varepsilon)^2 \leq \f(\varepsilon)^2$.
Since $\nn(\f^2)<+\infty$, we can use dominated convergence to deduce that $\lim_{q\to0}\mathbb{E}[\Phi(q)I_e^2]=0$. We can then apply Theorem \ref{main_theorem} with $\rho = \mathbb{P}(Z_t^0 \geq0) = 1/2$.
\end{proof}

\section{Application to generalized diffusions}\label{section_ito_mc_kean}

In this section we apply our result to a large class of one-dimensional Markov processes called \textit{generalized diffusions}. These processes are defined as a time and space changed Brownian motion. Originally, it was noted by Itô and McKean \cite{im63, im96} that regular diffusions, \textit{i.e.}\ regular strong Markov processes with continuous paths can be represented as a time and space changed Brownian motion through a \emph{scale function} $\sca$ and a \emph{speed measure} $\mm$. In the mean time, Stone \cite{s63} also observed that continuous-in-time birth and death processes could be represented this way. As we shall see, this construction leads to a general class of Markov processes. Our main goal is to provide conditions on the function $f$, the \emph{scale function} $\sca$ and \emph{speed measure} $\mm$ that ensure that Assumption~\ref{assump_null} holds. Let us now recall the notation of Section~\ref{sec:resII}.

\smallskip
Let $\mm: \bbR \to \bbR$ be a non-decreasing right-continuous function such that $\mm(0) = 0$, and $\sca: \bbR \to \bbR$ a continuous increasing function such that $\sca(0)=0$ and $\sca(\mathbb{R}) = \bbR$.
We assume moreover that $\mm$ is not constant and will also denote by $\mm$ the Radon measure associated to $\mm$, that is $\mm( (a,b] )= \mm(b)-\mm(a)$ for all $a<b$.
Recall that $\mm^\sca$ is the image of $\mm$ by $\sca$, \textit{i.e.}\ the Stieltjes measure associated to the non-decreasing function $\mm\circ\sca^{-1}$.
We consider a Brownian motion $(B_t)_{t\geq0}$ on some filtered probability space $(\Omega, \mathcal{F}, (\mathcal{F}_t)_{t\geq0}, \mathbb{P})$ with $(L_t^x)_{t\geq0, x\in \bbR}$ the usual family of its local times and we introduce
\[
A_t^{\mm^\sca} = \int_{\bbR} L^x_t \mm^\sca(\dr x) \,.
\]
The process $(A_t^{\mm^\sca})_{t\geq0}$ is a non-decreasing continuous additive functional of the Brownian motion $(B_t)_{t\geq0}$. For every $x \in \bbR$ we have $L_{\infty}^x = \infty$ a.s., and since the support of the measure $\mm^\sca$ is not empty we see by Fatou's lemma that $A_{\infty}^{\mm^\sca} = \infty$ a.s. 
Now we introduce $(\rho_t)_{t\geq0}$ the right-continuous inverse of $(A_t^{\mm^\sca})_{t\geq0}$ and we set 
\[
X_t = \sca^{-1}( B_{\rho_t} ).
\]
As the change of time through a continuous non-decreasing additive functional of a strong Markov process preserves the strong Markovianity, see Sharpe \cite[Ch. VIII Thm. 65.9]{sharpe1988general}, and since $\sca$ is bijective, it holds that $(X_t)_{t\geq0}$ is a strong Markov process with respect to the filtration $(\mathcal{F}_{\rho_t})_{t\geq0}$.
Let us denote by $\mathrm{supp}(\mm^\sca)$ the support of the measure $\mm^\sca$. It is rather classical, see again Sharpe \cite[Ch. VIII Thm. 65.9]{sharpe1988general} or Revuz-Yor \cite[Ch. X Prop. 2.17]{ry}, that $(X_t)_{t\geq0}$ is valued in $\sca^{-1}(\mathrm{supp}(\mm^\sca))=\mathrm{supp}(\mm)$.
From now on, we will always assume that $0\in\mathrm{supp}(\mm^\sca)$ so that, since $\sca(0) = 0$, $(X_t)_{t\geq 0}$ spends time in $0$.
Since $0$ is recurrent for the Brownian motion, it is also recurrent for $(X_t)_{t\geq0}$. We have the following proposition, which shows that the process $(X_t)_{t\geq0}$ satisfies the hypothesis of this article and that its local time at the level $0$ can be expressed with the local time of the Brownian motion. The proof is postponed to Appendix \ref{appendix_loc_times}.

\begin{proposition}\label{prop_ok_gen}
 The following assertions hold.
 \begin{enumerate}[label=(\roman*)]
  \item The family $(L_{\rho_t}^{\sca(x)})_{t\geq0, x\in\mathbb{R}}$ defines a family of local times of $(X_t)_{t\geq0}$ in the sense that \textit{a.s.}, for any non-negative Borel functions $h$, for any $t\geq0$, the following occupation times formula holds:
\[
 \int_0^t h(X_s)\dr s = \int_\bbR h(x)L_{\rho_t}^{\sca(x)}\mm(\dr x).
\]

\item The point $0$ is regular for $(X_t)_{t\geq0}$.

\item The process $(L_{\rho_t}^0)_{t\geq0}$ is a proper local time for $(X_t)_{t\geq0}$ at the level $0$ in the sense that it is a continuous additive functional whose support almost surely coincides with the closure of the zero set of $(X_t)_{t\geq0}$.

\item Let $(\tau_t)_{t\geq0}$ be the right-continuous inverse of $(L_{\rho_t}^0)_{t\geq0}$ and $(\tau_t^B)_{t\geq0}$ be the right continuous inverse of $(L_t^0)_{t\geq0}$, then we have
\[
\tau_t = \int_{\bbR} L^x_{\tau^B_t} \mm^\sca(\dr x) = A_{\tau_t^B}^{\mm^\sca}.
\]
Moreover, it also holds that for any $t\geq0$, $\rho_{\tau_t} = \tau_t^B$.
 \end{enumerate}
\end{proposition}

\noindent
In this framework, the Lévy process $Z_t= \int_0^{\tau_t} f(X_s) \dr s$ can be expressed, thanks to items~\textit{(i)} and~\textit{(iv)} of Proposition~\ref{prop_ok_gen}, as
\[
Z_t= \int_{\bbR} L^x_{\tau_t^B} \mm^f (\dr x), 
\]
where we have set $\mm^f(\dr x) := f\circ \sca^{-1}(x) \mm^\sca(\dr x)$.
Note that $\mm^f$ is a signed measure (recall that by Assumption \ref{assump-f}, the function $f$ preserves the sign).
We suppose in addition that $f\circ \sca^{-1}$ is locally integrable with respect to $\mm^\sca$ so that $\mm^f$ is also a Radon measure. We will also denote by $\mm^f$ the associated function, \textit{i.e.}\ 
$
\mm^f(x) = \int_0^x f\circ \sca^{-1}(u) \mm^\sca(\dr u) ,
$
which is non-decreasing on~$\bbR_+$ and non-increasing on~$\bbR_-$.
Then it holds that  $(Z_t)_{t\geq 0}$ is a L\'evy process with finite variations, with zero drift (since $f(0)=0$). The aim of this section is to show Propositions \ref{prop_brow_null_rec} and \ref{prop_assump_true}

\subsection{Excursions of $(X_t)_{t\geq0}$ using those of $(B_t)_{t\geq0}$}
Let us first describe the excursions away from $0$ of $(X_t)_{t\geq0}$ in terms of the excursions of the Brownian motion. Let us denote by $\mathcal{D}$ the usual space of cadlad functions from $\bbR_+$ to $\bbR$, and for $\varepsilon\in\mathcal{D}$, let us introduce
\[
 \ell(\varepsilon) = \inf\{t > 0,\: \varepsilon_t = 0\}.
\]

\noindent
Then the set of excursions $\mathcal{E}$ is the set of functions $\varepsilon\in\mathcal{D}$ such that $0 < \ell(\varepsilon) < \infty$, for every $t\geq\ell(\varepsilon)$, $\varepsilon_t = 0$, and for every $0 < t < \ell(\varepsilon)$, $\varepsilon_t \neq 0$. This space is endowed with the usual Skorokhod's topology and the associated Borel $\sigma$-algebra.

\medskip
We now introduce the excursion processes of $(B_t)_{t\geq0}$ and $(X_t)_{t\geq0}$, denoted by $(e_t^B)_{t\geq0}$ and $(e_t^X)_{t\geq0}$, which take values in $\mathcal{E}\cup\{\Upsilon\}$, where $\Upsilon$ is an isolated cemetery point, and are given by
\begin{equation*}
 e_t^B = \left\lbrace
        \begin{array}{lll}
         (B_{\tau_{t-}^B + s})_{s\in[0,\Delta\tau_t^B]}&\text{if } \Delta\tau_t^B >0  \\
         \Upsilon &\text{otherwise}  
        \end{array}
    \right.
\quad \text{and} \quad e_t^X = \left\lbrace
        \begin{array}{lll}
         (X_{\tau_{t-} + s})_{s\in[0,\Delta\tau_t]}&\text{if } \Delta\tau_t >0  \\
         \Upsilon &\text{otherwise}  
        \end{array}\right.
\end{equation*}

\noindent
A famous result, essentially due to Itô \cite{i72}, states that $(e_t^B)_{t\geq0}$ and $(e_t^X)_{t\geq0}$ are Poisson point processes and we will respectively denote by $\nn^B$ and $\nn$ their characteristic measure, which are defined by
\begin{equation}\label{def_exc_measure}
\nn^B(\Gamma) = \frac{1}{t}\mathbb{E}\left[N_{t,\Gamma}^B\right]\qquad\text{and}\qquad  \nn(\Gamma) = \frac{1}{t}\mathbb{E}\left[N_{t,\Gamma}^X\right],
\end{equation}
for any measurable set $\Gamma$, where
\[
 N_{t,\Gamma}^B = \sum_{s\leq t}\bm{1}_\Gamma(e_s^B) \qquad\text{and}\qquad N_{t,\Gamma}^X = \sum_{s\leq t}\bm{1}_\Gamma(e_s^X).
\]

Our first aim is to describe the measure $\nn$ in terms of $\nn^B$; we will see $\nn$ as a push-forward of $\nn^B$ by some application $T$ ()
To this end, we first introduce the subset $\mathcal{C}$ of $\mathcal{E}$ of functions which are continuous. It is clear that every function in $\mathcal{C}$ has constant sign, and that $\nn^B(\mathcal{E}\setminus\mathcal{C}) = 0$. We have the following lemma.
\begin{lemma}\label{holder_local_time}
 Under $\nn^B$, almost every path posseses a family of local times $(\bm{\mathrm{L}}_t^x)_{t \geq 0, x \in \mathbb{R}}$, in the sense that for any Borel function $g$ and every $t\geq0$, we have
\[
 \int_0^t g(\epsilon_s) \dr s = \int_{\mathbb{R}}g(x) \bm{\mathrm{L}}_t^x \dr x \,.
\]
The family $(\bm{\mathrm{L}}_t^x)_{t \geq 0, x \in \mathbb{R}}$ is jointly continuous and for any $ \gamma \in (0,1/2)$ the map $x\mapsto \bm{\mathrm{L}}_t^x$ is Hölder of order $\gamma$ uniformly on compact time intervals.
\end{lemma}

\begin{proof}
 Consider some $t\geq0$ such that $\Delta\tau_t^B >0$ and set for  any $x\in\bbR$ and any $s\in[0,\Delta\tau_t^B]$, $\bm{\mathrm{L}}_s^x = L_{\tau_{t-}^B + s}^x - L_{\tau_{t-}^B}^x$. Then $(\bm{\mathrm{L}}_s^x)_{s\in[0,\Delta\tau_t^B], x\in\mathbb{R}}$ is the family of local time of $e_t^B$. Indeed, for any Borel function $g$ and $s\in[0,\Delta\tau_t^B]$, we have,
\[
 \int_0^s g(e_t^B(u))\dr u = \int_0^s g(B_{\tau_{t-}^B + u}) \dr u = \int_\bbR g(x)(L_{\tau_{t-}^B + s}^x - L_{\tau_{t-}^B}^x) \dr x.
\]

\noindent
Since the Brownian local times are jointly continuous and almost surely Hölder of order $\gamma$ (for any $\gamma\in(0,1/2)$) in the variable $x$  uniformly on compact time intervals, so is $(\bm{\mathrm{L}}_t^x)_{t \geq 0, x \in \mathbb{R}}$. Hence we showed, almost surely, for any $t$ such $\Delta\tau_t^B > 0$, $e_t^B$ has the property stated in the lemma: by \eqref{def_exc_measure}, this shows that outside of a negligeable set for $\nn^B$, every path has the stated property.
\end{proof}

We denote by $\mathcal{C}_+$ and $\mathcal{C}_-$ the subsets of $\mathcal{C}$ of positive and negative (continuous) excursions. We introduce the first point of increase and decrease of $\mm^\sca$ around $0$, \textit{i.e.}\ the real numbers defined by
\[
 x_+ = \inf\{x > 0,\: \mm^\sca(x) > \mm^\sca(0) \} \quad \text{and} \quad x_- = \sup\{x < 0,\: \mm^\sca(x) < \mm^\sca(0-)\}.
\]
(Note that for standard diffusions we have $x_+=x_=0$, but one may have $x_-<0<x_+$, for instance for birth and death chains, see Section~\ref{sec:examples}.)
Under Assumption~\ref{assump_2_prime} they are finite, \textit{i.e.}\ $\mm^\sca$ eventually increases and decreases. For a path $\varepsilon\in\mathcal{C}$, we let $M(\varepsilon) = \sup\{|\varepsilon_t|, t\geq0\}$ and we introduce the measurable set 
\[
 \bm{C}_{x_+,x_-} = \left(\mathcal{C}_+ \cap \{M(\varepsilon) > x_+\}\right)\cup\left(\mathcal{C}_- \cap \{M(\varepsilon) > |x_-|\}\right).
\]
For $\varepsilon\in\bm{C}_{x_+,x_-}$, we define the time-change
\begin{equation}
\label{eq:Atsca}
 (A_t^\sca)_{0\leq t \leq \ell} = \left(\int_\bbR \bm{\mathrm{L}}_t^x \mm^\sca(\dr x)\right)_{0\leq t \leq \ell}.
\end{equation}
Observe that $A_{\ell}^\sca >0$ if $\varepsilon\in \bm{C}_{x_+,x_-}$ (whereas $A_{\ell}^\sca = 0$ if $\varepsilon\in \mathcal{C}\setminus \bm{C}_{x_+,x_-}$). We denote by  $(\rho_t^\sca)_{0 \leq t \leq A_{\ell}}$ the right-continuous inverse of $(A_t^\sca)_{0\leq t \leq \ell}$ and finally define the measurable application $T:\bm{C}_{x_+,x_-} \to \mathcal{E}$ such that, 
\[
 T(\epsilon)_t = \sca^{-1}(\varepsilon_{\rho_t^\sca})\quad \text{if} \quad t< A_{\ell}^\sca \quad \text{and} \quad T(\varepsilon)_t = 0 \quad \text{if} \quad t\geq A_{\ell}^\sca.
\]
A key tool in this section is the following result, which expresses the measure $\nn$ as the pushforward measure of $\nn^B$ by~$T$.

\begin{proposition}\label{pushforward_brownian}
For any measurable set $\Gamma$, we have $\nn(\Gamma) = \nn^B(T^{-1}(\Gamma))$.
\end{proposition}

\noindent
Let us note that, since $T^{-1}(\Gamma) \subset \bm{C}_{x_+,x_-}$, the measure $\nn$ is a finite measure if and only if $x_+ >0$ and $x_- < 0$, since in this case $\nn^B(\bm{C}_{x_+,x_-})<\infty$.

\begin{proof}
We first emphasize that thanks to Proposition~\ref{prop_ok_gen}-\textit{(iv)}, we have a.s., for any $t\geq0$,
\[
 \Delta\tau_t = A^{\mm^\sca}_{\tau_t^B} - A^{\mm^\sca}_{\tau_{t-}^B} =  \int_\bbR (L_{\tau_t^B}^x - L_{\tau_{t-}^B}^x)\mm^\sca(\dr x)\,.
\]
Thus, it should be clear that 
\[
 \Delta\tau_t > 0 \qquad \text{if and only if} \qquad \Delta\tau_t^B >0 \quad\text{and}\quad e_t^B \in \bm{C}_{x_+,x_-}.
\]

\noindent
Let us now consider some $t \geq 0$ such that $\Delta\tau_t > 0$, \textit{i.e.}\ some $t\geq0$ such that $\Delta\tau_t^B >0$ and $e_t^B \in \bm{C}_{x_+,x_-}$. We set for $x\in\bbR$ and $s\in[0,\Delta\tau_t^B]$, $\bm{\mathrm{L}}_s^x = L_{\tau_{t-}^B + s}^x - L_{\tau_{t-}^B}^x$, the local time of $e_t^B$. Then we introduce the time change $(A_s^t)_{s\in[0,\Delta\tau_t^B]}$, defined by
\[
 A_s^t = \int_\bbR\bm{\mathrm{L}}_s^x\mm^\sca(\dr x) = A_{\tau_{t-}^B + s}^{\mm^\sca} - \tau_{t-}.
\]

\noindent
If $(\rho_s^t)_{s\in[0,\Delta\tau_t]}$ denotes the right-continuous inverse of $(A_s^t)_{s\in[0,\Delta\tau_t^B]}$, then for every $s\in[0,\Delta\tau_t]$, we get the following identity.
\[
 \rho_s^t = \inf\{u > 0, \: A_{\tau_{t-}^B + u}^{\mm^\sca} >  \tau_{t-} +s\} = \rho_{\tau_{t-} + s} - \tau_{t-}^B.
\]

\noindent
Therefore, we conclude that, if $t$ is such that $\Delta\tau_t> 0$, then we have
\[
 e_t^X = (\sca^{-1}(B_{\rho_{\tau_{t-} + s}}))_{s\in[0,\Delta\tau_t]} = (\sca^{-1}(B_{\tau_{t-}^B + \rho_s^t}))_{s\in[0,\Delta\tau_t]} = T(e_t^B).
\]

\noindent
Therefore, for any measurable set $\Gamma$, we have for any $t\geq0$, a.s. $N_{t,\Gamma}^X = N_{t,T^{-1}(\Gamma)}^B$, which, by~\eqref{def_exc_measure}, shows the result.
\end{proof}

We are now able to express some quantities of interest  of $(\tau_t,Z_t)_{t\geq 0}$ using the Brownian excursion measure. We first define for any $\varepsilon \in \bm{C}_{x_+,x_-}$ 
\begin{equation}
\label{eq:Atf}
 (A_t^f)_{0\leq t\leq \ell} = \left(\int_{\bbR}\bm{\mathrm{L}}_t^x\mm^f(\dr x) \right)_{0\leq t\leq \ell} \,.
\end{equation}
\begin{lemma}\label{lemma_charac_brown}
 The following assertions hold.
 \begin{enumerate}[label=(\roman*)]
  \item For any $\lambda > 0$ and any $\mu\in\bbR$,
  \[
   -\log\mathbb{E}\left[\e^{-\lambda \tau_1 + i\mu Z_1}\right] = \nn^B\left(1 - \exp\left(-\lambda A_\ell^\sca + i\mu A_\ell^f \right)\right).
  \]
  \item Let $e = e(q)$ be an independent exponential random variable of parameter $q > 0$. Then for any $\mu\in\bbR$,
  \[
   \mathbb{E}\left[\e^{i\mu I_e}\right] = \frac{q}{\Phi(q)}\bigg(\mathrm{m} +\nn^B\bigg(\int_0^\ell\exp(-q A_t^{\sca} + i \mu A_t^f)\dr A_t^\sca\bigg)\bigg)\,,
  \]
  where we recall that $\mathrm{m}$ is the drift coefficient of $(\tau_t)_{t\geq 0}$, see Section~\ref{sec:exc_theory}.
 \end{enumerate}
\end{lemma}

\begin{proof}
 Let us start with the first item. By the exponential formula for Poisson point processes, it is straightforward that for any $\lambda > 0$ and any $\mu\in\bbR$,
 \[
 \psi(\lambda, \mu) :=-\log \mathbb{E}\left[\e^{-\lambda\tau_1 + i\mu Z_1}\right] = \nn\bigg(1 - \exp\bigg(-\lambda\ell + i\mu\int_0^\ell f(\varepsilon_s) \dr s\bigg)\bigg) \,.
\]
Applying Proposition \ref{pushforward_brownian}, we get that
\[
 \psi(\lambda, \mu) = \nn^B\bigg(\bm{1}_{\{\varepsilon \in \bm{C}_{x_+,x_-}\}}\bigg(1-\exp\bigg(-\lambda A_\ell^\sca + i\mu\int_0^{A_\ell^\sca}f(\sca^{-1}(\varepsilon_{\rho_s^{\sca}}))\dr s\bigg)\bigg)\bigg)
\]
For any $\varepsilon \in \bm{C}_{x_+,x_-}$, we have $\int_0^{A_\ell^\sca}f(\sca^{-1}(\varepsilon_{\rho_s^{\sca}}))\dr s = \int_0^\ell f(\sca^{-1}(\varepsilon_u))\dr A_u^\sca = A_\ell^f$. Remark that if $\varepsilon\notin\bm{C}_{x_+,x_-}$, then $A_\ell^\sca = A_\ell^f = 0$ so that
\[
 \psi(\lambda, \mu) = \nn^B\left(1 - \exp\left(-\lambda A_\ell^\sca + i\mu A_\ell^f \right)\right).
\]
We now show the second item. By Proposition~\ref{indep}, we have for any $\mu\in\bbR$,
\[
   \mathbb{E}\left[\e^{i\mu I_e}\right] = \frac{q}{\Phi(q)}\bigg(\mathrm{m} +\nn\bigg(\int_0^\ell\exp\bigg(-q t + i \mu \int_0^t f(\varepsilon_s)\dr s\bigg)\dr t\bigg)\bigg) =: \frac{q}{\Phi(q)}(\mathrm{m} + G(\mu, q)).
\]
Then, by Proposition \ref{pushforward_brownian} again, we get
\[
 G(\mu, q) = \nn^B\bigg(\bm{1}_{\{\varepsilon \in \bm{C}_{x_+,x_-}\}}\int_0^{A_\ell^\sca}\exp\bigg(-q t + i \mu \int_0^t f(\sca^{-1}(\varepsilon_{\rho_s^{\sca}}))\dr s\bigg)\dr t\bigg).
\]
Performing the change of variables $t = A_u^\sca$ and using the occupation time formula for the Brownian excursion, we get
\[
 G(\mu, q) = \nn^B\bigg(\bm{1}_{\{\varepsilon \in \bm{C}_{x_+,x_-}\}}\int_0^\ell\exp(-q A_t^{\sca} + i \mu A_t^f)\dr A_t^\sca\bigg).
\]
If $\varepsilon\notin\bm{C}_{x_+,x_-}$, then $A_t^\sca = 0$ for any $t\in[0,\ell]$ and we can remove $\bm{1}_{\{\varepsilon \in \bm{C}_{x_+,x_-}\}}$ from the above expression, which gives the desired statement.
\end{proof}

\subsection{Additive functionals for strings with regular variation}\label{section_strings}

Assumptions \ref{assump_phi_beta_ito} and \ref{assump_m_f_alpha_ito} state that $\mm^\sca$ and $\mm^f$ are regularly varying and we will show that it implies that $(\tau_t, Z_t)_{t\geq 0}$ belongs to the domain of attraction of a bivariate stable process.
To prove that, we remark that the jumps of the rescaled process $(\tau^h_t, Z^h_t)_{t\geq 0}$ (with $h\downarrow 0$) are also additive functionals of the Brownian motion $(B_t)_{t\geq0}$ but involving rescaled measures $\mm_h^\sca$ and $\mm_h^f$.
We need technical tools to go from the convergence of $(\mm_h^\sca,\mm_h^f)$ to the convergence of the law of those additive functionals. The goal of this section is to provide those technical tools which are mainly inspired by a work by Fitzsimmons and Yano \cite{fy08} and are based on classical result from regular variation theory \cite{bgt89}. All the proofs of the following results are postponed to Appendix~\ref{appendix_strings}.

Let us now give some convergence results for sequences for additive functionals of Brownian excursions, in the spirit of \cite[Thm.~2.9]{fy08}; we will apply them to $\mm^\sca$ and $\mm^f$ (more precisely to their restrictions to $\bbR_+$ and $\bbR_-$). 
We consider a \emph{string}, \textit{i.e.}\ a non-decreasing right-continuous function $\mm$ on $\bbR_+$ such that $\mm(0) = 0$,  with regular variation in the sense that there are some $\alpha\in(0,2)$, and some smooth, locally bounded slowly varying function $\Lambda:\bbR_+\to(0,\infty)$, such that, according to the values of $\alpha$, we have

\begin{itemize}
 \item If $\alpha < 1$, then $\mm(x) \sim \Lambda(x)x^{1/\alpha -1}$ as $x\to\infty$.
 \vspace{1 mm}
 \item If $\alpha = 1$, then for any $x > 0$, $(\mm(x/h) - \mm(1/h)) / \Lambda(1/h) \longrightarrow \log x$ as $h\to0$.
 \vspace{1 mm}
 \item If $\alpha\in(1,2)$, then $\lim_{x\to\infty}\mm(x) = \mm(\infty) < \infty$ and $\mm(\infty)-\mm(x) \sim \Lambda(x)x^{1/\alpha -1}$ as $x\to\infty$.
\end{itemize}

\noindent
For such a string $\mm$, we define the family of rescaled strings $\mm_h, h>0$ as follows:
\begin{itemize}
 \item If $\alpha < 1$, then $\mm_h(x) = h^{1/\alpha - 1} \mm(x/h) / \Lambda(1/h)$.
 \vspace{1 mm}
 \item If $\alpha = 1$, then $\mm_h(x) = (\mm(x/h) - \mm(1/h))/ \Lambda(1/h)$.
 \vspace{1 mm}
 \item If $\alpha\in(1,2)$, then $\mm_h(x) = h^{1/\alpha - 1} (\mm(x/h) - \mm(\infty)) / \Lambda(1/h)$.
\end{itemize}

\noindent
Note that in all cases, we have $\mm_h(\dr x) = \frac{h^{1/\alpha - 1}}{\Lambda(1/h)}\mm(\dr(x/h))$. 
Moreover, it is easily checked that
\begin{itemize}
 \item If $\alpha \in (0,1)$, then for any $x > 0$, $\mm_h(x) \longrightarrow x^{1/\alpha - 1}$ as $h\to0$.
 \vspace{1 mm}
 \item If $\alpha = 1$, then for any $x > 0$, $\mm_h(x)\longrightarrow \log x$ as $h\to0$.
 \vspace{1 mm}
 \item If $\alpha \in (1,2)$, then for any $x > 0$, $\mm_h(x) \longrightarrow -x^{1/\alpha - 1}$ as $h\to0$.
\end{itemize}

\noindent
The following lemma shows that the regular variations of $\mm$ implies the convergence for additive functionals with respect to the rescaled strings $\mm_h$ (as $h \to 0$). This will reveal to be a key result in the proof of Proposition \ref{prop_assump_true} and is close from results in \cite{fy08}.

\begin{lemma}\label{lemconv_excursion}
Let $\mm$ be such a string with regular variation and $g:[0,\infty) \to \bbR$ some compactly supported continuous function such that $g(0)=0$ and which is $\gamma$-Hölder at $0$ for any $\gamma<1/2$. Then 
\[
\lim_{h\downarrow 0}\int_{\bbR_+} g(x) \mm_h(\dr x)  = c_\alpha \int_{\bbR_+} g(x) x^{1/\alpha -2} \dr x,
\]
where $c_\alpha = \vert 1/\alpha -1 \vert $ if $\alpha\neq1$ and $c_\alpha =1$ for $\alpha =1$.
\end{lemma}

\begin{remark}
\label{rem_conv_unif}
Thanks to Lemma \ref{holder_local_time}, the previous lemma applies to the family of local times of the Brownian excursion, \textit{i.e.}\ $x\mapsto \bm{\mathrm{L}}_t^x$. For $t\in [0,\ell]$, we have $\nn^B_{+}$-a.e,
\[
\lim_{h\downarrow 0}\int_{\bbR_+} \bm{\mathrm{L}}_t^x \mm_h(\dr x) = c_\alpha \int_{\bbR_+} \bm{\mathrm{L}}_t^x\, x^{1/\alpha -2} \dr x.
\]
Note also that since $t\mapsto \int_{\bbR_+} \bm{\mathrm{L}}_t^x \mm_h(\dr x)$ is non-decreasing and the limiting function $t\mapsto  \int_{\bbR_+} \bm{\mathrm{L}}_t^x x^{1/\alpha -2} \dr x$ is continuous, the convergence holds in uniform norm on compact set. 
\end{remark}

We will use the following technical lemma to go from the convergence of additive functionals for an excursion to the convergence in measure (it is a truncation and domination lemma).  
\begin{lemma}\label{conv_domination}
Define $\mathbf{A}_\delta :=\{ \varepsilon \in \mathcal{C}, \sup_{s\in [0,\ell]}\vert \varepsilon_s \vert \leq \delta \}$ then for $0<\delta <1$ : 
\begin{enumerate}[label=(\roman*)]
\item If $\alpha \in (0,1)$ then $\underset{h\downarrow0}{\limsup} \; \nn^B_+ \left [  \Big (\int_{\bbR_+} \bm{\mathrm{L}}_\ell^x \mm_h(\dr x)  \Big) \mathbf{1}_{\mathbf{A}_\delta}  \right ] \leq  2 \delta^{\frac1\alpha-1}$, which goes to $0$ as $\delta\downarrow 0$.
\item If $\alpha \in [1,2)$ then 
\[
\lim_{\delta\downarrow0} \nn^B_+ \bigg [  \underset{h\in (0,1)}{\sup} \Big (  \int_{\bbR_+} \bm{\mathrm{L}}_\ell^x \mm_h(\dr x) \Big )^2 \mathbf{1}_{\mathbf{A}_\delta } \bigg ]  = 0.
\]
\item If $\alpha \in (1,2)$ then
\[
\nn_+^B \bigg [ \underset{h\in (0,1)}{\sup} \Big(  \int_{\bbR_+} \bm{\mathrm{L}}_\ell^x \mm_h(\dr x) \Big)   \mathbf{1}_{\mathbf{A}_\delta^c}\bigg ] <+\infty.
\]
\item If $\alpha=1$ then
\[
\nn_+^B \bigg [ \underset{h\in (0,1)}{\sup} \Big(  \int_0^{1} \bm{\mathrm{L}}_\ell^x \mm_h(\dr x) \Big)   \mathbf{1}_{\mathbf{A}_\delta^c}\bigg ] <+\infty.
\]
\end{enumerate}
\end{lemma}

Although these results are stated for functionals of positive Brownian excursions, they obviously hold for functionals of negative excursions as the measure $\nn^B$ is invariant by the application $\varepsilon \mapsto - \varepsilon$. In the following, we apply these results to the strings $\mm^\sca$ and $\mm^f$. Indeed, since we assumed that $\mm^\sca(0) = 0$ and $\mm^f(0) = 0$, we can first apply the results to $\mm^\sca$ and $\mm^f$ restricted to $\bbR_+$, denoted $\mm^{\sca}_+$ and $\mm^f_+$. Then we can also apply them to $\mm^\sca_-$ and $\mm^f_-$ defined on $\bbR_-$ as $\mm^\sca_- = \mm^\sca - \mm^\sca(0-)$, $\mm^f_- = \mm^f - \mm^f(0-)$ on $(-\infty, 0)$ and $\mm^\sca_-(0) = 0$, $\mm^f_-(0) = 0$.

\subsection{Proof of Proposition \ref{prop_assump_true}}
We finally prove here Proposition \ref{prop_assump_true}. 
We let the reader recall Assumptions \ref{assump_phi_beta_ito} and \ref{assump_m_f_alpha_ito}, which are assumed throughout this proof (they tell that $\mm^{\sca}_{\pm}$ and $\mm^f_{\pm}$ are \textit{strings} with regular variation).
Let us also introduce the following notation: for $a,b, x \in\mathbb{R}$, we set $\text{sgn}_{a,b}(x) = a\bm{1}_{\{x > 0\}} + b\bm{1}_{\{x < 0\}}$. The function $\text{sgn}$ will denote the usual sign function, \textit{i.e.}\ $\text{sgn}(x) = \text{sgn}_{1,-1}(x)$. Then, for $\varepsilon\in\mathcal{C}$, let us denote:
\[
A_\ell^\beta(\varepsilon) = c_{\beta}\int_0^{\ell}\text{sgn}_{m_+,m_-}(\varepsilon_s)|\varepsilon_s|^{1/\beta - 2}\dr s, \qquad A_\ell^\alpha(\varepsilon) = c_{\alpha}\int_0^\ell \text{sgn}_{f_+,-f_-}(\varepsilon_s)|\varepsilon_s|^{1/\alpha - 2}\dr s \,,
\]
where $c_{\alpha}$ is as in Lemma~\ref{lemconv_excursion}.
In the case $\alpha=1$, $f_+ =f_-=1$ so $A_\ell^1(\varepsilon) = \int_0^\ell \frac{\dr s}{\varepsilon_s}.$

We proceed in two steps: first, we prove that the rescaled Lévy process $(\tau^h_t, Z^h_t)_{t\geq 0}$ converges as $h\downarrow0$; then we prove the convergence of the other term $a(b^{-1}(q))I_e$ as $q\downarrow 0$.

\paragraph*{Step 1: Convergence of $(\tau_t^h,Z_t^{h})$.}
Let us set $b(h) = h^{1/\beta} / \Lambda_\sca(1/h)$ and $a(h) = h^{1/\alpha}/ \Lambda_f(1/h)$.
We show that  $(\tau^{h}_t, Z_{t}^{h})_{t\geq0} :=( b(h)\tau_{t/h}, a(h) Z_{t/h})_{t\geq0}$ converges in distribution to $(\tau^0_t, Z^0_{t})_{t\geq0}$,
where $(\tau^0_t, Z^0_{t})_{t\geq0}$ is a L\'evy process, without drift (except in the case $\alpha=1$ where the drift is some constant $\mathbf{\tilde{c}}$) and without Brownian component, whose L\'evy measure $\pi^0(\dr r, \dr z)$ is defined as
\[
 \pi^0(\dr r, \dr z) = \nn^B\left(A_\ell^\beta(\varepsilon) \in \dr r, A_\ell^\alpha(\varepsilon) \in \dr z\right).
\]

\noindent
To this end, we will show that the following convergence holds for every $\lambda > 0$ and $\mu\in\bbR$:
\begin{equation}
 \lim_{h\downarrow0}\mathbb{E}\left[\e^{-\lambda\tau_t^h + i\mu Z_t^h}\right] = \mathbb{E}\left[\e^{-\lambda\tau_t^o + i\mu Z_t^o}\right] \,.
\end{equation}

\noindent
Let us set for any $\lambda,h >0$, and any $\mu\in\bbR$, $\psi_h(\lambda, \mu)=-\log \mathbb{E}[\e^{-\lambda\tau_1^h + i\mu Z_1^h}]$. Then we get by Lemma \ref{lemma_charac_brown}-\textit{(i)} that
\[
 \psi_h(\lambda, \mu) = \frac{1}{h}\nn^B\left( 1 - \exp\left(-\lambda b(h)\int_{\bbR}\bm{\mathrm{L}}_\ell^x \mm^\sca(\dr x) + i\mu a(h)\int_{\bbR}\bm{\mathrm{L}}_\ell^x \mm^f(\dr x)\right)   \right).
\]

\noindent
We will now use the scaling property of the Brownian excursion measure: for every $h>0$, we have $\nn_B = h\,\nn_B \circ \lambda_h^{-1}$, where $\lambda_h:\mathcal{E}\to\mathcal{E}$ is defined by $\lambda_h(\varepsilon) = (h^{-1}\varepsilon_{th^2})_{t\geq0}$. Now if $(\bm{\mathrm{L}}_t^x)_{x\in\mathbb{R}, t\in[0,\ell]}$ denotes the family of local times of $\varepsilon$, then $(h^{-1}\bm{\mathrm{L}}_{th^2}^{hx})_{x\in\mathbb{R}, t\in[0,\ell/h^2]}$ is the family of local times of $\lambda_h(\varepsilon)$. Then we get
\[
 \psi_h(\lambda, \mu) = \nn^B\left(1 - \exp\left(-\lambda \int_{\bbR}\bm{\mathrm{L}}_\ell^x \mm_h^\sca(\dr x) + i\mu \int_{\bbR}\bm{\mathrm{L}}_\ell^x \mm_h^f(\dr x)\right)  \right),
\]
where the measures $\mm_h^\sca$ and $\mm_h^f$ are defined as follows:
\begin{itemize}[leftmargin=*]
 \item For every $x\in\mathbb{R}$, $\mm_h^\sca(x) = h^{1/\beta - 1}\mm^\sca(x/h) / \Lambda_\sca(1/h)$ (recall $\beta\in (0,1)$).
 \item According to the values of $\alpha$, we have:
 \begin{enumerate}[label=(\roman*)]
 \item If $\alpha\in(0,1)$, for every $x\in\mathbb{R}$, $\mm_h^f(x) = h^{1/\alpha - 1}\mm^f(x/h) / \Lambda_f(1/h)$
 \item If $\alpha = 1$, for every $x\geq 0$, $\mm_h^f(x) = (\mm^f(x/h) - \mm^f(1/h)) / \Lambda_f(1/h)$ and for every $x< 0$, $\mm_h^f(x) = (\mm^f(x/h) - \mm^f(-1/h)) / \Lambda_f(1/h)$.
 \item If $\alpha \in (1,2)$, then for every $x\in \bbR$, $\mm_h^f(x) = h^{1/\alpha - 1}(\mm^f(x/h) - \mm^f(\infty)) / \Lambda_f(1/h)$.\end{enumerate}
\end{itemize}

\noindent
Before going further, we introduce the following notation, analogous to~\eqref{eq:Atsca}-\eqref{eq:Atf}
\begin{equation}
\label{eq:Ahscaf}
A_\ell^{\sca,h}(\varepsilon) = \int_{\bbR}\bm{\mathrm{L}}_\ell^x \mm_h^\sca(\dr x), \qquad A_\ell^{f,h}(\varepsilon) = \int_{\bbR}\bm{\mathrm{L}}_\ell^x \mm_h^f(\dr x),
\end{equation}
so that 
$\psi_h(\lambda,\mu) = \nn^B\big(\varphi(A_\ell^{\sca,h},A_\ell^{f,h}) \big)$ with $\varphi(x,y)= 1-e^{-\lambda x+i\mu y}$.

Then, thanks to Assumptions~\ref{assump_phi_beta_ito} and~\ref{assump_m_f_alpha_ito} we can apply Lemma \ref{lemconv_excursion}: we have for $\nn^B$-almost every excursion $\varepsilon\in\mathcal{C}$,
\begin{equation}
\label{eq:limitAh}
\lim_{h\downarrow0} A_\ell^{\sca,h}(\varepsilon)= A_\ell^\beta(\varepsilon) \qquad \text{and} \qquad \lim_{h\downarrow0} A_\ell^{f,h}(\varepsilon)= A^\alpha(\varepsilon).
\end{equation}
It remains to prove that we can exchange the limits inside $\psi_h(\lambda,\mu)$.
Recall that we set, for $\delta>0$, $\mathbf{A}_\delta = \{ \varepsilon \in \mathcal{C}, M(\vert \varepsilon \vert) \leq \delta \}$ and that $\nn^B(\mathbf{A}_\delta^c)<+\infty$.

\medskip\noindent
\textit{Case $\alpha\in(0,1)$.\ } We need to prove that $\lim_{h\downarrow0} \psi_h(\lambda,\mu) =\psi_0(\lambda, \mu):= \nn^B [ \varphi(A_\ell^\beta, A_\ell^\alpha)]$, where $\psi_0$ is the characteristic exponent of the limit process $(\tau^0_t, Z^0_{t})_{t\geq0}$. We have
\begin{align*}
 \left|\psi_h(\lambda, \mu) -\psi_0(\lambda, \mu)\right| \leq  \nn^B \left [ \big \vert \varphi(A_\ell^{\sca,h} ,A_\ell^{f,h})-\varphi(A_\ell^\beta,A_\ell^\alpha) \big \vert  \mathbf{1}_{\mathbf{A}_\delta^c}\right ]  + \nn^B &\left [ \vert \varphi(A_\ell^{\sca,h} ,A_\ell^{f,h}) \vert \mathbf{1}_{\mathbf{A}_\delta } \right ] \\
  & + \nn^B \left [ \vert \varphi(A_\ell^\beta , A_\ell^\alpha) \vert \mathbf{1}_{\mathbf{A}_\delta } \right ] .
\end{align*}
Since $\varphi$ is bounded and $\nn^B(\mathbf{A}_{\delta}) <+\infty$, by~\eqref{eq:limitAh} and the dominated convergence theorem, it comes
\[
\underset{h \to 0}{\limsup} \;  \nn^B \left [ \big \vert \varphi(A_\ell^{\sca,h} ,A_\ell^{f,h})-\varphi(A_\ell^\beta,A_\ell^\alpha) \big \vert  \mathbf{1}_{\mathbf{A}_\delta^c}\right ] =0.
\]
Since there is some $C_{\lambda, \mu}>0$ such that $\vert \varphi (x,y) \vert \leq C_{\lambda,\mu}(x+ \vert y\vert)$ for all $x\geq0, y\in \bbR$, from item~\textit{(i)} of Lemma \ref{conv_domination} (decomposing for positive and negative excursions), we get that 
\[
\underset{h \to 0}{\limsup} \; \nn^B \left [ \vert \varphi(A_\ell^{\sca,h} , A_\ell^{f,h}) \vert \mathbf{1}_{\mathbf{A}_\delta }\right] \longrightarrow 0 \qquad \text{ as } \quad \delta\downarrow 0.
\]
By Fatou's lemma, we also get that $\nn^B \big [ \vert \varphi(A_\ell^\beta , A_\ell^\alpha) \vert \mathbf{1}_{\mathbf{A}_\delta }\big]$ converges to $0$ as $\delta \downarrow 0$. We now quickly explain the value of the constant $\rho = \mathbb{P}(Z_t^0 \geq0)$ in Proposition \ref{prop_assump_true}. One can easily see that a representation of the limiting $\alpha$-stable process $(Z_t^0)_{t\geq0}$ is
\[
 Z_t^0 = \int_0^{\tau_t^B}\text{sgn}_{f_+,-f_-}(B_s)|B_s|^{1/\alpha - 2}\dr s
\]
where we recall that $(B_t)_{t\geq0}$ is a Brownian motion and $(\tau_t^B)_{t\geq0}$ is its inverse local time at $0$. The characteristic function of $(Z_t^0)_{t\geq0}$ is computed in \cite[Lemma 11]{bethencourt2021stable} and we can deduce the value of $\rho$ using a formula due to Zolotarev \cite[\S2.6]{zolotarev1986one}.

\medskip\noindent
\textit{Case $\alpha\in(1,2)$.\ } Recall that there exists a positive constant $\mm^{f}(\infty) < \infty$ such that $\mm^f(x) \to \mm^{f}(\infty)$ as $x\to\pm\infty$. It follows that $\mm^f(\bm{1}) = 0$ and by Lemma \ref{lemma_moy} below (whose proof is postponed to Appendix~\ref{appendix_loc_times}), we have that $\nn\big( \int_0^\ell f(\varepsilon_s)\dr s \big) = 0$, which in turn implies $\nn^B(A_\ell^{f}) = 0$. 
Obviously, we also have $\nn^B(A_\ell^{f, h}) = 0$ for any $h>0$.

\begin{lemma}\label{lemma_moy}
 Let $g$ be a positive Borel function such that $\mm^f_+(g) < \infty$ and $|\mm^f_-(g)| < \infty$. Then we have $\mm^f(g) = \nn\big(\int_0^{\ell} ((g\circ\sca)\times f)(\epsilon_s) \dr s\big)$.
\end{lemma}

\noindent
Let us introduce $\bar\varphi(x,y)= 1-e^{-\lambda x +i \mu y}-i \mu \chi(y)$ where $\chi(y)=y$ on $[-1,1]$ and is continuous with compact support and odd. Thus
\[
 \psi_h(\lambda, \mu) = \nn^B\left [ \bar\varphi(A_\ell^{\sca,h},A_\ell^{f,h}) \right ] + i \mu \,\nn^B \left [ \chi(A_\ell^{f,h}) \right ].
\]
Since $\bar \varphi$ is bounded and since there exists a constant $C_{\lambda,\mu} > 0$ such that  $\vert \bar \varphi(x,y)\vert \leq C_{\lambda,\mu}(x+ y^2)$ for any $x \geq0$, $y\in\bbR$, we obtain by points~\textit{(i)}-\textit{(ii)} of Lemma ~\ref{conv_domination} and decomposing as previously on $\mathbf{A}_\delta$ and $\mathbf{A}_\delta^c$ that 
\[
\lim_{h\downarrow0}\nn^B\left [ \bar\varphi(A_\ell^{\sca,h},A_\ell^{f,h}) \right ]  = \nn^B\left [ \bar\varphi(A_\ell^\beta, A_\ell^\alpha) \right ].
\]
It remains to check that $ \nn^B [ \chi(A_\ell^{f,h} ) ]$ converges as $h\downarrow 0$. Recall that $\nn^B(A_\ell^{f,h}) = 0$, so that $\nn^B [ \chi(A_\ell^{f,h} ) ]=- \nn^B [A_\ell^{f,h} -\chi(A_\ell^{f,h})]$. Note that there exists $C>0$ such that $ \vert x-\chi(x) \vert \leq C x^2$,  by item~\textit{(ii)} of Lemma \ref{conv_domination}, it comes that 
\[
\limsup_{h\downarrow 0} \Big \vert  \nn^B \Big [ \big (A_\ell^{f,h} -\chi(A_\ell^{f,h} ) \big) \mathbf{1}_{\mathbf{A}_\delta} \Big ]  \Big \vert
\longrightarrow  0 \qquad \text{ as } \delta \downarrow 0.
\]
Moreover, for any $\delta >0$, by item~\textit{(iii)} of Lemma \ref{conv_domination} and since $\vert x -\chi(x) \vert \leq C \vert x \vert$, we get by the dominated convergence theorem that for any $\delta > 0$,
\[
\lim_{h\downarrow 0} \nn^B \Big [ \big (A_\ell^{f,h} -\chi(A_\ell^{f,h} ) \big) \mathbf{1}_{\mathbf{A}_\delta^c} \Big ] =  \nn^B \Big [ \big ( A_\ell^\alpha-\chi(A_\ell^\alpha ) \big) \mathbf{1}_{\mathbf{A}_\delta^c} \Big ].
\]
We conclude that that 
\[
\lim_{h\downarrow 0} \nn^B \Big [A_\ell^{f,h} -\chi(A_\ell^{f,h} )\Big ] = \nn^B \Big [A_\ell^\alpha -\chi(A_\ell^\alpha )\Big ],
\]
so that $\nn^B[\chi(A_\ell^{f,h} )]$ converges  to $\nn^B[\chi(A_\ell^{\alpha})]$ as $h\downarrow 0$. Again, let us quickly explain the value of $\rho$. In the case $\alpha\in(1,2)$, the limiting process has the following representation
\[
 Z_t^0 = \int_{\bbR}\text{sgn}_{f_+, -f_-}(x)|x|^{1/\alpha - 2} (L_{\tau_t^B}^x - t)\dr x
\]
where we recall that $(L_t^x)_{t\geq0, x\in\bbR}$ denotes the family of local times of $(B_t)_{t\geq0}$. Again, the characteristic function of $(Z_t^0)_{t\geq0}$ is computed in \cite[Lemma 12]{bethencourt2021stable} which is sufficient to deduce the value of $\rho$.

\medskip\noindent
\textit{Case $\alpha=1$.\ } Using the same notation as in the previous case and with the same reasoning, we also have that $\nn^B [\bar \varphi(A_\ell^{\sca,h}, A_\ell^{f,h})]$ converges to $\nn^B [ \bar \varphi(A_\ell^\beta, A_\ell^\alpha) ]$ when $h$ goes to $0$. It remains to prove that the extra assumption 
$\lim_{h\downarrow 0}\frac{1}{\Lambda_f(1/h)}(\mm^f(1/h)-\mm^f(-1/h)) = \mathbf{c}$
implies that $\nn^B [ \chi(A_\ell^{f,h}) ]$ converges. Note that the limit is not $\nn^B \left [ \chi(A_\ell^\alpha)\right ]$ which is infinite when $\alpha=1$. Writing that $\mm^f(1/h)-\mm^f(-1/h) = \int_\bbR \bm{1}_{\{x\in (-1/h,1/h]\}}\mm^f(\dr x)$ and using Lemma \ref{lemma_moy} with $g(x) = \bm{1}_{\{x\in (-1/h,1/h]\}}$, we get
\[
 \mm^f(1/h)-\mm^f(-1/h) = \nn\left(\int_0^\ell((g\circ\sca) \times f)(\varepsilon_s) \dr s\right).
\]
Then, using Proposition \ref{pushforward_brownian} and the occupation time formula for the Brownian excursion, we easily get that
\[
 \frac{1}{\Lambda_f(1/h)}\big (\mm^f(1/h)-\mm^f(-1/h) \big ) = \nn^B\left[\tilde{A}_\ell^{f,h}\right]
\]
where $\tilde{A}_\ell^{f,h} := \int_\bbR \bm{\mathrm{L}}_\ell^x \bm{1}_{\{x\in (-1,1]\}}\mm_h^f(\dr x)$.
Note that as in~\eqref{eq:limitAh}, thanks to Lemma~\ref{lemconv_excursion} (and Assumption~\ref{assump_m_f_alpha_ito}) we have for $\nn^B$-almost every excursion $\varepsilon\in\mathcal{C}$
\[
\lim_{h\downarrow 0} \tilde{A}_\ell^{f,h}  = \tilde{A}^\alpha_\ell := c_\alpha \int_{(-1,1]} \bm{\mathrm{L}}_\ell^x \; \text{sgn}_{f_+,f_-}(x) x^{1/\alpha -2} \dr x \,.
\]
(Recall $f_+=f_-$ in the case $\alpha=1$).
Remark that we have $ \tilde{A}_\ell^{f,h}=A_\ell^{f,h}$ for every $\varepsilon \in \mathbf{A}_{1}$: we can then decompose
\begin{align*}
\nn^B [ \chi(A_\ell^{f,h})] &= \nn^B [ \chi \big (\tilde{A}_\ell^{f,h} \big ) \mathbf{1}_{\mathbf{A}_1}]+ \nn^B [ \chi \big (A_\ell^{f,h} \big ) \mathbf{1}_{\mathbf{A}_1^c} ] \\
&=  \nn^B \Big [ \Big ( \chi \big (  \tilde{A}_\ell^{f,h} \big ) - \tilde{A}_\ell^{f,h} \Big )\mathbf{1}_{\mathbf{A}_1}\Big ] + \nn^B \big [ \tilde{A}_\ell^{f,h} \big ] - \nn^B \big [ \tilde{A}_\ell^{f,h}  \mathbf{1}_{\mathbf{A}_1^c}\big ] + \nn^B \big [ \chi(A_\ell^{f,h}) \mathbf{1}_{\mathbf{A}_1^c} \big ].
\end{align*}
The second term converges (to $\mathbf{c}$) by assumption. It remains to check that the other three terms converge. Regarding the first term, using item~\textit{(ii)} in Lemma \ref{conv_domination}, and the fact that $\vert \chi(x)-x\vert \leq C x^2$ for some constant $C > 0$, we get that $\nn^B [\sup_{h\in(0,1)}  ( \chi (  \tilde{A}_\ell^{f,h})- \tilde{A}_\ell^{f,h})\mathbf{1}_{\mathbf{A}_1}]<+\infty$. By the dominated convergence theorem, it comes that
\[
 \lim_{h\downarrow0} \nn^B \Big [ \Big ( \chi \big ( \tilde{A}_\ell^{f,h} \big)- \tilde{A}_\ell^{f,h} \Big )\mathbf{1}_{\mathbf{A}_1}\Big ]  = \nn^B \Big [ \Big ( \chi \big ( \tilde{A}_\ell^\alpha \big )- \tilde{A}_\ell^\alpha \Big )\mathbf{1}_{\mathbf{A}_1}\Big ] \,.
\]
We obtain the convergence of the third term applying point~\textit{(iv)} in Lemma \ref{conv_domination} (and dominated convergence). As a conclusion, we have
\[
 \lim_{h\downarrow0}\nn^B [ \chi(A_\ell^{f,h})] = \nn^B \Big [ \Big ( \chi \big ( \tilde{A}_\ell^\alpha \big )- \tilde{A}_\ell^\alpha \Big )\mathbf{1}_{\mathbf{A}_1}\Big ] - \nn^B \big [ \tilde{A}_\ell^\alpha  \mathbf{1}_{\mathbf{A}_1^c}\big ] + \nn^B \big [ \chi(A_\ell^\alpha) \mathbf{1}_{\mathbf{A}_1^c} \big ] + \mathbf{c}.
\]
Now we point out that, since $\chi$ is odd and $f_+ = f_-$, the first three terms in the above limit are null by symmetry. Hence the limiting process $(Z_t^0)_{t\geq0}$ is a Cauchy process drifted by $\mathbf{c}$ (whence the value of $\rho$ from the statement).

\paragraph*{Step 2: Convergence of $a(b^{-1}(q))I_e$.}
Let us stress that since $b(h) \tau_{t/h}$ converges in distribution, $b$ is an asymptotic inverse of $\Phi$ (see Remark~\ref{rem:bPhi}); in particular, $\Phi$ is regularly varying at $0$ with index $\beta \in(0,1)$.

By Lemma \ref{lemma_charac_brown}-\textit{(ii)}, we have that for any $\mu\in\bbR$,
\[
 \mathbb{E}\left[\e^{ia(b^{-1}(q))\mu I_e}\right] = \frac{q}{\Phi(q)}\bigg(\mathrm{m} +\nn^B\bigg(\int_0^\ell\exp\big(-q A_t^{\sca} + i a(b^{-1}(q))\mu A_t^f\big) \dr A_t^\sca\bigg)\bigg) =: \frac{q}{\Phi(q)}(\mathrm{m} + G(\mu, q)),
\]
where we recall that $\mathrm{m}$ is the drift coefficient of $(\tau_t)_{t\geq0}$ and $A_t^\sca$, $A_t^f$ are defined in~\eqref{eq:Atsca} and~\eqref{eq:Atf} respectively.
Since $\Phi$ is regularly varying with index $\beta \in(0,1)$ we have $\lim_{q\downarrow0}\mathrm{m} / \Phi(q) = 0$. 

Therefore, it remains to show that $qG(\mu, q) / \Phi(q)$ converges as $q\downarrow 0$, or equivalently that $qG(\mu, q) / b^{-1}(q)$ converges.
Using the scaling property of the Brownian excursion measure, \textit{i.e.}\ $\nn_B = h\, \nn_B \circ \lambda_h^{-1}$, with $h = b^{-1}(q)$, we get 
\begin{equation}\label{fourier_ie}
 \frac{q}{b^{-1}(q)}G(\mu,q) = \nn^B\left(\int_0^{\ell} \exp\left(-A_t^{\sca,q} +i\mu A_t^{f,q}\right)\dr A_t^{\sca,q}\right),
\end{equation}
where, similarly as in~\eqref{eq:Atsca}-\eqref{eq:Atf}, we have set
\[
 A_t^{\sca,q} = \int_{\bbR}\bm{\mathrm{L}}_t^x \mm^\sca_{b^{-1}(q)}(\dr x), \quad \quad A_t^{f,q} = \int_{\bbR}\bm{\mathrm{L}}_t^x \mm^f_{b^{-1}(q)}(\dr x),
\]
and the measures $\mm_h^\sca$ and $\mm^f_{h}$ are the measures defined in Step~1, see above~\eqref{eq:Atsca}. By Lemma~\ref{lemconv_excursion} (and Assumptions~\ref{assump_phi_beta_ito} and~\ref{assump_m_f_alpha_ito}), we have that for $\nn^B$-almost every excursion $\varepsilon\in\mathcal{C}$, for any $t\in[0,\ell]$,
$\lim_{q\downarrow 0}A_t^{\sca,q}  = A_t^\beta$  and  $\lim_{q\downarrow 0}A_t^{f,q} = A_t^\alpha$.

Let us show that for $\nn^B$-almost every excursion,
\begin{equation}\label{conv_int_petit_bout}
\lim_{q\downarrow 0} \int_0^{\ell} \exp\left(- A_t^{\sca,q} +i\mu A_t^{f,q}\right)\dr A_t^{\sca,q}  = \int_0^{\ell} \exp\left(- A_t^{\beta} +i\mu A_t^{\alpha}\right)\dr A_t^{\beta}.
\end{equation}

\noindent
First, the finite measures $\dr A^{\sca, q}$ on $([0,\ell], \mathcal{B}([0,\ell]))$ converge weakly to the finite measure $\dr A^\beta$ as $q\to0$. Since $t\mapsto\exp(-cA_t^{\beta} +i\mu A_t^{\alpha})$ is continuous and bounded, we deduce that
\[
\lim_{q\downarrow 0} \int_0^{\ell}\exp\left(-A_t^{\beta} +i\mu A_t^{\alpha}\right)\dr A_t^{\sca,q} = \int_0^{\ell} \exp\left(-A_t^{\beta} +i\mu A_t^{\alpha}\right)\dr A_t^{\beta}.
\]
Next, we bound
\begin{align*}
 \Big|\int_0^{\ell} \Big(\exp\Big(-&  A_t^{\sca,q} +i\mu A_t^{f,q}\Big) - \exp\Big(-A_t^{\beta} +i\mu A_t^{\alpha}\Big) \Big)\dr A_t^{\sca,q}\Big| \\
  & \leq A_\ell^{\sca, q} \sup_{t\in[0,\ell]}\Big|\exp\Big(-A_t^{\sca,q} +i\mu A_t^{f,q}\Big) - \exp\Big(-A_t^{\beta} +i\mu A_t^{\alpha}\Big)\Big|.
\end{align*}
By Remark~\ref{rem_conv_unif}, the convergence of $A_t^{\sca, q}$ and $A_t^{f, q}$ is uniform on $[0,\ell]$, \textit{i.e.}\  $\sup_{t\in[0,\ell]}|A_t^{\sca, q} - A_t^{\beta}| + |A_t^{f, q} - A_t^{\alpha}|$ vanishes as $q\downarrow 0$. Since $(x,y) \mapsto\exp(-x + iy)$ is Lipschitz, it is clear the above quantities converges to $0$ as $q\downarrow 0$, which proves \eqref{conv_int_petit_bout}.

Finally, we bound the integrand in \eqref{fourier_ie} to apply dominated convergence. We have
\begin{equation}
  \label{bound}
\begin{split}
 \Big \vert \int_0^{\ell} \exp\big(- A_t^{\sca,q} +i\mu A_t^{f,q}\big)\dr A_t^{\sca,q}\Big|  &\leq \int_0^\ell \exp\big(-A_t^{\sca,q}\big)\dr A_t^{\sca,q} \\ 
  & = 1 - \exp\big(-A_\ell^{\sca,q}\big) \leq 1 \wedge A_\ell^{\sca,q} \,.
  \end{split}
\end{equation}
We conclude as before, by introducing some arbitrary small $\delta >0$ and by splitting the integral~\eqref{fourier_ie} into two parts, on the sets $\mathbf{A}_\delta$ and $\mathbf{A}_\delta^c$. Using the bound~\eqref{bound}, we obtain thanks to item~\textit{(i)} in Lemma~\ref{conv_domination} (recall $\beta \in (0,1)$) that
\[
\limsup_{q \to 0}\nn^B\bigg [ \Big (\int_0^{\ell} \exp \Big(-A_t^{\sca,q} +i\mu A_t^{f,q}\Big)\dr A_t^{\sca,q} \bigg) \mathbf{1}_{\mathbf{A}_\delta}\Big ] \longrightarrow 0 \qquad \text{ as } \quad \delta\downarrow 0 \,.
\]
Moreover, the bound \eqref{bound} dominates uniformly in $q$ the integrand and we conclude using the dominated convergence theorem (recall that $\nn^B$ is finite on $\mathbf{A}_\delta^c$) that
\[
\lim_{q\downarrow 0} \nn^B\bigg [ \Big (\int_0^{\ell} \exp\Big(- A_t^{\sca,q} +i\mu A_t^{f,q}\Big)\dr A_t^{\sca,q} \Big) \mathbf{1}_{\mathbf{A}_\delta^c}\bigg ] 
=  \nn^B\bigg [ \Big (\int_0^{\ell} \exp\Big(- A_t^{\beta} +i\mu A_t^{\alpha}\Big)\dr A_t^{\beta} \Big) \mathbf{1}_{\mathbf{A}_\delta^c}\bigg ].
\]
From this, we get that
\[
\lim_{q\downarrow 0}\nn^B\bigg [ \int_0^{\ell} \exp\Big(- A_t^{\sca,q} +i\mu A_t^{f,q}\Big)\dr A_t^{\sca,q} \bigg ]  
=\nn^B\bigg [ \int_0^{\ell} \exp\Big(-cA_t^{\beta} +i\mu A_t^{\alpha}\Big)\dr A_t^{\beta} \bigg ],
\]
which completes the proof of Step 2. 
\qed

\subsection{Proof of Proposition \ref{prop_brow_null_rec}}
In this section, we prove Proposition \ref{prop_brow_null_rec}. We let the reader recall Assumptions~\ref{assump_phi_beta_ito} and~\ref{assump_gaus_ito} (which corresponds to the case~$\alpha=2$). 
Our goal is to show that Assumption \ref{assum_gauss} holds, \textit{i.e.}\ that $\Phi$ is regularly varying at $0$ with index $\beta \in (0,1)$, that $\nn(\f) =0$ and that $\nn(\f^2) < \infty$. First, as in the proof of Proposition \ref{prop_assump_true}, we can show that, by setting $b(h) = h^{1/\beta} / \Lambda_\sca(1/h)$, the rescaled process $(b(h)\tau_{t/h})_{t\geq0}$ converges as $h\to0$ toward a $\beta$-stable subordinator. This entails that $\Phi$ is an asymptotic inverse of $b$ (see Remark~\ref{rem:bPhi}) and in particular that it is regularly varying with index $\beta$. 
 
 Next, we apply Lemma \ref{lemma_moy} with $g = \bm{1}$. This tells us that $\mm^f(\bm{1}) = \nn(\int_0^\ell f(\varepsilon_s) \dr s) = \nn(\f)$. Since $\mm^f(\bm{1}) = \lim_{x\to\infty}(\mm^f(x) - \mm^f(-x)) = 0$, it comes that $\nn(\f) = 0$. It remains to show that $\nn(\f^2) < \infty$. Recalling that the Lévy measure of $(Z_t)_{t\geq0}$ is $\nn(\f \in \dr x)$, this is equivalent to having $\mathbb{E}[Z_t^2] < \infty$ for any $t\geq0$, see for instance Sato \cite[Thm. 25.3]{ken1999levy}. We will show that the condition $\mm^f\in \mathrm{L}^2(\dr x)$ implies the latter one (it is actually equivalent). To do so, let us first define the function $g:\bbR \to \bbR$ as
 \[
  g(x) = \int_0^x \sca'(v)\int_v^{\infty} f(u) \mm(\dr u) \dr v.
 \]
Making two changes of variables, it holds that
\[
 g(\sca^{-1}(x)) = \int_0^x h(v) \dr v ,\quad 
 \text{ where }\quad h(v) = \int_v^{\infty} f\circ\sca^{-1}(u) \mm^\sca(\dr u) = \mm^f(\infty) - \mm^f(v) .
\]
Since $\mm^f$ in non-decreasing on $\bbR_+$ and non-increasing on $\bbR_-$, $g\circ \sca^{-1}$ is a difference of two convex function and we can apply the Itô-Tanaka formula, see \cite[Ch. VI Thm. 1.5]{ry}, which tells us that
\[
 g(\sca^{-1}(B_t)) = \int_0^t h(B_s)\dr B_s - \frac{1}{2}\int_{\bbR} f(\sca^{-1}(x))L_t^x \mm^\sca(\dr x) \,.
\]
Substituting $\rho_t$ in the preceding equation, and using the occupation time formula from item~\textit{(i)} of Proposition \ref{prop_ok_gen}, we see that (with a change of variable $y=\sca^{-1}(x)$)
\[
 g(X_t) = \int_0^{\rho_t}h(B_s)\dr B_s - \frac{1}{2}\int_0^t f(X_s) \dr s.
\]
We now substitute $\tau_t$ in the preceding equation, and using that $X_{\tau_t} = 0$ so $g(X_{\tau_t})=0$ and $\rho_{\tau_t} = \tau_t^B$ (see Proposition~\ref{prop_ok_gen}-\textit{(iv)}), we get that
\[
 Z_t = \int_0^{\tau_t} f(X_s) \dr s = 2\int_0^{\tau_t^B}h(B_s)\dr B_s.
\]
Let us set for any fixed $t\geq0$, $(M_u^t)_{u\geq0} = (\int_0^{u\wedge\tau_t^B}h(B_s)\dr B_s)_{u\geq0}$ which is a (centered) local martingale. Its quadratic variation is such that, for any $u\geq0$,
\[
 \mathbb{E}\left[\langle M^t\rangle_u\right] = \mathbb{E}\bigg[\int_{\bbR}L^x_{u\wedge\tau_t^B} \, h(x)^2 \dr x\bigg] \leq \int_{\bbR}\mathbb{E}\big[L_{\tau_t^B}^x\big][h(x)]^2 \dr x = t\int_{\bbR} h(x)^2 \dr x,
\]
where in the last equality, we used that by Ray-Knight's theorem, $\mathbb{E}[L_{\tau_t^B}^x] = t$ for any $x\in \bbR$. Since $h\in\mathrm{L}^2(\dr x)$, $(M_u^t)_{u\geq0}$ is a martingale bounded in $\mathrm{L}^2(\bbP)$, \textit{i.e.}\ $\sup_{u\geq0}\mathbb{E}[(M_u^t)^2] < \infty$. Therefore, it converges in $\mathrm{L}^2(\bbP)$, which entails that
\[
 \mathbb{E}\bigg[\bigg(\int_0^{\tau_t^B}h(B_s)\dr B_s\bigg)^2\bigg] = \mathbb{E}\bigg[\int_0^{\tau_t^B}h(B_s)^2\dr s\bigg] = t\int_{\bbR} h(x)^2 \dr x.
\]
We have shown that for any $t\geq0$, $\mathbb{E}[Z_t^2] = 4t\int_{\bbR}h(x)^2 \dr x <+\infty$, which completes the proof.
\qed

\section{Hitting time of zero}\label{sec:hitting}

In this section, we will assume for simplicity that the process $(X_t)_{t\geq0}$ is a regular diffusion valued in an open interval $\mathrm{I} = (a,b)$ with $a \in (-\infty,0)\cup\{-\infty\}$ and $b \in (0,\infty)\cup\{\infty\}$. We will consider the process $(\zeta_t, X_t)_{t\geq0}$ as a strong Markov process. For a pair $(z,x)\in\mathbb{R}\times \mathrm{I}$, we will denote by $\bm{\mathrm{P}}_{(z,x)}$ the law of $(\zeta_t, X_t)_{t\geq0}$ when started at $(z,x)$ i.e. the law of $(z + \int_0^t f(X_s)\dr s, X_t)_{t\geq0}$ under $\mathbb{P}_x$, where $\mathbb{P}_x$ denotes the law of $(X_t)_{t\geq0}$ started at $x$.

\medskip
The aim of this section is to describe the asymptotic behavior of $\bm{\mathrm{P}}_{(z,x)}(T_0 > t)$ as $t\to\infty$ when $(z,x)\neq(0,0)$. We will naturally place ourselves under Assumption \ref{assump_pos} or \ref{assump_null}. As a rather classical application, we will identify a harmonic function for the killed process, which will in turn enable us to construct the additive functional conditionned to stay negative through Doob's h-transform. 

First, we recall that a regular diffusion is a continuous strong Markov process such that, if set $\eta_x = \inf\{t > 0, \: X_t = x\}$, then for any $(x,y)\in \mathrm{I}^2$, $\mathbb{P}_x(\eta_y < \infty) > 0$. It is well-know that regular diffusions are space and time changed Brownian motion. First, there exists a continuous and increasing function $\sca : \mathrm{I}\to\mathbb{R}$ such that $(\sca(X_t))_{t\geq0}$ is a local martingale, see Kallenberg \cite[Chapter 23, Theorem 23.7]{kallenberg2006foundations}. We will assume that $\sca$ is such that $\sca(\mathrm{I}) = \mathbb{R}$ and will denote by $\sca^{-1}$ its inverse, defined on $\mathbb{R}$ and valued in $\mathrm{I}$. Let $(W_t)_{t\geq0}$ be a Brownian on some probability space and denote by $(L_t^x)_{t\geq0, x\in\mathbb{R}}$ its family of local times. Then there exists an increasing function $\mm^\sca: \mathbb{R}\to\mathbb{R}$ such that, if we set
\[
 A_t = \int_{\mathbb{R}}L_t^x \mm^\sca(\dr x) \quad \text{and} \quad \rho_t = \inf\{s\geq0, \: A_s > t\},
\]

\noindent
then for any $x\in\mathrm{I}$, the probability measure $\mathbb{P}_x$ coincides with the law of $(\sca^{-1}(W_{\rho_t}))_{t\geq0}$ with $(W_t)_{t\geq0}$ started at $\sca(x)$. We refer to Kallenberg \cite[Chapter 23, Theorem 23.9]{kallenberg2006foundations} for more details. We will first show the following lemma. Let us also define 
\[ 
\mathrm{I}_+^*:=\mathrm{I}\cap(0,\infty), \quad \mathrm{and} \ \ \mathrm{I}_-^* = \mathrm{I}\cap(-\infty, 0).
\]
\begin{lemma}\label{lemma_max_exc}
 Recall that we denote by $\f(\gep)=\int_0^\ell f(\gep_s) \dr s $ for some excursion $\epsilon\in\mathcal{E}$ of $(X_t)_{t\geq 0}$.  The two following assertions hold.
 \begin{enumerate}[label=(\roman*)]
  \item For any $x\in \mathrm{I}_+^*$, we have
  \[
   \nn(M > x, \f < 0) = 1 / (2|\sca(x)|) \quad \text{and} \quad \nn(M > x, \f > 0) = 1 / (2|\sca(x)|),
  \]

  \noindent
  where $M = M(\varepsilon) = \sup_{s\geq0}|\varepsilon_s|$.
  
  \item For any $x\in \mathrm{I}_+^*$ (respectively any $x\in \mathrm{I}_-^*$), the law of $(X_{g_{\eta_x} + t})_{t\in[0,d_{\eta_x} - g_{\eta_x}]}$ under $\mathbb{P}$ is exactly $\nn(\cdot | M > x, \:\f > 0)$ (respectively $\nn(\cdot | M > x, \:\f < 0)$).
 \end{enumerate}
\end{lemma}

\begin{proof}
 Remember first that, since excursions have constant sign, $\f(\varepsilon) > 0$ if and only if $\varepsilon$ is a positive excursion. By Proposition \ref{pushforward_brownian} and since $A_t$ is increasing  it comes that for any $x\in\mathrm{I}_+^*$, we have $\nn(M > x, \f > 0) = \nn_B^+(M > \sca(x))$ (respectively $\nn(M > x, \f < 0) = \nn_B^-(M > \sca(x))$), where $\nn_B^+$ is the Brownian excursion measure restricted to the set of positive excursions (respectively $\nn_B^-$ is the Brownian excursion measure restricted to the set of negative excursions). The first result follows immediately since $\nn_B^+(M > x) = \nn_B^-(M > x)  = 1/(2x)$ for any $x > 0$.
 
 \medskip
 Regarding the second point, let us denote for $x\in\mathrm{I}_+^*$, $U_x^+ = \{\varepsilon\in\mathcal{E},\: M(\varepsilon) > x,\: \f(\varepsilon) > 0\}$. Recall that $(e_t)_{t\geq0}$ denote the excursion process and set $T_{U_x^+} = \inf\{t > 0, e_t \in U_x^+\}$. Clearly, we have $e_{T_{U_x^+}} = (X_{g_{\eta_x} + t})_{t\in[0,d_{\eta_x} - g_{\eta_x}]}$ and since $\nn(U_x^+) < \infty$ by the first point, the result follows from \cite[Chapter XII Lemma 1.13]{ry} which tells us that for any measurable set $\Gamma$, we have
\[
 \nn(\Gamma \cap U_x^+) = \mathbb{P}(e_{T_{U_x^+}} \in \Gamma) \nn(U_x^+).
\]

\noindent
The same proof holds for $x\in \mathrm{I}_-^*$.
\end{proof}

\medskip
Let us consider the set $\mathrm{H} := [(-\infty,0)\times \mathrm{I}]\cup [\{0\}\times \mathrm{I}_-^*]$. We will only consider starting points $(z,x)\in \mathrm{H}$ so that $\bm{\mathrm{P}}_{(z,x)}(T_0 > 0) = 1$ and the process $(\zeta_{t\wedge T_0})_{t\geq0}$ is negative under $\bm{\mathrm{P}}_{(z,x)}$. We recall that $\mathcal{V}$ denotes the renewal function and that for any $x \geq0$, $\mathcal{V}(x) = \int_0^\infty\mathbb{P}(H_t \leq x) \dr t$. Finally, we set $\eta_0 = \inf\{t > 0, X_t = 0\}$, the first hitting time of zero of $(X_t)_{t\geq0}$, and we consider the positive function $h:\mathrm{H}\to\mathbb{R}_+\cup\{\infty\}$ defined for every $(z,x)\in \mathrm{H}$ as 
\[
 h(z,x) = \bm{\mathrm{E}}_{(z,x)}\left[\mathcal{V}(-\zeta_{\eta_0})\bm{1}_{\{\zeta_{\eta_0} < 0\}}\right].
\]

\begin{proposition}
 For any $(z,x)\in\mathrm{H}$, $h(z,x) < \infty$.
\end{proposition}

\noindent
This result is not obvious at first sight. However since $\zeta_{\eta_0}$ is a piece a integrated excursion, its law is directly linked with the Lévy measure of $(Z_t)_{t\geq0}$, and so is the function $\mathcal{V}$. As we shall see, this link is given by the équations amicales of Vigon. 

\begin{proof}
 Let $(z,x)\in\mathrm{H}$ and remember that
\[
 h(z,x) = \mathbb{E}_x\left[\mathcal{V}\Big(-z - \int_0^{\eta_0}f(X_s) \dr s\Big)\bm{1}_{\{z + \int_0^{\eta_0}f(X_s) \dr s < 0\}}\right].
\]

\noindent
First case: $z < 0$ and $x \geq 0$. If $x = 0$, then $\eta_0 = 0$ a.s. and $h(z,x) = \mathcal{V}(-z)$. If $x > 0$, then $\int_0^{\eta_0}f(X_s) \dr s > 0$ a.s. and since $\mathcal{V}$ is non-decrasing, $h(z,x) \leq \mathcal{V}(-z)$.

\medskip\noindent
Second case: $z \leq 0$ and $x < 0$. Then in this case $\int_0^{\eta_0}f(X_s) \dr s < 0$ a.s. By the strong Markov property, and since $(X_t)_{t\geq0}$ is continuous, the law of $\int_0^{\eta_0}f(X_s) \dr s$ under $\mathbb{P}_x$ is equal to the law of $\int_{\eta_x}^{d_{\eta_x}}f(X_s) \dr s$ under $\mathbb{P}_0$. But since $\int_{\eta_x}^{d_{\eta_x}}f(X_s) \dr s \geq \int_{g_{\eta_x}}^{d_{\eta_x}}f(X_s) \dr s$ and by Lemma \ref{lemma_max_exc}-\textit{(ii)}, the law of $\int_{g_{\eta_x}}^{d_{\eta_x}}f(X_s) \dr s$ under $\mathbb{P}_0$ is equal to the law of $\f(\gep)$ under $\nn(\cdot | M > x, \:\f < 0)$ we get the following bound:
\[
 h(z,x) \leq \frac{\nn\left(\mathcal{V}(-z - \f)\bm{1}_{\{M > x\}}\bm{1}_{\{\f < 0\}}\right)}{\nn(M > x, \:\f < 0)}.
\]

\noindent
Since $\nn(M > x, \:\f< 0) < \infty$, it remains to show that $\nn(\mathcal{V}(-z - \f)\bm{1}_{\{M > x\}}\bm{1}_{\{\f < 0\}}) < \infty$. We split this quantity in two parts:
\begin{align*}
 \nn\left(\mathcal{V}(-z - \f)\bm{1}_{\{M > x\}}\bm{1}_{\{\f < 0\}}\right) = & \:\nn\left(\mathcal{V}(-z - \f)\bm{1}_{\{M > x\}}\bm{1}_{\{-1 \leq \f <0\}}\right) \\
  & + \nn\left(\mathcal{V}(-z - \f)\bm{1}_{\{M > x\}}\bm{1}_{\{\f < -1\}}\right).
\end{align*}

\noindent
The first term on the right-hand-side is smaller than $\mathcal{V}(-z+1)\nn(M > x) < \infty$ (since $\mathcal{V}$ is non-decreasing). Regarding the second term, we first recall that $\mathcal{V}$ is subadditive, see for instance \cite[Chapter III]{b}, i.e. for any $x,y\geq0$, $\mathcal{V}(x+y) \leq \mathcal{V}(x) + \mathcal{V}(y)$, so that we can write
\[
 \nn\left(\mathcal{V}(-z - \f)\bm{1}_{\{M > x\}}\bm{1}_{\{\f < -1\}}\right) \leq \mathcal{V}(-z + 1/2)\nn(M > x) + \nn\left(\mathcal{V}(-1/2 - \f)\bm{1}_{\{\f < -1\}}\right).
\]

\noindent
We finally show that $\nn(\mathcal{V}(-1/2 - \f)\bm{1}_{\{\f < -1\}}) < \infty$. Recall now that the Lévy measure $\nu$ of $(Z_t)_{t\geq0}$ is $\nu(\dr u) = \nn(\f \in \dr u)$ so that
\begin{align*}
 \nn\left(\mathcal{V}(-1/2 - \f)\bm{1}_{\{\f < -1\}}\right) = & \int_{(-\infty, - 1)}\mathcal{V}(-1/2 - u) \nu(\dr u) \\
 = & \int_{\mathbb{R}}\int_{\mathbb{R}_+}\bm{1}_{\{u < -1\}}\bm{1}_{\{y \leq -1/2 - u\}}\mathcal{V}(\dr y)\nu(\dr u),
\end{align*}

\noindent
where $\mathcal{V}(\dr y)$ stands for the Stieltjes measure associated to the non-increasing function $\mathcal{V}$. We then have the following bound
\[
 \nn\left(\mathcal{V}(-1/2 - \f)\bm{1}_{\{\f < -1\}}\right) \leq \int_{\mathbb{R}}\int_{\mathbb{R}_+}\bm{1}_{\{y + u \leq -1/2\}}\nu(\dr u)\mathcal{V}(\dr y) = (\nu \ast \mathcal{V})((- \infty, -1/2]),
\]

\noindent
where $\nu \ast \mathcal{V}$ is the convolution of the measures $\nu$ and $\mathcal{V}$. The équations amicales of Vigon, see for instance \cite{vigon2002votre, vigon2002simplifiez}, states that the measure $\nu\ast\mathcal{V}$ coincides on $(-\infty,0)$ with the Lévy measure of the dual ladder height process $(-\hat{H}_t)_{t\geq0}$. Therefore $(\nu \ast \mathcal{V})((- \infty, -1/2]) < \infty$, which completes the proof.
\end{proof}

We will now show the following technical result, which describes the asymptotic behavior of the Laplace transform $\bm{\mathrm{P}}_{(z,x)}(T_0 > t)$ and recall that $e=e(q)$ is an independent exponential random variable of parameter $q$.

\begin{proposition}\label{prop_domi_return}
 Grant Assumption \ref{assump_pos} or \ref{assump_null}. There exists a constant $c_0 > 0$ such that for any $(z,x)\in \mathrm{H}$, we have
 \[
  \bm{\mathrm{P}}_{(z,x)}(T_0 > e) \sim c_0h(z,x)\kappa(0,q,0)  \quad \text{as }q\to0.
 \]

\noindent
Moreover, there exists a constant $M > 0$ such that for any $(z,x)\in \mathrm{H}$, for any $q\in(0,1)$,
\[
 \bm{\mathrm{P}}_{(z,x)}(T_0 > e) \leq M\kappa(0,q,0)(|\sca(x)| + h(z,x)).
\]
\end{proposition}

This proposition, combined with the Tauberian theorem, the monotone density theorem, Theorem \ref{equivalence_theorem} and Proposition \ref{prop_rec_null_kappa} leads to the following theorem.

\begin{theorem}\label{main_thm_hitting}
 Grant Assumption \ref{assump_pos} or \ref{assump_null}. Then there exists a slowly varying function $\Lambda$ such that for any $(z,x) \in \mathrm{H}$,
 \[
  \bm{\mathrm{P}}_{(z,x)}(T_0 > t) \sim h(z,x)\Lambda(t)t^{-\rho} \quad \text{as }t\to\infty,
 \]
\noindent
where $\rho \in(0,1)$ is either given by Assumption \ref{assump_pos}, or $\rho = \beta \mathbb{P}(Z_t^0 \geq0)$ if Assumption \ref{assump_null} is in force.
\end{theorem}

\begin{proof}[Proof of Proposition \ref{prop_domi_return}]
 \textit{Step 1:} We first focus on the quantity $\mathbb{P}_x(\eta_0 > e)$ where $x\in \mathrm{I}$. We will first show that there exists a constant $C > 0$ such that for any $q\in(0,1)$ and any $x\in \mathrm{I}$,
 \begin{equation}\label{eq_lim_eta0}
 \mathbb{P}_x(\eta_0 > e) \leq C \kappa(0,q,0) |\sca(x)|.
 \end{equation}
 
\noindent
Let $x\in\mathrm{I}_+^*$ and remember that $\eta_x = \inf\{t>0, X_t = x\}$, and that by the strong Markov property the law of $\eta_0$ under $\mathbb{P}_x$ is equal to the law of $d_{\eta_x} - \eta_x$ under $\mathbb{P}_0$. Therefore we get
\[
 \mathbb{P}_x(\eta_0 > e) \leq \mathbb{P}_0(d_{\eta_x} - g_{\eta_x} > e) = q\int_0^{\infty} \e^{-qt}\mathbb{P}_0(d_{\eta_x} - g_{\eta_x} > t)\dr t.
\]

\noindent
By Lemma \ref{lemma_max_exc}-\textit{(ii)}, the law of $d_{\eta_x} - g_{\eta_x}$ under $\mathbb{P}_0$ is the law of $\ell$ under $\nn(\cdot | M > x, \f > 0)$ and by Lemma \ref{lemma_max_exc}-\textit{(i)}, $\nn(M > x, \f > 0) = 1/(2|\sca(x)|)$ so that

\begin{equation}\label{inter_domi}
 \mathbb{P}_x(\eta_0 > e) \leq  \frac{|\sca(x)|}{2}\nn(1-\e^{-q\ell})
\end{equation}

\noindent
This bound also holds for $x\in\mathrm{I}_-^*$ by Lemma \ref{lemma_max_exc} again, and for $x = 0$ since $\mathbb{P}_x(\eta > e) = 0$. 
Remember now that $\nn(1-\e^{-q\ell}) = \Phi(q) - \mathrm{m}q$ and that under Assumption \ref{assump_pos} or \ref{assump_null}, $\Phi$ is regularly varying at $0$ with index $\beta \in(0,1]$ ($\beta = 1$ under Assumption~\ref{assump_pos} and $\beta\in(0,1)$ under Assumption \ref{assump_null}). 
By Theorem \ref{equivalence_theorem} and Proposition \ref{prop_rec_null_kappa}, $q\mapsto \kappa(0,q,0)$ is regularly varying at $0$ with some index $\rho \in (0, \beta)$. It is clear that $\nn(1-\e^{-q\ell}) / \kappa(0,q,0) \to 0$ as $q\to0$, which shows that \eqref{eq_lim_eta0} holds and that for any $x\in\mathrm{I}$,
\begin{equation}\label{lim_eta}
 \frac{\mathbb{P}_x(\eta_0 > e)}{\kappa(0,q,0)} \to 0 \quad \text{as }q\to0.
\end{equation}

\noindent
This completes the first step.

\medskip\noindent
\textit{Step 2:} We now focus on the quantity $\bm{\mathrm{P}}_{(z,x)}(T_0 - \eta_0 > e)$ for $(z,x)\in\mathrm{H}$. We first remark that $\bm{\mathrm{P}}_{(z,x)}$-almost surely, $T_0 > \eta_0$ if and only if $\zeta_{\eta_0} < 0$. First, if $x < 0$, then $\bm{\mathrm{P}}_{(z,x)}$-almost surely, $(\zeta_t)_{t\geq0}$ is decreasing on $[0,\eta_0]$ so that $\bm{\mathrm{P}}_{(z,x)}(\zeta_{\eta_0} < 0) = \bm{\mathrm{P}}_{(z,x)}(T_0 > \eta_0) = 1$. Next, if $x \geq0$, then $\bm{\mathrm{P}}_{(z,x)}$-almost surely, $(\zeta_t)_{t\geq0}$ is non-increasing on $[0,\eta_0]$ and thus, we see that if $\zeta_{\eta_0} \geq 0$, then $T_0 \leq \eta_0$ whereas if $\zeta_{\eta_0} < 0$, then $T_0 > \eta_0$.

\medskip
We define the processes $(\Bar{\zeta}_t, \Bar{X}_t)_{t\geq0} = (\zeta_{t+\eta_0} - \zeta_{\eta_0}, X_{t+ \eta_0})_{t\geq0}$ and $\Bar{\xi}_t = \sup_{s\in[0,t]}\bar{\zeta_s}$. By the strong Markov property again, $(\Bar{\zeta}_t, \Bar{X}_t)_{t\geq0}$ and $\zeta_{\eta_0}$ are independent under $\bm{\mathrm{P}}_{(z,x)}$. Moreover the law of $(\Bar{\zeta}_t, \Bar{X}_t)_{t\geq0}$ under $\bm{\mathrm{P}}_{(z,x)}$ is equal to the of $(\zeta_t, X_t)_{t\geq0}$ under $\bm{\mathrm{P}}_{(0,0)} = \mathbb{P}$. We see that $\bm{\mathrm{P}}_{(z,x)}$-almost surely, $T_0 - \eta_0 > e$ if and only if $\xi_e < - \zeta_{\eta_0}$ and $\zeta_{\eta_0} < 0$. By independence, we get
\begin{equation}\label{aaa}
 \bm{\mathrm{P}}_{(z,x)}(T_0 - \eta_0 > e) = \int_{(-\infty, 0)}\mathbb{P}(\xi_e < -u)\bm{\mathrm{P}}_{(z,x)}(\zeta_{\eta_0} \in \dr u).
\end{equation}

\noindent
We will first show that for any $q\in(0,1)$, for any $(z,x) \in \mathrm{E}$, we have
\begin{equation}\label{ineq_h_domi}
 \bm{\mathrm{P}}_{(z,x)}(T_0 - \eta_0 > e) \leq h(z,x) \kappa(0,q,0).
\end{equation}

\noindent
Recall from \eqref{split} that $\xi_e= \xi_{g_e}+ \max( \Delta_e ,0)$ and
that $\xi_{g_e}$ and $\Delta_e$ are independent.  Since moreover $\mathbb{P}(\xi_{g_e} < -u) = \kappa(0,q,0)\mathcal{V}_q(-u) \leq \kappa(0,q,0)\mathcal{V}(-u)$ we get the following bound:
\[
 \mathbb{P}(\xi_e < -u) \leq \mathbb{P}(\xi_{g_e} < -u)\mathbb{P}(\Delta_e < -u) \leq \kappa(0,q,0)\mathcal{V}(-u),
\]

\noindent
which, combined with \eqref{aaa}, leads to \eqref{ineq_h_domi}. We will now show that there exists a constant $c_0 > 0$ such that for any $(z,x) \in \mathrm{E}$,
\begin{equation}\label{conv_interm}
 \bm{\mathrm{P}}_{(z,x)}(T_0 - \eta_0 > e) \sim c_0 h(z,x) \kappa(0,q,0) \quad \text{as }q\to0.
\end{equation}

\noindent
It is shown in the proofs of Theorem \ref{main_theorem} that $\mathbb{P}(\xi_e < -u) \sim c_0\mathcal{V}(-u)\kappa(0,q,0)$ as $q\to0$, where $c_0 = \lim_{q\to0}\mathbb{P}(\Delta_e \leq 0)$. Since $\mathbb{P}(\xi_e < -u) \leq \kappa(0,q,0)\mathcal{V}(-u)$, it is clear that we can apply the dominated convergence theorem in \eqref{aaa} to deduce that \eqref{conv_interm} holds. This completes the second step.

\medskip\noindent
\textit{Step 3:} We conclude. First, we have by \eqref{conv_interm} that for any $(z,x)\in\mathrm{H}$,
\[
 c_0 h(z,x) = \liminf_{q\to0}\frac{\bm{\mathrm{P}}_{(z,x)}(T_0 - \eta_0 > e)}{\kappa(0,q,0)} \leq \liminf_{q\to0} \frac{\bm{\mathrm{P}}_{(z,x)}(T_0 > e)}{\kappa(0,q,0)}.
\]

\noindent
Let $\delta \in(0,1)$ and write
\begin{align*}
 \bm{\mathrm{P}}_{(z,x)}(T_0 > e) = & \bm{\mathrm{P}}_{(z,x)}(T_0 > e, T_0 - \eta_0 > \delta e) + \bm{\mathrm{P}}_{(z,x)}(T_0 > e, T_0 - \eta_0 \leq \delta e) \\
  & \leq \bm{\mathrm{P}}_{(z,x)}(T_0 - \eta_0 > \delta e) + \mathbb{P}_x(\eta_0 > (1 - \delta) e).
\end{align*}

\noindent
By \eqref{lim_eta}, and since $q\mapsto \kappa(0,q,0)$ is regularly varying at $0$ with index $\rho \in(0,1)$, the second term divided by $\kappa(0,q,0)$ vanishes as $q\to0$. Since $\delta e$ is an exponential random variable of parameter $q/ \delta$, we have by \eqref{conv_interm} that
\[
 \bm{\mathrm{P}}_{(z,x)}(T_0 - \eta_0 > \delta e) \sim c_0\delta^{-\rho} h(z,x) \kappa(0,q,0) \quad \text{as }q\to0.
\]

\noindent
Therefore, we get that
\[
 \limsup_{q\to0}\frac{\bm{\mathrm{P}}_{(z,x)}(T_0 > e)}{\kappa(0,q,0)} \leq c_0\delta^{-\rho} h(z,x).
\]

\noindent
Letting $\delta\to1$ shows the first part of the proposition. Regarding the second part, it comes by \eqref{eq_lim_eta0} that there exists a constant $C > 0$ such that for any $q\in(0,1)$ and any $(z,x) \in \mathrm{H}$, we have $\mathbb{P}_x(\eta_0 > (1 - \delta) e) \leq M|\sca(x)| \kappa(0,q,0)$. Then, by \eqref{ineq_h_domi}, we see that for $q\in(0,1)$ and any $(z,x) \in \mathrm{H}$, we have
\[
 \bm{\mathrm{P}}_{(z,x)}(T_0 - \eta_0 > \delta e) \leq h(z,x) \kappa(0,q/\delta, 0) \leq h(z,x) \kappa(0,q,0) \times \sup_{q\in(0,1)} \frac{\kappa(0,\delta / q, 0)}{\kappa(0,q,0)}.
\]

\noindent
This finishes the proof.
\end{proof}

From this proposition, we are able to show that the function $h$ is harmonic and thus, we are able to define the additive functional conditionned to stay negative.
\begin{corollary}\label{coro_harmonic}
 Grant Assumption \ref{assump_pos} or \ref{assump_null} and assume that for any $x\in\mathrm{I}$ and any $t\geq0$, $\mathbb{E}_x[|\sca(X_t)|] < \infty$. Then the function $h$ is harmonic for the killed process $(\zeta_{t\wedge T_0}, X_{t\wedge T_0})_{t\geq0}$, i.e. for any $(z,x)\in\mathrm{H}$ and any $t\geq0$, we have
 \[
  \bm{\mathrm{E}}_{(z,x)}\left[h(\zeta_t, X_t)\bm{1}_{\{T_0 > t\}}\right] = h(z,x).
 \]
\end{corollary}

We emphasize that the assumption $\mathbb{E}_x[ \vert \sca(X_t)\vert ] < \infty$ for any $x\in\mathrm{I}$ and any $t\geq0$ is satisfied by a large class of processes (as for example Brownian motion or Ornstein-Uhlenbeck processes). However one can easily find a recurrent process for which it is not satisfied.

\begin{proof}
 \textit{Step 1:} We first show that for any $(z,x)\in \mathrm{H}$, $\bm{\mathrm{E}}_{(z,x)}[h(\zeta_t, X_t)\bm{1}_{\{T_0 > t\}}] \leq  h(z,x) < \infty$. This is a direct application of Fatou's Lemma, Proposition \ref{prop_domi_return} and the Markov property. Let $t\geq0$ and note that as long as $T_0 > t$, $(\zeta_t, X_t) \in \mathrm{H}$. We have
 \begin{align*}
  \bm{\mathrm{E}}_{(z,x)}\left[h(\zeta_t, X_t)\bm{1}_{\{T_0 > t\}}\right] = & \bm{\mathrm{E}}_{(z,x)}\left[\liminf_{q\to0} \frac{\bm{\mathrm{P}}_{(\zeta_t,X_t)}(T_0 >  e)}{c_0 \kappa(0,q,0)}\bm{1}_{\{T_0 > t\}}\right] \\
   & \leq \liminf_{q\to0}\bm{\mathrm{E}}_{(z,x)}\left[\frac{\bm{\mathrm{P}}_{(\zeta_t,X_t)}(T_0 >  e)}{c_0 \kappa(0,q,0)}\bm{1}_{\{T_0 > t\}}\right] \\
   & \leq  \liminf_{q\to0} \frac{\bm{\mathrm{P}}_{(z,x)}(T_0 >  e + t)}{c_0 \kappa(0,q,0)}.
 \end{align*}

\noindent
For any $(z,x)\in\mathrm{H}$ and for any $t\geq0$, we have
\[
 \bm{\mathrm{P}}_{(z,x)}(T_0 >  e + t) = q\int_0^{\infty}\e^{-qs}\bm{\mathrm{P}}_{(z,x)}(T_0 >  s + t)\dr s = \bm{\mathrm{E}}_{(z,x)}\left[(1 - \e^{-q(T_0 - t)})\bm{1}_{\{T_0 > t\}}\right].
\]

\noindent
Since $\bm{\mathrm{P}}_{(z,x)}(T_0 >  e) = \bm{\mathrm{E}}_{(z,x)}[1 - \e^{-qT_0}]$, we get that
\begin{align*}
 \bm{\mathrm{P}}_{(z,x)}(T_0 >  e) - \bm{\mathrm{P}}_{(z,x)}(T_0 >  e + t) = & \bm{\mathrm{E}}_{(z,x)}\left[(1 - \e^{-qT_0})\bm{1}_{\{T_0 \leq t\}}\right] + \bm{\mathrm{E}}_{(z,x)}\left[\e^{-qT_0}(\e^{qt} - 1)\bm{1}_{\{T_0 > t\}}\right] \\
  & \leq qt + \e^{qt} - 1.
\end{align*}

\noindent
At this point we conclude that for any $t\geq0$ and for any $(z,x) \in \mathrm{H}$, $\bm{\mathrm{P}}_{(z,x)}(T_0 >  e + t) \sim c_0 h(z,x) \kappa(0,q,0)$ as $q\to0$ and thus
\[
 \bm{\mathrm{E}}_{(z,x)}\left[h(\zeta_t, X_t)\bm{1}_{\{T_0 > t\}}\right] \leq h(z,x).
\]
\noindent
\textit{Step 2:} We now use the dominated convergence theorem. Indeed, we have
\[
 \bm{\mathrm{E}}_{(z,x)}\left[h(\zeta_t, X_t)\bm{1}_{\{T_0 > t\}}\right] = \bm{\mathrm{E}}_{(z,x)}\left[\lim_{q\to0}\frac{\bm{\mathrm{P}}_{(\zeta_t,X_t)}(T_0 >  e)}{c_0 \kappa(0,q,0)}\bm{1}_{\{T_0 > t\}}\right],
\]

\noindent
and by Proposition \ref{prop_domi_return}, there exists a constant $M > 0$ such that for any $t\geq 0$ and any $q\in(0,1)$,
\[
 \frac{\bm{\mathrm{P}}_{(\zeta_t,X_t)}(T_0 >  e)}{c_0 \kappa(0,q,0)}\bm{1}_{\{T_0 > t\}} \leq M\left(|\sca(X_t)| + h(\zeta_t, X_t)\right)\bm{1}_{\{T_0 > t\}}.
\]

\noindent
By assumption $\mathbb{E}_x[\sca(X_t)] < \infty$ and by the first step, $\bm{\mathrm{E}}_{(z,x)}[h(\zeta_t, X_t)\bm{1}_{\{T_0 > t\}}] < \infty$ so that we have by the strong Markov property and by the computations from the first step,
\[
 \bm{\mathrm{E}}_{(z,x)}\left[h(\zeta_t, X_t)\bm{1}_{\{T_0 > t\}}\right] = \lim_{q\to0} \frac{\bm{\mathrm{P}}_{(z,x)}(T_0 >  e + t)}{c_0 \kappa(0,q,0)} = h(z,x).
\]

\noindent
This completes the proof.
\end{proof}

Let $(\mathcal{F}_t)_{t\geq0}$ be the filtration generated by $(X_t)_{t\geq0}$. For $(z,x)\in\mathrm{H}$, we introduce a new probability measure $\bm{\mathrm{Q}}_{(z,x)}$ such that for every $t\geq0$, for every $A\in\mathcal{F}_t$,
\[
 \bm{\mathrm{Q}}_{(z,x)}(A) = \frac{1}{h(z,x)}\bm{\mathrm{E}}_{(z,x)}\left[h(\zeta_t, X_t) \bm{1}_{A\cap\{T_0 > t\}}\right].
\]

\noindent
The preceding proposition ensures that this measure is indeed a probability measure. It is in fact the law of $(\zeta_t, X_t)_{t\geq0}$ where $(\zeta_t)_{t\geq0}$ is conditioned to remain negative. Using similar arguments as in the previous proof, we could easily show the following result, which justifies the terminology.

\begin{proposition}
\label{prop:conditioned}
For any $(z,x)\in\mathrm{H}$, for any $t\geq0$ and any $A\in\mathcal{F}_t$, we have
\[
 \bm{\mathrm{Q}}_{(z,x)}(A) = \lim_{q\to0} \bm{\mathrm{P}}_{(z,x)}(A | T_0 > e),
\]
\end{proposition}

\appendix

\section{Wiener--Hopf factorization}\label{appendix_wiener}

In this section, we derive the Wiener-Hopf factorization from Subsection \ref{section_WH}, \textit{i.e.}\ Theorem~\ref{WH}. We consider, in all generality, a bivariate Lévy process $(\tau_t, Z_t)_{t\geq0}$ with respect to some filtration $(\mathcal{F}_t)_{t\geq0}$, where $(\tau_t)_{t\geq0}$ is a subordinator. This section is very close to \cite[Ch.~VI]{b} and follows the same approach. However, our result requires some new arguments and we chose to establish properly Theorem~\ref{WH}.

\subsection{Preliminaries}
Let us recall briefly the notation introduced in Section~\ref{section_WH}.
Let $S_t= \sup_{[0,t]} Z_t$ the running supremum of $Z_t$ and consider the reflected process $(R_t)_{t\geq0} = (S_t - Z_t)_{t\geq0}$, which is a strong Markov process (see \cite[Prop.~VI.1]{b}) and posesses a local time $(L_t^{R})_{t\geq 0}$ at $0$. 
We denote by $(\sigma_t)_{t\geq0}$ its right-continuous inverse, and we define $(\sigma_t, \theta_t, H_t)_{t\geq0} := (\sigma_t, \tau_{\sigma_t}, S_{\sigma_t})_{t\geq0}$.

\begin{lemma}
\label{lem:trivariate}
The two following assertions hold.
\begin{enumerate}[label=(\roman*)]
 \item If $0$ is recurrent for the reflected process, then $(\sigma_t, \theta_t, H_t)_{t\geq0}$ is a trivariate subordinator.
 
 \item If $0$ is transient for the reflected process, then there exists some $\bm{q} >0$ such that $L_{\infty}^R$ has an exponential distribution with parameter $\bm{q}$. Moreover, the process $(\sigma_t, \theta_t, H_t)_{0 \leq t < L_{\infty}^R}$ is a trivariate subordinator killed at rate $\bm{q}$.
\end{enumerate}
\end{lemma}

\begin{proof}
The proof is essentially the same as in Kyprianou \cite[Ch. VI Thm 6.9]{Kyprianou14} and we will show the two items at once. Let $t\geq0$, we first place ourselves on the event $\{t< L_{\infty}^R\} = \{\sigma_t < \infty\}$ so that $\sigma_t$ is a finite stopping time.

The strong Markov property tells us that the process $(\tilde{\tau}_s, \tilde{Z}_s)_{s\geq0} = (\tau_{\sigma_t + s} - \tau_{\sigma_t}, Z_{\sigma_t + s} - Z_{\sigma_t})_{s\geq0}$ is a Lévy process independent of $\mathcal{F}_{\sigma_t}$. Then it is clear that the corresponding local time $(\tilde{L}_s)_{s\geq0}$ is such that for any $s\geq0$, $\tilde{L}_s = L_{\sigma_t + s}^R - t$ so that its right-continuous inverse $(\tilde{\sigma}_s)_{s\geq0}$ is $\tilde{\sigma}_s = \sigma_{t+s} - \sigma_t$. Moreover, since $S_{\sigma_t} = Z_{\sigma_t}$, it comes that for any $s\geq0$ 
\[
 \tilde{S}_s = \sup_{u\in[0,s]}\tilde{Z}_u = \sup_{u\in[0,s]}(Z_{\sigma_t + u} - Z_{\sigma_t}) = S_{\sigma_t + s} - S_{\sigma_t}.
\]

\noindent
This shows that, on the event $\{t< L_{\infty}^R\}$, the shifted process
\[
 (\tilde{\sigma}_s, \tilde{\theta}_s, \tilde{H}_s)_{s\geq0} = (\tilde{\sigma}_s, \tilde{\tau}_{\tilde{\sigma}_s}, \tilde{S}_{\tilde{\sigma}_s})_{s\geq0} = (\sigma_{t+s} - \sigma_t, \theta_{t+s} - \theta_t, H_{t+s} - H_t)_{s\geq0}
\]
is independent of $\mathcal{F}_{\sigma_t}$ and has the same law as $(\sigma_s, \theta_s, H_s)_{s\geq0}$.
Finally, we see that for any $s,t \geq0$ and any $\alpha,\beta,\gamma \geq0$,
\begin{align*}
 \mathbb{E}\bigg[\e^{-\alpha \sigma_{t+s} - \beta \theta_{t+s} - \gamma H_{t+s}}&\bm{1}_{\{t+s < L_{\infty}^R\}}\bigg] \\
  & = \mathbb{E}\bigg[\e^{-\alpha \sigma_{t} - \beta \theta_{t} - \gamma H_{t}}\bm{1}_{\{t < L_{\infty}^R\}}\mathbb{E}\bigg[\e^{-\alpha \tilde{\sigma}_s - \beta \tilde{\theta}_{s} - \gamma \tilde{H}_{s}}\bm{1}_{\{s < \tilde{L}_{\infty}\}}|\mathcal{F}_{\sigma_t}\bigg]\bigg] \\
  & = \mathbb{E}\bigg[\e^{-\alpha \sigma_{t} - \beta \theta_{t} - \gamma H_{t}}\bm{1}_{\{t < L_{\infty}^R\}}\bigg]\mathbb{E}\bigg[\e^{-\alpha \sigma_{s} - \beta \theta_{s} - \gamma H_{s}}\bm{1}_{\{s < L_{\infty}^R\}}\bigg].
\end{align*}

\noindent
We classicaly deduce from this, and the right-continuity of $(\sigma_t, \theta_t, H_t)_{s\geq0}$, that for any $t\geq0$ and any $\alpha, \beta, \gamma \geq0$, we have
\[
 \mathbb{E}\left[\e^{-\alpha \sigma_{t} - \beta \theta_{t} - \gamma H_{t}}\bm{1}_{\{t < L_{\infty}^R\}}\right] = \e^{-\kappa(\alpha,\beta,\gamma)t},
\]

\noindent
where
\[
 \kappa(\alpha,\beta,\gamma) = -\log \mathbb{E}\left[\e^{-\alpha \sigma_{1} - \beta \theta_{1} - \gamma H_{1}}\bm{1}_{\{1 < L_{\infty}^R\}}\right] \geq0.
\]

\noindent
We also see that for any $t\geq0$, $\mathbb{P}(L_{\infty}^R  > t) = \e^{-\kappa(0,0,0)t}$ so that $L_{\infty}^R$ is exponentially distributed with parameter $\kappa(0,0,0) \geq0$ (if $\kappa(0,0,0) = 0$, then $L_{\infty}^R = \infty$ a.s.). It is a well-known fact that $0$ is recurrent for the reflected process if and only if $L_{\infty}^R = \infty$ a.s., which completes the proof.
\end{proof}

Using the convention $\e^{-\infty} = 0$, we have for any $t\geq0$ and any $\alpha, \beta, \gamma \geq 0$ such that $\alpha + \beta + \gamma > 0$,
\[
 \mathbb{E}\left[\e^{-\alpha \sigma_{t} - \beta \theta_{t} - \gamma H_{t}}\right] = \e^{-\kappa(\alpha,\beta,\gamma)t},
\]


From now on, we will assume for simplicity that $0$ is recurrent for the reflected process but the proof carries through if it is not the case, with minor adaptations. Recall now that we have defined $G_t = \sup\{s<t, Z_s = S_s\}$ the last return to $0$ before $t$ of $(R_t)_{t\geq0}$.
Let us state the following lemma, of which we omit the proof since it is no different from the proof of \cite[Ch.~VI Lem.~6]{b}.
\begin{lemma}\label{indep_wiener}
 Let $e = e(q)$ be an exponential random variable of parameter $q$, independent of $(\tau_t,Z_t)_{t\geq 0}$.
 \begin{enumerate}[label=(\roman*)]
  \item If $0$ is irregular for the reflected process $(R_t)_{t\geq0}$, then the processes $(\tau_t, Z_t)_{t\in[0,G_e]}$ and $(\tau_{G_e +t} - \tau_{G_e}, Z_{t + G_e} - Z_{G_e})_{t\in[0,e-G_e)}$ are independent.
  
  \item If $0$ is regular for the reflected process $(R_t)_{t\geq0}$, then the processes $(\tau_t, Z_t)_{t\in[0,G_e)}$ and $(\tau_{G_e +t} - \tau_{G_e-}, Z_{t + G_e} - Z_{G_e-})_{t\in[0,e-G_e)}$ are independent. 
 \end{enumerate}
\end{lemma}

Recalling Remark~\ref{rem_dual}, we also introduce the same objects for the dual process $(\hat{Z}_t)_{t\geq0} = (-Z_t)_{t\geq0}$. If we set $(\hat S_t)_{t\geq 0} = (\sup_{[0,t]} \hat Z_t)_{t\geq 0}$, then the dual reflected process $(\hat{R}_t)_{t \geq 0}= (\hat{S}_t-\hat{Z}_t)_{t\geq 0}$ also posesses a local time at $0$ denoted by $(Lt^{\hat{R}})_{t\geq 0}$, with inverse $(\hat{\sigma}_t)_{t\geq0}$.
Lemma~\ref{lem:trivariate} also holds and so the process $(\hat{\sigma}_t, \hat{\theta}_t, \hat{H}_t)_{t\geq0} = (\hat{\sigma}_t, \tau_{\hat{\sigma}_t}, \hat{S}_{\hat{\sigma}_t})_{t\geq0}$ is a trivariate subordinator with Laplace exponent that we denote by $\hat{\kappa}$.

\subsection{Laplace transform of $(G, \tau_G , S)$}

An crucial step in the proof of Theorem~\ref{WH} is the following result, where we recall that $\kappa$ is the Laplace exponent of $(\sigma_t, \theta_t, H_t)_{t\geq0}$, defined in~\eqref{def:kappa}.

\begin{proposition}\label{lemWH_app}
Let $e = e(q)$ be an exponential random variable of parameter $q$, independent of $(\tau_t,Z_t)_{t\geq 0}$.

\noindent
(A) If $0$ is irregular for the reflected process $(R_t)_{t\geq0}$, then
\[
  \mathbb{E}\left[\e^{-\alpha G_{e} - \beta \tau_{G_{e }} - \gamma S_{e}}\right] = \frac{\kappa(q,0,0)}{\kappa(\alpha+q, \beta, \gamma)}.
\]
\noindent
(B) If $0$ is regular for the reflected process, then the same result holds with $\tau_{G_e}$ replaced by $\tau_{G_e-}$.
\end{proposition}

We split the proof into two parts: we first treat the case where $0$ is irregular for the reflected process and then we treat the case where $0$ is regular (which is more involved and needs some preliminary estimates).

\begin{remark}
\label{rem:regular}
We mention the following classification, see for instance Bertoin \cite{MR1465812}: if $(Z_t)_{t\geq0}$ has infinite variations, then $0$ is regular for $(R_t)_{t\geq0}$. If it has finite variation, then denote by $\dr$ its drift coefficient. If $\dr >0$ then $0$ is regular for $(R_t)_{t\geq0}$ whereas if $\dr <0$, it is irregular. In the remaining case $\dr = 0$, let us set $Z_t^+ = \sum_{s\leq t}\Delta Z_s \bm{1}_{\{\Delta Z_s > 0\}}$ and $Z_t^- = -\sum_{s\leq t}\Delta Z_s \bm{1}_{\{\Delta Z_s < 0\}}$ so that $Z_t = Z_t^+ - Z_t^-$. Then $0$ is irregular for $(R_t)_{t\geq0}$ if and only if $\lim_{t\to0}Z_t^+ / Z_t^- = 0$ a.s.
\end{remark}

\paragraph*{Case (A): assume $0$ is irregular for the reflected process}

\begin{proof}[Proof of Proposition \ref{lemWH_app} in case (A)]
We assume here that $0$ is irregular for the reflected process. Therefore, the zero set of $(R_t)_{t\geq0}$ is discrete (without accumulation points) and we can define for any $n\geq0$, $T_{n+1} = \inf\{t > T_n, \: R_t = 0\}$ with $T_0 = 0$. Then the sequence $(T_n)_{n\geq0}$ is an increasing random walk. Moreover, for any $n\geq0$ and any $t\in[T_n, T_{n+1})$, $G_t = T_n$ and $S_t = S_{T_n}$.

\smallskip
Note that the the ladder time process $(\sigma_t)_{t\geq0}$ is a compound Poisson process and its Lévy measure is proportional to the law of $T_1$, where $T_1 = \inf\{t > 0, \: R_t = 0\}$. This forces the trivariate subordinator $(\sigma_t, \theta_t, H_t)_{t\geq0}$ to be a trivariate compound Poisson process with Lévy measure proportional to the law of $(T_1, \tau_{T_1}, S_{T_1})$. For any non-negative $\alpha, \beta,\gamma$, we have
\begin{align*}
 \mathbb{E}\left[\e^{-\alpha G_{e} - \beta \tau_{G_{e }} - \gamma S_{e}}\right] & = \mathbb{E}\bigg[q\int_0^{\infty}\e^{-qt -\alpha G_{t} - \beta \tau_{G_{t}} - \gamma S_{t}}\dr t\bigg] \\
 & = \sum_{n\geq0}\mathbb{E}\bigg[\e^{-(\alpha + q)T_n - \beta\tau_{T_n} - \gamma S_{T_n}}\int_{0}^{T_{n+1}-T_n}q\e^{-qt}\dr t\bigg].
\end{align*}

\noindent
By the strong Markov property, $(\tau_{t+T_n} - \tau_{T_n}, R_{t+T_n})_{t\in[0, T_{n+1} - T_n]}$ is independent of $\mathcal{F}_{T_n}$ and is equal in law to $(\tau_t, R_t)_{t\in[0,T_1]}$. Therefore, we get
\[
 \mathbb{E}\left[\e^{-\alpha G_{e} - \beta \tau_{G_{e }} - \gamma S_{e}}\right] = \mathbb{E}\left[1 - \e^{-q T_1}\right]\sum_{n\geq0}\mathbb{E}\left[\e^{-(\alpha + q)T_n - \beta\tau_{T_n} - \gamma S_{T_n}}\right].
\]

\noindent
By the strong Markov property, the sequence $(T_n, \tau_{T_n}, S_{T_n})_{n\geq0}$ is a random walk and thus, we get
\[
 \mathbb{E}\left[\e^{-\alpha G_{e} - \beta \tau_{G_{e}} - \gamma S_{e}}\right] = \frac{\mathbb{E}\left[1 - \e^{-q T_1}\right]}{\mathbb{E}\left[1 - \e^{-(\alpha + q)T_1 - \beta\tau_{T_1} - \gamma S_{T_1}}\right]} = \frac{\kappa(q,0,0)}{\kappa(\alpha + q, \beta, \gamma)}, 
\]
where we recall that $\kappa$ is the Laplace exponent of $(\sigma_t, \theta_t, H_t)_{t\geq0}$.
\end{proof}

\paragraph*{Case (B): assume $0$ is regular for the reflected process.}

Before proving Proposition \ref{lemWH_app} in the regular case, we first recall and show a few resuts when~$0$ is regular.

\begin{lemma}\label{lemma_continuity}
If $0$ is regular for the reflected process $(R_t)_{t\geq 0}$, then a.s.\ we have $S_e = S_{G_e-} = Z_{G_e-}$. Moreover, if $(Z_t)_{t\geq0}$ is \emph{not} a compound Poisson process and $0$ is also regular for the dual reflected process $(\hat R_t)_{t\geq 0}$, we have $Z_{G_e-} = Z_{G_e}$ and $\tau_{G_e-} = \tau_{G_e}$.
\end{lemma}

\begin{proof}
 It is proved properly in \cite[Ch.~VI Thm.~5-(i)]{b} that if $0$ is regular, then $S_e = S_{G_e-} = Z_{G_e-}$, so we refer the reader to it. It is also proved there that $(Z_t)_{t\geq0}$ cannot make a positive jump at time $G_e$, i.e. that $Z_{G_e-} \geq Z_{G_e}$.
 
\smallskip
We now turn to the second part of the lemma: we assume now that $(Z_t)_{t\geq0}$ is not a compound Poisson process. Let us define the process $(\tilde{Z}_t)_{t\geq0} = (Z_{(e-t)-} - Z_e)_{t\in[0,e]}$, which, by duality, is equal in law to $(-Z_t)_{t\in[0,e]}$.
 Since $(Z_t)_{t\geq0}$ is not a compound Poisson process, \cite[Proposition 4 Chapter VI]{b} implies that the supremum is reached at a unique time. This ensures that, in the obvious notations, we have $\tilde{G}_e = e - G_e$ and $\tilde{S}_e = S_e - Z_e$. Now if $0$ is regular for the dual reflected process, we have $\tilde{Z}_{\tilde{G}_e-} \geq \tilde{Z}_{\tilde{G}_e}$, i.e. $Z_{G_e} \geq Z_{G_e-}$ which shows that $Z_{G_e-} = Z_{G_e}$.
 
Finally, we show that if $G_e$ is a point of continuity of $(Z_t)_{t\geq0}$, then it is also a point of continuity of $(\tau_t)_{t\geq0}$. We denote by $\Pi$ the Lévy measure of $(\tau_t, Z_t)_{t\geq0}$, and $\mathrm{m} \geq0$ the drift coefficient of $(\tau_t)_{t\geq0}$. Now remark that the Lévy measure $\Pi_\tau$ of $(\tau_t)_{t\geq0}$ can be decomposed into two parts:
\[
 \Pi_\tau(\dr r) = \int_{z\in\mathbb{R}}\Pi(\dr r, \dr z) = \int_{z\in\mathbb{R}^*}\Pi(\dr r, \dr z) + \int_{z=0}\Pi(\dr r, \dr z)\,.
\]
In other words, we can decompose the subordinator $(\tau_t)_{t\geq0}$ as
\[
 \tau_t = \mathrm{m} t + \sum_{s < t}\Delta\tau_s \bm{1}_{\{|\Delta Z_s| > 0\}} + \sum_{s < t}\Delta\tau_s \bm{1}_{\{|\Delta Z_s| = 0\}} = \mathrm{m} t + \tau_t^* + \tau_t^0,
\]
which is then a sum of a drift and two independent indenpendent subordinator (since they do not jump at the same time).
We see first that if $t$ is a jumping time of $(\tau_s^*)_{s\geq0}$, then it is also a jumping time of $(Z_s)_{s\geq0}$ so that any point of continuity of $(Z_s)_{s\geq0}$ is a point of continuity of $(\tau_s^*)_{s\geq0}$. Next, we see that $(\tau_t^0)_{t\geq0}$ and $(Z_t)_{t\geq0}$ are independent since they do not jump at the same time and since $G_e$ is a functional of $(Z_t)_{t\geq0}$, it is also independent from $(\tau_t^0)_{t\geq0}$ and therefore it can not be a point of discontinuity of $(\tau_t^0)_{t\geq0}$. This completes the proof.
\end{proof}

When $0$ is regular for $(R_t)_{t\geq0}$, the ladder time process $(\sigma_t)_{t\geq0}$ is a strictly increasing subordinator.
For the local time $(L_t^R)_{t\geq0}$ of the reflected process at the level $0$, we have that there exists some $\mathrm m_R \geq0$ such that a.s., for any $t\geq0$, 
\[
 \mathrm m_R L_t^{R} = \int_0^t\bm{1}_{\{R_s = 0\}} \dr s \quad \text{and} \quad \sigma_t = \mathrm m_R t + \sum_{s\leq t}\Delta\sigma_s.
\]

\noindent
Then, it is well-known that the excursion process $(e_t^R)_{t\geq0}$ defined by 
\begin{equation*}
 e_t^R = \begin{cases}
         (R_{\sigma_{t-} + s})_{s\in[0,\Delta\sigma_t]}&\text{if } \Delta\sigma_t >0,  \\
         \Upsilon &\text{otherwise,}  
        \end{cases}
\end{equation*}
is a Poisson point process valued in the excursion space $\mathcal{E}_0$ (with the notation introduced in Section~\ref{sec:exc_theory}). We will denote by $\nn_R$ its characteristic measure. Finally, since $(\sigma_t)_{t\geq0}$ is strictly increasing if $0$ is regular, we have almost surely for every $t\geq0$, $\theta_{t-} = \tau_{(\sigma_{t-})-}$ and $H_{t-} = S_{(\sigma_{t-})-}$.

\begin{proof}[Proof of Proposition \ref{lemWH_app} in case (B)]
The idea of the proof is similar to that in case (A), although it is a little bit more tedious.
Using that $S_e = S_{G_e -}$, and summing over all excursions of $(R_t)_{t\geq0}$, indexed by their left end-point $g$ and their right end-point $d$ (using the same decomposition and similar notation as in the proof of Proposition~\ref{indep}), we have for any non-negative $\alpha, \beta, \gamma$
\begin{align*}
 & \mathbb{E}\Big[\e^{-\alpha G_e - \beta\tau_{G_e -} - \gamma S_e}\Big]  =  \mathbb{E}\left[\int_0^{\infty}q\e^{-qt - \alpha G_t - \beta \tau_{G_t-} - \gamma S_{G_t-}}\dr t\right] \\
  & \qquad =  \mathbb{E}\bigg[\int_0^{\infty}q\e^{-(\alpha +q)t - \beta \tau_{t-} - \gamma S_{t-}}\bm{1}_{\{R_t = 0\}}\dr t\bigg] + \mathbb{E}\bigg[\sum_{g \in \mathcal G_R}\e^{-(\alpha + q)g - \beta\tau_{g-} -\gamma S_{g-}}\left(1-\e^{-q(d-g)}\right)\bigg].
\end{align*}
Since for every $t\geq0$, almost surely $t$ is not a discontinuity of $(\tau_t)_{t\geq0}$ and $(S_t)_{t\geq0}$, the first term on the right-hand-side is equal to
\begin{align*}
 \mathbb{E}\left[\int_0^{\infty}q\e^{-(\alpha +q)t - \beta \tau_{t-} - \gamma S_{t-}}\bm{1}_{\{R_t = 0\}}\dr t\right] & =  q \mathrm m_R\mathbb{E}\left[\int_0^{\infty}\e^{-(\alpha +q)t - \beta \tau_{t} - \gamma S_{t}}\dr L_t^R\right] \\ 
 & =  q\mathrm m_R \int_0^{\infty} \mathbb{E}\left[\e^{-(\alpha + q)\sigma_t - \beta\theta_t - \gamma H_t}\right] \dr t 
  = \frac{q\mathrm m_R}{\kappa(\alpha + q, \beta, \gamma)} \,,
\end{align*}
where we have also used that $(\sigma_t)_{t\geq 0}$ is the right-continuous inverse of $(L_t^R)_{t\geq 0}$ for the second identity.
For the second term, we use the Master formula for Poisson point processes which tells us that
\[
\mathbb{E}\bigg[\sum_{g\in \mathcal G_R}\e^{-(\alpha + q)g - \beta\tau_{g-} -\gamma S_{g-}}\big(1-\e^{-q(d-g)}\big)\bigg] =  \nn_R(1 - \e^{-q\ell})\mathbb{E}\left[\int_0^{\infty}\e^{-(\alpha + q)\sigma_{t-} - \beta \tau_{\sigma_{t-}-} - \gamma S_{\sigma_{t-}-}}\dr t\right]\,.
\]
Then, using that almost surely for every $t\geq0$, $\theta_{t-} = \tau_{(\sigma_{t-})-}$ and $H_{t-} = S_{(\sigma_{t-)}-}$, and that for every $t\geq0$, almost surely, $t$ is not a discontinuity of $(\sigma_t, \theta_t, H_t)_{t\geq0}$, this is equal to
\[
 \nn_R(1 - \e^{-q\ell})\int_0^{\infty}\mathbb{E}\left[\e^{-(\alpha + q)\sigma_t - \beta\theta_t - \gamma H_t}\right] \dr t 
   = \frac{\nn_R(1-\e^{-q\ell})}{\kappa(\alpha + q, \beta, \gamma)}.
\]
Since $\kappa(q,0,0)$ is the Laplace exponent of $(\sigma_t)_{t\geq0}$, it follows from the exponential formula for Poisson point processes that $\kappa(q,0,0) = q\mathrm m_R + \nn_R(1-\e^{-q\ell})$,  which gives the desired result.
\end{proof}

\subsection{Proof of Theorem~\ref{WH} and Proposition~\ref{wiener_prop_corps}}

We are finally ready to give the proof of Theorem~\ref{WH}.
Afterwards, we deduce Proposition~\ref{wiener_prop_corps} from it.

\begin{proof}[Proof of Theorem \ref{WH}]
The independence of the two triplets follows from Lemma \ref{indep_wiener} and that $S_e = Z_{G_e-}$ in the regular case and $S_e = Z_{G_e}$ in the irregular case (recall Lemma~\ref{lemma_continuity}).
We now turn to proving items \textit{(ii)}-\textit{(iii)}: we first deal with the case where $(Z_t)_{t\geq 0}$ is not a compound Poisson process, and then we treat the case of a compound Poisson process by approximation.
The main idea is to prove that the law of $(G_e, \tau_{G_e}, S_e)$ and $(e - G_e, \tau_e - \tau_{G_e}, S_e - Z_e)$ (resp.\ of $(G_e, \tau_{G_e-}, S_e)$ and $(e - G_e, \tau_e - \tau_{G_e-}, S_e - Z_e)$ in the irregular case) are infinitely divisible; we then deduce their Lévy measure.
 
\medskip\noindent
\textit{Step 1:} We show that the law of $(G_e, \tau_{G_e}, S_e)$, resp.\ of $(G_e, \tau_{G_e-}, S_e)$, is infinitely divisible if $0$ is irregular, resp.\ regular, for the reflected process $(R_t)_{t\geq 0}$.
We only prove it in the irregular case as the proofs are identical.

For any positive $q,\alpha,\beta,\gamma$, the Lévy-Khintchine formula tells us that
\[
 \kappa(\alpha + q, \beta, \gamma) - \kappa(q,0,0) = \int_{[0,\infty)^3}(1 - \e^{-\alpha t - \beta r - \gamma x})\e^{-qt}\Pi(\dr t, \dr r, \dr x),
\]
where $\Pi$ is the Lévy measure of $(\sigma_t, \theta_t, H_t)_{t\geq0}$.
By Proposition~\ref{lemWH_app}, we have
\[
 \mathbb{E}\left[\e^{-\alpha G_e - \beta \tau_{G_e} - \gamma S_e}\right] = \frac{\kappa(q,0,0)}{\kappa(q,0,0) + (\kappa(\alpha + q, \beta, \gamma) - \kappa(q,0,0))},
\]
and we therefore see that the law of $(G_e, \tau_{G_e}, S_e)$ is equal to the law of a pure jump Lévy process with Lévy measure $\e^{-qt}\Pi(\dr t, \dr r, \dr x)$, evaluated at an independent exponential time of parameter $\kappa(q,0,0)$. Therefore, it is infinitely divisible and it has no drift part and no Brownian part, so it is characterized by its Lévy measure.

\medskip\noindent
\textit{Step 2:} We assume that $(Z_t)_{t\geq0}$ is not a compound Poisson process, and we show that the law of $(e - G_e, \tau_e - \tau_{G_e}, S_e - Z_e)$,
resp.\ of $(e - G_e, \tau_e - \tau_{G_e-}, S_e - Z_e)$, is infinitely divisible if $0$ is irregular, resp.\ regular for the reflected process $(R_t)_{t\geq 0}$.
Note that the corresponding distributions are again characterized by their Lévy measure.

Let us define the process $(\tilde{\tau}_t, \tilde{Z}_t)_{t\geq0} = (\tau_e - \tau_{(e-t)-}, Z_{(e-t)-} - Z_e)_{t\in[0,e]}$, which, by duality, is equal in law to $(\tau_t, -Z_t)_{t\in[0,e]}$.
Since $(Z_t)_{t\geq0}$ is not a compound Poisson process its supremum is reached at a unique time so we have $\tilde{G}_e = e - G_e$ and $\tilde{S}_e = S_e - Z_e$, with the obvious notation. Moreover, we have $\tilde{\tau}_{\tilde{G}_e} = \tau_e - \tau_{G_e-}$ and $\tilde{\tau}_{\tilde{G}_e -}=  \tau_e - \tau_{G_e}$.

If $0$ is irregular for $(R_t)_{t\geq0}$, then it is necessarily regular for $(\hat{R}_t)_{t\geq0}$, see Remark~\ref{rem:regular} and we can apply Step~1 with $(\tilde{G}_e, \tilde{\tau}_{\tilde{G}_e-}, \tilde{S}_e)$. If $0$ is regular for $(R_t)_{t\geq0}$ and irregular for $(\hat{R}_t)_{t\geq0}$, then we can apply Step 1 with $(\tilde{G}_e, \tilde{\tau}_{\tilde{G}_e}, \tilde{S}_e)$. Finally, if $0$ is regular for both sides, then $\tau_{G_e} = \tau_{G_e -}$ by Lemma \ref{lemma_continuity} and we can again apply the first step.

\medskip\noindent
\textit{Step 3:} We are now able to conclude the proof of items \textit{(ii)} and \textit{(iii)} when $(Z_t)_{t\geq0}$ is not a compound Poisson process. 
Again, we only show it in the irregular case.
It is well-known that $(e, \tau_e, Z_e)$ is infinitely divisible, see \cite[Ch.~VI Lem.~7]{b}, and that its Lévy measure is given by
\[
\mu(\dr t , \dr r, \dr x) = t^{-1} \e^{-q t} \mathbb{P}(\tau_t \in \dr r, Z_{t} \in \dr x)\dr t.
\]
Let us denote the Lévy measures of $(G_e, \tau_{G_e}, S_e)$ and $(e-G_e, \tau_e - \tau_{G_e}, Z_e - S_e)$ by $\mu_+$ and $\mu_-$ and recall that by Steps~1 and 2 that they characterize their distributions. Now, observe that
\[
(e, \tau_{e}, Z_{e}) =(G_{e}, \tau_{G_{e}}, S_{e}) +(e-G_{e}, \tau_{e} - \tau_{G_{e}}, Z_{e} - S_{e}) \,,
\]
with the triplets being independent. It follows that $\mu = \mu_+ + \mu_-$. Since $(Z_t)_{t\geq0}$ is not a coumpound Poisson process, we have $\mathbb{P}(Z_t = 0) = 0$ for every $t\geq0$, and since $\mu_+$ is supported on $[0,+\infty)^{3}$ and $\mu_-$ on $[0,+\infty)^2\times (-\infty, 0]$, we get
\begin{align*}
\mu_{+}(  \dr t , \dr r, \dr x) & = t^{-1} \e^{-q t} \mathbb{P}(\tau_t \in \dr r, Z_{t} \in \dr x) \dr t, \qquad  (t>0, r>0, x>0) \\
\mu_{-}(  \dr t , \dr r, \dr x)  &= t^{-1} \e^{-q t} \mathbb{P}(\tau_t \in \dr r, Z_{t} \in \dr x) \dr t, \qquad  (t>0, r>0, x<0) \,,
\end{align*}
which finishes this step.

\medskip\noindent
\textit{Step 4:} We now extend the proof of items \textit{(ii)} and \textit{(iii)} when $(Z_t)_{t\geq0}$ is a compound Poisson process. Note that in this case, $0$ is regular for both sides. We proceed by approximation and consider $(Z_t^\epsilon)_{t\geq0} = (Z_t + \epsilon t)_{t\geq0}$ which is a Lévy process but not a compound Poisson process. It is easy to see that $0$ is regular for $(R_t^\epsilon)_{t\geq0}$ but is irregular for $(\hat{R}_t^\epsilon)_{t\geq0}$, with the obvious notation.

\smallskip
Let us now show that a.s. $G_e^\epsilon \to G_e$, $S_e^\epsilon \to S_e$ and $\tau_{G_e^\epsilon -} \to \tau_{G_e-}$ as $\epsilon\to0$.
This is obvious on the event $\{G_e = e\}$, since we have $G_e^{\epsilon} = G_e$ and $S_e^\epsilon = S_e + \epsilon e$.
Now, on the event $\{G_e < e\}$, since $(Z_t)_{t\geq 0}$ is a compound Poisson process, we have $Z_t < Z_{G_e-} = S_e$ for any $t\in[G_e, e]$.
Let $D_e = \sup_{t\in[G_e, e]} Z_t < S_e$ and set $\epsilon_0 = (S_e - D_e) / e$: we see that for any $t\in[G_e, e]$ and any $\epsilon\in(0, \epsilon_0)$,
\[
 Z_t^{\epsilon} = Z_t + \epsilon t < S_e < Z_{G_e-}^\epsilon.
\]
Therefore, for any $\epsilon\in(0, \epsilon_0)$, we have $G_e^\epsilon = G_e$, which implies again that a.s. $G_e^\epsilon \to G_e$, $S_e^\epsilon \to S_e$ and $\tau_{G_e^\epsilon -} \to \tau_{G_e-}$ as $\epsilon\to0$.

Finally, as $\gep \to 0$, the measures
\[
 \bm{1}_{\{x > 0\}}\mathbb{P}(\tau_t \in \dr r, Z_t^\epsilon \in \dr x) \quad \text{and} \quad \bm{1}_{\{x < 0\}}\mathbb{P}(\tau_t \in \dr r, Z_t^\epsilon \in \dr x)
\]
respectively converge to
\[
 \bm{1}_{\{x \geq 0\}}\mathbb{P}(\tau_t \in \dr r, Z_t \in \dr x) \quad \text{and} \quad \bm{1}_{\{x < 0\}}\mathbb{P}(\tau_t \in \dr r, Z_t \in \dr x) \,.
\]
Then the result follows by approximation.
\end{proof}

\begin{proof}[Proof of Proposition~\ref{wiener_prop_corps}]
Paraphrasing Corollary VI.10 p 165 in \cite{b}:
for $\alpha'>1$, the formula 
\[
\kappa(\alpha', \beta, \gamma) = c \exp\bigg(\int_0^{\infty}\int_{[0,\infty)\times\mathbb{R}}\frac{\e^{-t} - \e^{-\alpha' t- \beta r - \gamma x}}{t}\bm{1}_{\{x\geq0\}}\mathbb{P}(\tau_t\in \dr r, Z_t\in \dr x)\dr t\bigg)
\]
follows from Theorem \ref{WH} and Proposition \ref{lemWH_app} applied with $q =1$ (and $\alpha'=\alpha+q$). The identity is extended by analyticity. 

Recall also the definition
 \[
  \bar{\kappa}(\alpha, \beta, \gamma) = \exp\bigg(\int_0^{\infty}\int_{[0,\infty)\times\mathbb{R}}\frac{\e^{-t} - \e^{-\alpha t- \beta r + \gamma x}}{t}\bm{1}_{\{x<0\}}\mathbb{P}(\tau_t\in \dr r, Z_t\in \dr x)\dr t\bigg).
 \]
Then, if $0$ is irregular for $(R_t)_{t\geq0}$, we have for any positive $\alpha, \beta, \gamma, \hat \alpha, \hat \beta, \hat \gamma$, we have
\begin{equation}
\label{eq2}
\mathbb{E} \left[ \e^{-\alpha G_{e} - \beta \tau_{G_{e}} - \gamma S_{e} -\hat{\alpha} (e -G_{e} ) -\hat{\beta} ( \tau_{e} -\tau_{G_{e}}) -\hat{\gamma}(Z_{e} -S_{e})} \right] = \frac{\kappa(q,0,0)}{\kappa(\alpha+q, \beta, \gamma)}   \frac{\bar{\kappa}(q,0,0)}{\bar{\kappa}(\hat{\alpha}+q, \hat{\beta} \hat{\gamma})}
\,.
\end{equation}
If $0$ is regular for $(R_t)_{t\geq0}$, then the same identity holds with $\tau_{G_e}$ replaced by $\tau_{G_e-}$ and $\tau_e - \tau_{G_e}$ replaced by $\tau_e - \tau_{G_e -}$.
This formula yields~\eqref{eq1}.

Moreover, the duality entails that, if $(Z_t)_{t\geq0}$ is not a compound Poisson process, there exists a constant $\hat{c}>0$ such that $\hat{\kappa}(\alpha, \beta, \gamma) = \hat{c}\bar{\kappa}(\alpha, \beta, \gamma)$.
\end{proof}


\section{Convergence of L\'evy processes}\label{appendix_levy}

In this section, we give some results on converging sequences of Lévy processes.

\subsection{Convergence of processes and convergence of Laplace exponents}

We consider a family of Lévy processes $\{ (\tau_t^h, Z_t^h)_{t\geq0}, h>0\}$ with $(\tau_t^h)_{t\geq0}$ a subordinator (for every $h>0$), which converges in law for the Skorokhod topology as $h\to0$, to some Lévy process $(\tau_t^0, Z_t^0)_{t\geq0}$, with $(\tau_t^0)_{t\geq0}$ also a subordinator and $(Z_t^0)_{t\geq0}$ not a compound Poisson process.
Note that, according to Jacod-Shiryaev \cite[Thm.~2.9 p 396]{js}, $(\tau_t^h, Z_t^h)_{t\geq0}$ converges in law for the Skorokhod topology if and only if $(\tau_1^h, Z_1^h)$ converges in law to $(\tau_1^0, Z_1^0)$ as $h\to0$.
For non-negative $\alpha, \beta, \gamma$ such that $\alpha + \beta + \gamma > 0$, we introduce the quantities
\[
 \phi_h(\alpha,\beta, \gamma) = \int_0^{\infty}\int_{[0,\infty)\times\mathbb{R}}\frac{\e^{-t} - \e^{-\alpha t- \beta r - \gamma x}}{t}\bm{1}_{\{x\geq0\}}\mathbb{P}(\tau_t^h\in \dr r, Z_t^h\in \dr x)\dr t
\]
and 
\[
 \bar{\phi}_h(\alpha,\beta, \gamma) = \int_0^{\infty}\int_{[0,\infty)\times\mathbb{R}}\frac{\e^{-t} - \e^{-\alpha t- \beta r + \gamma x}}{t}\bm{1}_{\{x<0\}}\mathbb{P}(\tau_t^h\in \dr r, Z_t^h\in \dr x)\dr t.
\]

\noindent
We will denote by $\phi_0(\alpha,\beta, \gamma)$ and $\bar{\phi}_0(\alpha,\beta, \gamma)$ the corresponding quantities for $(\tau_t^0, Z_t^0)_{t\geq0}$. As we should expect, we have the following result.
\begin{proposition}\label{conv_laplace_exponent}
 Let $(\tau_t^h, Z_t^h)_{t\geq0}$ and $(\tau_t^0, Z_t^0)_{t\geq0}$ be as above. Then we have for any non-negative $\alpha, \beta, \gamma$ such that $\alpha + \beta > 0$,
 \[
\lim_{h\to 0} \phi_h(\alpha,\beta, \gamma) =\phi_0(\alpha,\beta, \gamma) \quad \text{and} \quad \lim_{h\to 0}\bar{\phi}_h(\alpha,\beta, \gamma) = \bar{\phi}_0(\alpha,\beta, \gamma) \,.
 \]
\end{proposition}

\begin{proof}
 We will only show the convergence of $\phi_h$ as the proof for $\bar{\phi}_h$ is similar. We only stress that since $\mathbb{P}(Z_t^0 = 0) = 0$, by dominated convergence we have
\[
\lim_{h\to0} \mathbb{E}\left[(\e^{-t} - \e^{-\alpha t - \beta \tau_t^h - \gamma Z_t^h})\bm{1}_{\{Z_t^h \geq0\}}\right] = \mathbb{E}\left[(\e^{-t} - \e^{-\alpha t - \beta \tau_t^0 - \gamma Z_t^0})\bm{1}_{\{Z_t \geq0\}}\right] \,.
\]
Now since
\[
 \phi_h(\alpha,\beta, \gamma) = \int_0^{\infty} \frac1t \mathbb{E}\left[(\e^{-t} - \e^{-\alpha t - \beta \tau_t^h - \gamma Z_t^h})\bm{1}_{\{Z_t^h \geq0\}}\right] \dr t,
\]
it only remains to dominate the integrand. We first dominate for $t\in(0,1)$. We have
\begin{align*}
 \left|\mathbb{E}\left[(\e^{-t} - \e^{-\alpha t - \beta \tau_t^h - \gamma Z_t^h})\bm{1}_{\{Z_t^h \geq0\}}\right]\right| \leq & 1-\e^{-t} + \mathbb{E}\left[(1 - \e^{-\alpha t - \beta \tau_t^h - \gamma Z_t^h})\bm{1}_{\{Z_t^h \geq0\}}\right] \\
 \leq &  1-\e^{-t} + \mathbb{E}\left[(\beta \tau_t^h + \gamma |Z_t^h|)\wedge 1\right].
\end{align*}

\noindent
Then by Lemma \ref{tempspetit} below (applied to the sequence of bivariate Lévy process $(\tau_t^h,Z_t^h)$ valued in $\mathbb R^2$), there exists a constant $C_{\beta, \gamma} >0$ such that for any $h\in(0,1)$ and for any $t\in(0,1)$, $\mathbb{E}[(\beta \tau_t^h + \gamma |Z_t^h|)\wedge 1] \leq C_{\beta, \gamma} \sqrt{t}$. Thus we get for any $h\in(0,1)$ and any $t\in(0,1)$,
\[
 t^{-1}\left|\mathbb{E}\left[(\e^{-t} - \e^{-\alpha t - \beta \tau_t^h - \gamma Z_t^h})\bm{1}_{\{Z_t^h \geq0\}}\right]\right| \leq 1 + \frac{C_{\beta, \gamma}}{\sqrt{t}},
\]
which is integrable on $(0,1)$. Next, we dominate on $[1,\infty)$. We have
\[
 \left|\mathbb{E}\left[(\e^{-t} - \e^{-\alpha t - \beta \tau_t^h - \gamma Z_t^h})\bm{1}_{\{Z_t^h \geq0\}}\right]\right| \leq \e^{-t} + \mathbb{E}\left[\e^{-\alpha t - \beta \tau_t^h - \gamma Z_t^h}\bm{1}_{\{Z_t^h \geq0\}}\right].
\]
If $\alpha > 0$, then we can dominate by $\e^{-t} + \e^{-\alpha t}$ which, divided by $t$, is integrable on $[1,\infty)$.
If $\alpha = 0$, then $\beta > 0$ and we have
\[
 \left|\mathbb{E}\left[(\e^{-t} - \e^{-\alpha t - \beta \tau_t^h - \gamma Z_t^h})\bm{1}_{\{Z_t^h \geq0\}}\right]\right| \leq \e^{-t} + \e^{-\Phi_h(\beta)t},
\]
where $\Phi_h$ is the Laplace exponent of $(\tau_t^h)_{t\geq0}$.
Since the latter converges in law to $(\tau_t)_{t\geq0}$, we get that $\lim_{h\to0}\Phi_h(\beta) = \Phi_0(\beta) > 0$ as $h\to0$ where $\Phi_0$ is the Laplace exponent of $(\tau_t^0)_{t\geq0}$. Therefore, there exists a constant $c_{\beta}>0$ such that for any $h\in(0,1)$, $\Phi_h(\beta) \geq c_{\beta}$. Finally, we see that for any $h\in(0,1)$, for any $t\geq1$, we have 
\[
 \left|\mathbb{E}\left[(\e^{-t} - \e^{-\alpha t - \beta \tau_t^h - \gamma Z_t^h})\bm{1}_{\{Z_t^h \geq0\}}\right]\right| \leq \e^{-t} + \e^{-c_{\beta}t},
\]
which, divided by $t$, is integrable on $[1,\infty)$. This completes the proof.
\end{proof}

\begin{proposition}
 Let $(\tau_t^h, Z_t^h)_{t\geq0}$ and $(\tau_t^0, Z_t^0)_{t\geq0}$ be as above and assume that $0$ is recurrent for the reflected limiting process $(R_t^0)_{t\geq0}$ Then we also have for any $\gamma > 0$, $\phi_h(0,0,\gamma) \longrightarrow \phi_0(0,0,\gamma)$ as $h\to0$.
\end{proposition}

Here, we cannot use the dominated convergence theorem as it is not clear how to dominate $\mathbb{E}[\e^{-\gamma Z_t^h} \bm{1}_{\{Z_t^h \geq0\}}]$. We will instead use an argument of tightness combined with the continuity of the Laplace transform.

\begin{proof}
 For every $h > 0$, let $(\sigma_t^h, \theta_t^h, H_t^h)_{t\geq0}$ be the trivariate subordinator associated with $(\tau_t^h, Z_t^h)_{t\geq0}$, as in Appendix \ref{appendix_wiener}. Since the local time of the reflected process $(R_t^h)_{t\geq0}$ is defined up to a constant, we can normalize it so that for any non-negative $\alpha, \beta, \gamma$, we have
 \[
  \kappa_h(\alpha, \beta, \gamma) = \exp(\phi_h(\alpha, \beta, \gamma)),
 \]
where $\kappa_h$ is the Laplace exponent of $(\sigma_t^h, \theta_t^h, H_t^h)_{t\geq0}$. In other words, we can choose the constant $c_h$ in Proposition \ref{wiener_prop_corps} to be equal to $1$. Similarly, we consider $(\sigma_t^0, \theta_t^0, H_t^0)_{t\geq0}$ the trivariate subordinator associated with $(\tau_t^0, Z_t^0)_{t\geq0}$, whose Laplace exponent $\kappa_0$ is such that for any non-negative $\alpha,\beta,\gamma$, we have $\kappa_0(\alpha,\beta,\gamma) = \exp(\phi_0(\alpha, \beta, \gamma))$. Then by Proposition \ref{conv_laplace_exponent}, $\kappa_h(\alpha,\beta,\gamma)$ converges to $\kappa_0(\alpha,\beta,\gamma)$ for any non-negative $\alpha,\beta,\gamma$ such that $\alpha + \beta >0$. In particular, we have for any $\gamma > 0$,
\[
 \lim_{h\to0}\mathbb{E}\left[\e^{-\gamma(\sigma_1^h + H_1^h)}\right] = \lim_{h\to0}\e^{-\kappa_h(\gamma, 0, \gamma)} = \e^{-\kappa_0(\gamma, 0, \gamma)} = \mathbb{E}\left[\e^{-\gamma(\sigma_1^0 + H_1^0)}\right],
\]
which implies that $\sigma_1^h + H_1^h$ converges in law to $\sigma_1^0 + H_1^0$ as $h\to0$. Since $0$ is recurrent for $(R_t^0)_{t\geq0}$, the random variables $\sigma_1^0$ and $H_1^0$ are a.s. finite. This implies that the family of random variables $(\sigma_1^h + H_1^h)_{h\in(0,1)}$ is tight in $\bbR_+$, which in turn implies that the family $((\sigma_1^h, H_1^h))_{h\in(0,1)}$ is tight in $\bbR_+^2$. Let $((\sigma_1^{h_k}, H_1^{h_k}))_{k\in\mathbb{N}}$ be a subsequence which converges in law towards some finite random variables $(\bar{\sigma}, \bar{H})$. Using Proposition \ref{conv_laplace_exponent} again, we see that for any $\alpha >0$ and any $\gamma\geq0$, we have
\[
 \mathbb{E}\left[\e^{-\alpha\bar{\sigma} - \gamma\bar{H}}\right]= \mathbb{E}\left[\e^{-\alpha\sigma_1^0 -\gamma H_1^0}\right].
\]
Since $\bar{H}$ and $H_1^0$ are finite, the above equality also holds for $\alpha = 0$ and $\gamma > 0$ by the monotone convergence theorem, which shows that the law of $(\bar{\sigma}, \bar{H})$ is uniquely determined and therefore $(\sigma_1^h, H_1^h)$ converges in law as $h\to0$ to $(\sigma_1^0, H_1^0)$. Hence $H_1^h$ converges in law to $H_1^0$, which completes the proof, since this implies that $\kappa_h(0,0,\gamma)$ converges to $\kappa_0(0,0,\gamma)$ for any $\gamma > 0$.
\end{proof}

\subsection{Technical lemmas}

\begin{lemma}\label{tempspetit}
Let $(Z^n_t)_{t \geq 0}$, $n\geq 1$ be a sequence of L\'evy processes on $\mathbb{R}^d$ which converges in law to a L\'evy process $(Z_t)_{t\geq0}$. Then there exists a constant $C > 0$ such that for any $n\geq1$, for any $t\geq0$, we have
\[
\mathbb{E}[ \vert Z_t^n \vert \wedge 1 ] \leq C \sqrt{t}.
\]
\end{lemma}

\begin{proof}

We classically decompose the sequence of Levy processes as a sum of a Brownian motion, a compound Poisson process, a pure jump martingale and a drift: $Z_t^n = c_nB_t^n + C_t^n + M_t^n + b_nt$, where $b_n \in \mathbb{R}^d$, $c_n$ is a $d\times d$ matrix. Let us denote by $\nu_n$ the Levy measure of $X^n$, then the Levy measure of $C^n$ is $\bm{1}_{\{|x|\geq1\}}\nu_n(\dr x)$ and $M^n$ has Levy measure $\bm{1}_{\{|x|<1\}}\nu_n(\dr x)$. For any $n\geq1$, for any $t\geq0$, we have
\[
\mathbb{E}[ \vert Z_t^n \vert^2 \wedge 1 ] \leq 4\left(\mathbb{E}[ \vert c_nB_t^n \vert^2 ] + \mathbb{E}[ \vert M_t^n \vert^2 ] + \mathbb{E}[ \vert C_t^n \vert^2 \wedge 1 ] +  t^2\vert b_n \vert^2 \wedge 1\right).
\]
By the maximum inequality for compensated sums, we have
\[
\mathbb{E}\left[ \vert M_t^n \vert^2 \right] \leq \mathbb{E}\Big[ \sup_{[0,t]}\vert M_s^n \vert^2 \Big] \leq t\int_{\{|x|<1\}\setminus\{0\}}|x|^2\nu_n(\dr x).
\]
We introduce the truncation function $\chi : \mathbb{R}^d \to \mathbb{R}^d$ which is a bounded continuous function such that $\chi(x)=x$ if $\vert x \vert \leq 1$. Setting $\tilde{c}_n:= c_n^* c_n + \int \chi (x)\otimes \chi(x)  \nu_n (\dr x)$, which is a $d\times d$ symmetric nonnegative matrix, we get by the above inequalities $\mathbb{E}[ \vert c_nB_t^n \vert^2 ] + \mathbb{E}[ \vert M_t^n \vert^2 ] \leq \mathrm{Tr}(\tilde{c}_n)t$.

\medskip
Let us now introduce the smooth function $\varphi : \mathbb{R}^d \to \mathbb{R}^{+}$ such that $\varphi(x)= \vert x\vert^2$ if $\vert x\vert \leq 1$ and $\varphi(x)=2$ if $\vert x \vert \geq 2$ and $1\leq \varphi(x)\leq 2$ if $\vert x \vert \in [1,2]$. By Itô formula, we have
\[
\mathbb{E}[ \vert C_t^n \vert^2 \wedge 1 ] \leq \mathbb{E}[\varphi(C_t^n)] = \int_0^t \int_{|x|>1}\mathbb{E}[\varphi(C_s^n + x) - \varphi(C_s^n)]\nu_n(\dr x) \leq 2t\int_{|x|>1}\nu_n(\dr x).
\]
\noindent
By Theorem 2.9 p 396 in \cite{js}, since $(Z_t^n)_{t\geq 0}$ converges in law to $(Z_t)_{t\geq 0}$: 
\[
b_n \to b, \ \ \tilde{c}_n \to \tilde{c} \ \ \ \mathrm{and} \ \ \int g(x) \nu_n (\dr x) \to \int g(x) \nu( \dr x), 
\]
for all continuous bounded function $g:\mathbb{R}^d \to \mathbb{R}$ such that g is equal to $0$ in a neighborhood of $0$. Thus, setting $C_1 = \sup_n \mathrm{Tr}(\tilde{c}_n)$, $C_2 = 2\sup_n\int_{|x|>1}\nu_n(\dr x)$, $C_3 = (\sup_n \vert b_n \vert^2)^{1/2}$, $C_4 = C_3\vee1$ and $C = 4(C_1 + C_2 + C_4)$, which are finite quantities, we get for any $n\geq1$, for any $t\geq0$,
\[
\mathbb{E}[ \vert Z_t^n \vert^2 \wedge 1 ] \leq Ct.
\]
Indeed, we have $(t^2\vert b_n \vert^2) \wedge 1 \leq (t^2 C_3^2) \wedge 1 \leq C_4 t$. The result follows from Cauchy-Schwartz inequality.
\end{proof}

\begin{lemma}\label{indep_limite}
Let $(\tau^n_t, Z^n_t)_{t\geq 0}$ be a sequence of L\'evy processes taking values in $\mathbb{R}_{+}\times \mathbb{R}$. Suppose that the sequence of subordinators $(\tau^n_t)_{t\geq 0}$ converges in law to $(\tau_t)_{t\geq 0}$ and $(Z_t^n)_{t\geq 0}$ converges to a Brownian motion $(Z_t)_{t\geq 0}$. Then $(\tau^n_t, Z^n_t)_{t\geq 0}$ converges in law to $(\tau_t, Z_t)_{t \geq 0}$ where $(\tau_t)_{t\geq0}$ and $(Z_t)_{t\geq0}$ are independent. 
\end{lemma}

\begin{proof}Denote by $\pi_n( \dr r, \dr x)$  the L\'evy measure of $(\tau_t^n, Z_t^n)_{t\geq 0}$ and define by $\chi(r,x)=(1\wedge r, -1\vee (x\wedge 1) )$ a truncation function on $\mathbb{R}_{+}\times \mathbb{R}$. It is continuous, bounded and equals $(r,x)$ on a neighborhood of $0$ in $\mathbb{R}_{+}\times \mathbb{R}$. Then for all $\alpha>0$, $\beta \in \mathbb{R}$, 
\[
\mathbb{E}[ \e^{-\alpha \tau_t^n +i \beta Z_t^n}]=\exp \big ( t \psi_n(\alpha, \beta) \big )
\]
where 
\[
\psi_n(\alpha, \beta)=  -\alpha b_n +i \beta \hat{b}_n -\frac{\sigma_n^2}{2} \beta^2+ \int_{\mathbb{R}_{+}\times \mathbb{R}} \e^{-\alpha r + i \beta x}-1 - (-\alpha, i \beta)\cdot \chi(r,x)  \; \pi_n (\dr r, \dr x). 
\]
We denoted by $(b_n, \hat{b}_n)\in \mathbb{R}_{+}\times \mathbb{R}$ the drift coefficient (which depends on the choice of the truncation function) and $\sigma_n^2$ the Brownian coefficient of $(Z_t^n)_{t\geq0}$. Then making $\alpha=0$ and $\beta=0$ in the expression we deduce that the  characteristics of $(\tau^n_t)_{t\geq 0}$ and $(Z^n_t)_{t\geq 0}$ are respectively $(b_n,\nu_n)$ and $(\hat{b}_n, \sigma_n^2, \mu_n)$, where $\nu_n$ and $\mu_n$ are defined as
\[
\nu_n(\dr r) := \int_{\mathbb{R}} \pi_n( \dr r, \dr x), \quad \mu_n(\dr x) := \int_{\mathbb{R}_+} \pi_n(\dr r, \dr x).
\]
By Theorem VII.2.9 p 396 \cite{js}, the convergence in law of $(\tau_t^n)_{t\geq0}$ and $(Z_t^n)_{t\geq0}$ implies that $b_n$, $\hat{b}_n$, $\int_{\mathbb{R}} 1\wedge r^2 \nu_n(\dr r)$ and $\sigma_n^2 + \int_{\mathbb{R}} 1\wedge x^2 \mu_n(\dr x)$ converges and that for all $g: \mathbb{R}_{+} \to \mathbb{R}$ and $\hat{g}:\mathbb{R}\to \mathbb{R}$  continuous bounded which are $0$ around $0$ 
\[
\int_{\mathbb{R}_+} g(r) \nu_n(\dr r)  \underset{n\to +\infty}{\longrightarrow} \int_{\mathbb{R}_+} g(r) \nu(\dr r), \ \ \mathrm{and} \ \ \int_{\mathbb{R}_+} \hat{g}(x) \mu_n(\dr x)  \underset{n\to +\infty}{\longrightarrow} 0.
\]
Here we denote by $\nu(\dr r)$ the L\'evy measure of the limit process $(\tau_t)_{t\geq 0}$. It follows that for all $\delta>0$, $\int_{\bbR} \mathbf{1}_{\vert x\vert \geq \delta} \mu_n(\dr x)$ goes to $0$ and which implies that
\begin{align}
\int_{\mathbb{R}_+\times \mathbb{R}} \mathbf{1}_{\vert x\vert \geq \delta} \,  \pi_n(\dr r , \dr x)  \underset{n\to +\infty}{\longrightarrow} 0. \label{dirac}
\end{align}
By Theorem VII.2.9 p 396 \cite{js} again, the convergence of $(\tau^n_t,Z^n_t)_{t\geq 0}$ will follow if we can show that \textit{(i)}
\[
\int_{\mathbb{R}_+\times \mathbb{R}} (1\wedge r )(-1\vee(x\wedge 1)) \, \pi_n( \dr r, \dr x)
\]
converges as $n\to\infty$, and that \textit{(ii)} for all continuous function $g: \mathbb{R}_+ \times \mathbb{R} \to \mathbb{R}_+$ which is constant outside a compact set and is $0$ around $0$, $\int g(r,x) \pi_n(\dr r, \dr x)$ converges. 

Regarding item \textit{(i)}, we set $f(r,x):= (1\wedge r )(-1\vee(x\wedge 1))$ and fix $\delta>0$. By \eqref{dirac}, and since $f$ is bounded, $\int_{\vert x\vert > \delta} \vert f(r,x) \vert \pi_n(\dr r, \dr x)$ vanishes as $n\to\infty$. It follows that
\[
\underset{n\to\infty}{\limsup} \bigg\vert \int_{\mathbb{R}_+\times \mathbb{R}} f(r,x) \pi_n(\dr r, \dr x) \bigg\vert  \leq \underset{n\to\infty}{\limsup}  \int_{\vert x \vert \leq \delta} \vert f(r,x)\vert \pi_n(\dr r, \dr x)  \leq \delta \, \underset{n\to\infty}{\limsup} \int 1\wedge r \, \nu_n(\dr r). 
\]
Now, remark that $b_n-\int_{\bbR} 1\wedge r \nu_n(\dr r)$ is the drift coefficient of the subordinator $(\tau_t^n)_{t\geq 0}$ and therefore, it must be positive. Since $b_n$ converges, it follows that 
\[
 0 \leq \limsup_{n\to\infty} \int_{\bbR} 1\wedge r\nu_n(\dr r) \leq \limsup_{n\to\infty}b_n < \infty.
\]
Making $\delta$ to $0$ shows that $\int_{\mathbb{R}_+\times \mathbb{R}} f(r,x) \pi_n( \dr r, \dr x)$ converges to $0$.

For item \textit{(ii)}, we consider a function $g : \mathbb{R}_+ \times \mathbb{R} \to \mathbb{R}_+$ which is constant outside a compact set and is $0$ around $0$, say on $K:=\{r \leq 1,  \vert x \vert \leq 1\}$. It is uniformly continuous and for all $\varepsilon>0$ there is some $1>\delta >0$ such that $\vert x \vert \leq \delta$ implies $\vert g(r,x)-g(r,0)\vert \leq \varepsilon$. Using \eqref{dirac} again, and using that $\int_{\bbR_+}g(r,0)\nu_n(\dr r) = \int_{\bbR_+\times \bbR}g(r,0)\mu_n(\dr r, \dr x)$ 
\begin{align*}
\underset{n\to\infty}{\limsup} &\bigg\vert \int_{\bbR_+ \times \bbR} g(r,x) \pi_n(\dr r, \dr x) - \int_{\bbR_+} g(r,0) \nu_n(\dr r) \bigg\vert \leq \underset{n\to\infty}{\limsup} \int_{\vert x \vert \leq \delta} \vert g(r,x)-g(r,0) \vert \pi_n(\dr r, \dr x) \\ 
&=  \underset{n\to\infty}{\limsup} \int_{K^c\cap\{\vert x \vert \leq \delta\}} \vert g(r,x)-g(r,0) \vert \pi_n(\dr r, \dr x)\leq \varepsilon \underset{n\to\infty}{\limsup}\, \nu_n( \mathbb{R}_+ \setminus [0,1]).
\end{align*}
Since $\limsup_{n\to\infty} \, \nu_n( \mathbb{R}_+ \setminus [0,1]) <+\infty$ and since $ \int_{\bbR_+}g(r,0) \nu_n(\dr r)$ converges to $ \int g(r,0) \nu (\dr r)$, we conclude that 
\[
 \int g(r,x) \pi_n(\dr r, \dr x) \underset{n\to +\infty}{\longrightarrow} \int g(r,0) \nu(\dr r). 
\]
To summarize we get that $(\tau_t^n, Z_t^n)_{t\geq 0}$ converges in law and the limiting characteristic Laplace-Fourier exponent is
\[
\psi (\alpha, \beta) = -\alpha b +i \beta \hat{b} -\frac{\sigma^2}{2} \beta^2+ \int_{\mathbb{R}_{+}} \big ( \e^{-\alpha r}-1 +\alpha 1\wedge r  \big )\; \nu (\dr r),
\]
where $b:= \lim_n b_n$, $\hat{b}:= \lim_n \hat{b}_n$ and $\sigma^2:=\lim_n \sigma_n^2 + \int 1\wedge x^2 \mu_n(\dr x)$. This is the characteristic exponent of $(\tau_t,Z_t)_{t\geq 0}$ where $\tau_t$ and $Z_t$ are independent.
\end{proof}

\section{Technical results on generalized one-dimensional diffusions}
\label{app:generaldiffusion}

In this appendix, we prove some technical results from Section \ref{section_ito_mc_kean}. 
We do not recall here the notation here, so we refer the reader to the introduction of Section~\ref{section_ito_mc_kean}.


\subsection{Local times and the occupation time formula}
\label{appendix_loc_times}

The goal of this section is to prove Proposition~\ref{prop_ok_gen} and Lemma~\ref{lemma_moy}.
We do not recall the statements and refer the reader to Section~\ref{section_ito_mc_kean}.
%
%
%

\begin{proof}[Proof of Proposition~\ref{prop_ok_gen}]
 \textit{Item (i).\ } Let $t\geq0$ and $h$ be a Borel function. Using a change of variables, we have
 \[
  \int_0^t h(X_s)\dr s = \int_0^{t}h(\sca^{-1}(B_{\rho_s}))\dr s = \int_0^{\rho_t}h(\sca^{-1}(B_s))\dr A_s^{\mm^\sca}.
 \]
By Corollary 2.13 in \cite[Ch. X]{ry}, it holds that for any $u\geq0$,
\[
 \int_0^{u}h(\sca^{-1}(B_s))\dr A_s^{\mm^\sca} = \int_{\mathbb{R}}h(\sca^{-1}(x))L_u^x \mm^\sca(\dr x), 
\]
and the first item follows by substituting $\rho_t$ in the preceding equation and performing a the change of variables $y = \sca^{-1}(x)$.

\bigskip\noindent
\textit{Item (ii).\ } Let us now show that $0$ is regular for $(X_t)_{t\geq0}$. More precisely, we will show that the time $t = 0$ is an accumulation point of $\mathcal{Z}_X = \{t\geq0, X_t = 0\} = \{t\geq0, B_{\rho_t} = 0\}$. First we remark that for every $(\mathcal{F}_t)_{t\geq0}$-stopping time $T$ such that $B_T = 0$, we have $A_{T+\epsilon}^{\mm^\sca}-A_T^{\mm^\sca} > 0$ for every $\epsilon > 0$. Indeed, the family $(L_{T+t}^x - L_T^x)_{t\geq0, x\in\bbR}$ is the family of local times of the Brownian motion $(B_{T + t})_{t\geq0}$. Therefore, $L_{T+\epsilon}^0 - L_T^0 > 0$ for every $\epsilon > 0$ and by the continuity of the Brownian local times in the space variable, $L_{T+\epsilon}^x - L_T^x > 0$ in a neighbourhood of $0$. Finally, since $0\in \mathrm{supp}(\mm^\sca)$, we get
\[
 A_{T+\epsilon}^{\mm^\sca}-A_T^{\mm^\sca} = \int_{\bbR}(L_{T+\epsilon}^x - L_T^x)\mm^\sca(\dr x) > 0 
\]
for every $\epsilon > 0$. Hence, for such a stopping time $T$, we have $\rho_{A_T^{\mm^\sca}} = \inf\{s > 0, \: A_s^{\mm^\sca}  > A_T^{\mm^\sca}\} = T$. Let us define for any $t \geq 0$, $d_t = \inf\{s > t, \: B_s = 0\}$ the first zero of $(B_u)_{u\geq0}$ after time $t$, which is an $(\mathcal{F}_t)_{t\geq0}$-stopping time. Consider a decreasing sequence $(t_n)_{n\in\mathbb{N}}$ which converges to $0$ as $n\to\infty$. Then for any $n\in\mathbb{N}$, $d_{t_n}$ is a stopping time such that $B_{d_{t_n}} = 0$ and therefore $\rho_{A_{d_{t_n}}^{\mm^\sca}} = d_{t_n}$. Let us set $(u_n)_{n\in\mathbb{N}} = (A_{d_{t_n}}^{\mm^\sca})_{n\in\mathbb{N}}$, then it is clear that it is a non-increasing sequence of points in $\mathcal{Z}_X$. Since $d_{t_n} \to 0$ as $n\to\infty$ and since $(A_{t}^{\mm^\sca})_{t\geq0}$ is continuous, it comes that $u_n \to 0$ as $n\to\infty$ which shows that $0$ is an accumulation point of $\mathcal{Z}_X$.

\bigskip\noindent
\textit{Item (iii).\ } We show item \textit{(ii)} and start by showing that $(L_{\rho_t}^0)_{t\geq0}$ is continuous. Let $\Sigma_L$ be the complement of the open set $\bigcup_{s\geq0}]\tau_{s-}^B, \tau_s^B[$ where $(\tau_t^B)_{t\geq0}$ is the right-continuous inverse of $(L_t^0)_{t\geq0}$. The set $\Sigma_L$ a.s. coincides with the set $\mathcal{Z}_B = \{t\geq0, \: B_t = 0\}$. Let $\Sigma_A$ be the complement of $\bigcup_{s\geq0}]\rho_{s-}, \rho_s[$. It is shown in \cite[Ch. X Proposition 2.17]{ry} that $\Sigma_A$ a.s. coincides with the set $\Gamma_A = \{t\geq0, \: B_t \in \mathrm{supp}(\mm^\sca)\}$ and with the support of the measure $\dr A_t^{\mm^\sca}$. Since $0\in \mathrm{supp}(\mm^\sca)$, it comes that $\Gamma_L \subset \Gamma_A$ a.s. and therefore, we have
\[
 \bigcup_{s\geq0}]\rho_{s-}, \rho_s[\: \subset\: \bigcup_{s\geq0}]\tau_{s-}^B, \tau_s^B[.
\]
Since $\bigcup_{s\geq0}]\tau_{s-}^B, \tau_s^B[$ also coincides with the flat sections of $(L_t^0)_{t\geq0}$ we get that for any $t\geq0$ such that $\rho_{t-} < \rho_t$, then $L_{\rho_{t-}}^0 = L_{\rho_t}^0$ which shows that $(L_{\rho_t}^0)_{t\geq0}$ is continuous. The fact that $(L_t^0)_{t\geq0}$ remains an additive functional after time-change is rather classical. It follows from the fact that $(\rho_t)_{t\geq0}$ is itself the inverse of a continuous additive functional, and $(L_t^0)_{t\geq0}$ is a \textit{strong} additive functional. We refer to Revuz-Yor \cite[Ch. X, Propositions 1.2 and 1.3]{ry} for more details.

\medskip
Let us show that the support of the measure $\dr L_{\rho_t}^0 = \mu(\dr t)$ is included in the closure of $\mathcal{Z}_X$. We introduce the measure $\nu(\dr t)$ which is the pushforward measure of $\mu$ by $(\rho_t)_{t\geq0}$. It is the Stieltjes measure associated to the non-decreasing process $(L_{\rho_{A_t^{\mm^\sca}}}^0)_{t\geq0}$. Let $t\geq0$ be fixed, then if $A_{t+\epsilon} > A_t$ for every $\epsilon > 0$, we have $L_{\rho_{A_t^{\mm^\sca}}}^0 = L_t^0$. On the other hand, if $t$ is such that there exists $\epsilon >0$ such that $A_t^{\mm^\sca} = A_{t+\epsilon}^{\mm^\sca}$, then $A_t^{\mm^\sca}$ is a jumping time of $(\rho_s)_{s\geq0}$ and since $(L_s^0)_{s\geq0}$ is constant on $[\rho_{u-}, \rho_u]$ for every $u$ such that $\Delta\rho_u > 0$, we get $L_t^0 = L_{\rho_{A_t^{\mm^\sca}}}^0$. In the end, we see that $\nu$ is the Stieltjes measure associated to $(L_t^0)$ and, applying the change of variables, we get
\[
 \int_0^t \bm{1}_{\{X_s \neq 0\}}\dr L_{\rho_s}^0 = \int_{[0,t]} \bm{1}_{\{B_{\rho_s} \neq 0\}}\mu(\dr s) = \int_{\rho([0,t])}\bm{1}_{\{B_{s} \neq 0\}} \nu(\dr s) \leq \int_0^{\rho_t}\bm{1}_{\{B_{s} \neq 0\}}\dr L_{s}^0 = 0.
\]
This shows that the support of the measure $\mu$ is included in the closure of $\mathcal{Z}_X$.

We now show the converse and we will assume that $0$ is an instantaneous point for $(X_t)_{t\geq0}$, as the proof is easier when $0$ is a holding point. In this case, the closure of $\mathcal{Z}_x$ is a perfect set with empty interior. For any $t\geq0$, we denote by $d_t^X = \inf\{s > t, \: X_t = 0\}$ the first zero of $(X_s)_{s\geq0}$ after $t$. Then for any $t\geq0$, $B_{\rho_{d_t^X}} = 0$ and since $\rho_{d_t^X}$ is an $(\mathcal{F}_t)$-stopping time, $L_{\rho_{d_t} + \epsilon}^0 > L_{\rho_{d_t}}^0$ for every $\epsilon > 0$. Since $(A_t^{\mm^\sca})_{t\geq0}$ is continuous, $(\rho_t)_{t\geq0}$ is increasing and therefore $L_{\rho_{d_t +
 \epsilon}}^0 > L_{\rho_{d_t}}^0$ for any $\epsilon > 0$. This shows that almost surely, for any $t\in\mathbb{Q}_+$, $d_t$ belongs to the support of $\mu$.  Consider now some $t$ in the closure of $\mathcal{Z}_x$ and some $\epsilon > 0$. Then since $\overline{\mathcal{Z}_X}$ has empty interior, there exists $r\in\mathbb{Q}\cap[t-\epsilon, t)$ such that $r\notin\overline{\mathcal{Z}_X}$ and $d_r \leq t$. Hence $t$ is a limit of points belonging to the support of $\mu$, which is a closed set, and therefore $t$ belongs to this set.

\bigskip\noindent
\textit{Item (iv).\ } We now consider $(\tau_t)_{t\geq0}$ the right-continuous inverse of $(L_{\rho_t}^0)_{t\geq0}$. Let us show that for any $t\geq0$, $\tau_t = A^{\mm^\sca}_{\tau_t^B}$. The continuous non-decreasing process $(A_t^{\mm^\sca})_{t\geq0}$ is also the right-continuous inverse of $(\rho_t)_{t\geq0}$ so that we have for any $t>0$,
\[
 A^{\mm^\sca}_{\tau_t^B} = \inf\{s > 0, \: \rho_s > \tau_t^B\} = \inf\{s > 0, \: L_{\rho_s}^0 > t\} = \tau_t.
\]
For any $t\geq0$, $\tau_t^B$ is a stopping time such that $B_{\tau_t^B} = 0$ and thus, by the argument from the second step, we have $\rho_{A^{\mm^\sca}_{\tau_t^B}} = \rho_{\tau_t} = \tau_t^B$.
\end{proof}

\begin{proof}[Proof of Lemma~\ref{lemma_moy}]
 We introduce the Lévy process $Z_t^g = \int_0^{\tau_t}((g\circ\sca)\times f)(X_s) \dr s$. It is clear that, by assumption and by the occupation time formula from item~\textit{(i)} of Proposition~\ref{prop_ok_gen}, this process is well defined. Let us denote by $\nu_g(\dr z) = \nn\big(\int_0^{\ell} ((g\circ\sca)\times f\big)(\epsilon_s) \dr s \in \dr z)$ the Lévy measure of $(Z_t)_{t\geq0}$. On one hand, we have
 \[
  \mathbb{E}\left[Z_t\right] = t \int_{\bbR}z \nu_g(\dr z) = t \nn\left(\int_0^{\ell} ((g\circ\sca)\times f)(\epsilon_s) \dr s\right).
 \]
On the other hand, it holds by items \textit{(i)} and \textit{(iv)} of Proposition~\ref{prop_ok_gen} that 
\[
  \mathbb{E}\left[Z_t\right] = \mathbb{E}\left[\int_{\bbR}L_{\rho_{\tau_t}}^x g(x)\mm^f(\dr x)\right] = \int_{\bbR}\mathbb{E}\big[L_{\tau_t^B}^x\big]g(x)\mm^f(\dr x).
 \]
By Ray-Knight's theorem, see~\cite[Chp.~XI, Thm.~2.3]{ry}, for any fixed $t\geq0$, the processes $(L_{\tau_t^B}^x)_{x\geq0}$ and $(L_{\tau_t^B}^x)_{x\leq0}$ are two independent squared Bessel processes of dimensions $0$ starting at $t$ and therefore $\mathbb{E}[L_{\tau_t^B}^x] = t$ for any $x\in\bbR$, which completes the proof.
\end{proof}

\subsection{Convergence of strings under the Brownian excursion measure}\label{appendix_strings}

The goal of this section is to prove the technical Lemmas~\ref{lemconv_excursion} and~\ref{conv_domination} about strings with regular variation.
Again, we do not recall the notation nor the statements of the results; we refer the reader to Section~\ref{section_strings}.
In particular $\mm$ is a string of regular variation with regular variation of index $\alpha \in (0,2)$ and $\mm_h$ denote its rescaled version.

Let us start with a technical lemma in the case $\alpha \in [1,2)$, which is an almost direct application of Potter's bound~\cite[Thm.~1.5.6]{bgt89}.
\begin{lemma}\label{Potter'sbounds}
Let $\mm$ be a string with regular variation of index $\alpha \in [1,2)$. 
Then for any $\eta>0$, there exists a constant $C >0$, such that for any $x>0$ and any $h\in(0,1)$,
\begin{equation}\label{bound_mh_1_2}
 |\mm_h(x)| \leq C x^{1/\alpha - 1 -\eta}\vee x^{1/\alpha - 1 + \eta}.
\end{equation}
\end{lemma}
\begin{proof}[Proof of Lemma \ref{Potter'sbounds}]
When $\alpha = 1$, this is a direct application of \cite[Thm.~3.8.6]{bgt89}. 
When $\alpha \in (1,2)$, we have that there is a constant $C_{\alpha}$ such that 
\[
\mm(\infty) - \mm(y) \leq C_\alpha \Lambda(y) y^{1/\alpha-1} \qquad \text{for any $y>0$},
\]
using also that $1/\alpha-1 <0$ so the upper bound diverges as $x\downarrow 0$.
It then simply remains to show that
$\frac{\Lambda(x/h)}{\Lambda(1/h)} \leq C x^{-\eta} \vee x^{\eta}$ for any $x>0$, which is exactly the content of Potter's bound~\cite[Thm.~1.5.6]{bgt89}.
\end{proof}
%
%
%
%

\begin{proof}[Proof of Lemma \ref{lemconv_excursion}]
 When $\alpha \in (0,1)$ there is nothing to prove since $x^{1/\alpha-1}= c_\alpha \int_0^x y^{1/\alpha-2} \dr y$ is the cumulant  to the Radon measure $x^{1/\alpha-2} \dr x$ on $\bbR_+$ and $x\mapsto g(x)$ is continuous with compact support. When $\alpha\in[1,2)$, then for any $\delta \in(0,1)$, then we have
\[
\lim_{\delta \downarrow 0} \int_{\delta}^{\infty}\bm{\mathrm{L}}_t^x \mm_h(\dr x)
= c_\alpha\int_{\delta}^{\infty}\bm{\mathrm{L}}_t^x x^{1/\alpha-2} \dr x.
\]

\noindent
Again, this is a direct consequence of the fact that the measure $\mm_h$ restricted to $(\delta, \infty)$ converges weakly to the measure $x^{1/\alpha-2} \dr x$ restricted to $(\delta, \infty)$, and that $x\mapsto \bm{\mathrm{L}}_t^x$ is continuous with compact support. 

Now consider some $\gamma \in(0,1/2)$ such that $1/\alpha - 1 + \gamma > 0$, which is possible since $\alpha<2$. Since $x\mapsto g(x)$ is Hölder of order $\gamma$, and $g(0)=0$, there exists a constant $D > 0$  such for any $x \in (0, \delta)$, $g(x) \leq D x^{\gamma}$. Therefore we have
\[
 c_\alpha\int_0^{\delta}g(x) x^{1/\alpha-2} \dr x \leq  D'  \delta^{1/\alpha - 1 +\gamma}, 
\]
for some constant $D'>0$. Now pick some $\eta \in(0,1/\alpha - 1 +\gamma)$, so by \eqref{bound_mh_1_2} there exists $C > 0$ such that $|\mm_h(x)| \leq C x^{1/\alpha - 1 -\eta}$ for any $x\in(0,\delta]$. Therefore we have
\[
 \int_0^\delta g(x)\mm_h(\dr x) \leq D\int_0^\delta x^{\gamma} \mm_h(\dr x) \leq \delta^{\gamma}|\mm_h(\delta)| + \gamma \int_0^{\delta}x^{\gamma - 1}|\mm_h(x)|\dr x \leq C\, \delta^{1/\alpha - 1 +\gamma - \eta},
\]
for some constant $C>0$, the second inequality being obtained by integration by parts. Therefore, we have just proved that for all $\delta\in(0,1)$, we have
\[
 \limsup_{h\to0}\left|\int_{\bbR_+}g(x) \mm_h(\dr x) - c_\alpha\int_{\bbR_+}g(x) x^{1/\alpha -2}\dr x\right| \leq D'\delta^{1/\alpha - 1 +\gamma} + C\delta^{1/\alpha - 1 +\gamma - \eta}.
\]
Letting $\delta \downarrow 0$ proves the result.
\end{proof}

%
\begin{proof}[Proof of Lemma~\ref{conv_domination}]
Before we start with the proof of teh four items, let us recall William's decomposition of the excursion measure, see~\cite[Chp.~XII, Thm.~4.5]{ry}.
Take $(U_t)_{t\geq0}$ and $(\Tilde{U}_t)_{t\geq0}$ two independent 3-dimensional Bessel processes starting from $0$, defined on some probability space. For any $c>0$, let $T_c$ and $\Tilde{T}_c$ be their respective first hitting of $c$. We define the process $(Q_t^c)_{t\geq0}$ by
\[
Q_t^c =
\left \{
\begin{array}{lll}
 U_t  \quad &\text{if } 0\leq t \leq T_c  \\
 c - \Tilde{U}_{t- T_c} &\text{if } T_c \leq t \leq T_c + \Tilde{T}_c \\
 0 \quad &\text{if } t\geq \Tilde{T}_c.
 \end{array} \right.
\]
Then, the law of a positive excursion conditioned to $M(\varepsilon) = c$ is equal to the law of $(Q_t^c)_{t\geq0}$, and for any measurable set $\Gamma$, we have
\begin{equation}
\label{william}
 \nn^B_+(\Gamma) = \frac{1}{2}\int_0^\infty\mathbb{P}\big((Q_t^y)_{t\geq0} \in \Gamma\big)y^{-2} \dr y.
\end{equation}
Then, denoting by $L^x_t(X)$ the local time in $x$ at time $t$ of a process $(X_s)_{s\geq 0}$, we have
\begin{equation}\label{tpslocbound}
L_{T_y+ \tilde{T}_y}^x \big ( Q^y \big) = L_{T_y}^x ( U )+ L_{\tilde{T}_y}^x (\tilde{U}) \leq  L_{\infty}^x ( U )+ L_{\infty}^x (\tilde{U}) =:R_x.
\end{equation}
Recalling the variant of the second Ray-Knight's theorem, see \cite[RK2.a)]{y}, $(L_{\infty}^x ( U ) )_{x\geq0}$ and $(L_{\infty}^x ( \tilde{U} ) )_{x\geq0}$ are two independant 2-dimensional squared Bessel process started at $0$ and so $(R_u)_{u\geq0}$ is a square Bessel process of dimension $4$ started at $0$.

\medskip\noindent
\textit{Item (i).\ }
Consider the first point and take $\alpha \in (0,1)$, then combining~\eqref{william} and~\eqref{tpslocbound} we have
\begin{align*}
\nn^B_+ \left [  \Big (  \int_{\bbR_+} \bm{\mathrm{L}}_\ell^x \mm_h(\dr x) \Big ) \mathbf{1}_{\mathbf{A}_\delta } \right ]
 &=\frac{1}{2}\int_0^\delta \bbE \left [  \int_0^y L^u_{T_y+ \tilde{T}_y}(Q^y) \mm_h(\dr u) \right ]y^{-2} \dr y\\
&\leq \frac{1}{2}\int_0^\delta \bbE \left [ \int_0^y R_u \mm_h(\dr u) \right ]y^{-2} \dr y,  \,.
\end{align*}
Since $\mathbb{E}[R_u]=4u$, this is bounded by
\begin{align*}
2\int_0^\delta  \left [  \int_0^y u \; \mm_h(\dr u) \right ]y^{-2} \dr y 
=  2 \int_0^\delta u \int_u^{+\infty} \frac{\dr y}{y^2} \mm_h(\dr u) = 2 \int_0^\delta \mm_h(\dr u ) = 2 \mm_h(\delta).
\end{align*}
The first point of the Lemma follows since $\underset{h \to 0}{\lim}\mm_h(\delta)=\delta^{1/\alpha -1}$ and $\alpha \in (0,1)$.

\medskip\noindent
\textit{Item (ii).\ } Take  $\alpha \in [1,2)$ and consider $\varphi = x^2$ so that 
\begin{align*}
\nn^B_+ \bigg [ \sup_{h\in (0,1]}\Big (  \int_{\bbR_+} \bm{\mathrm{L}}_\ell^x \mm_h(\dr x) \Big )^2 \mathbf{1}_{\mathbf{A}_\delta } \bigg ] &\leq \frac{1}{2}\int_0^\delta \bbE \bigg [ \sup_{h \in (0,1]} \Big ( \int_0^y R_u \mm_h(\dr u) \Big )^2 \bigg ]y^{-2} \dr y.
\end{align*}
Fix $\eta>0$ (how small depends on $\alpha$).
Since $(R_u)_{u\geq 0}$ is a sum of independent square of Brownian process, by Kolmogorov's regularity criterion (see \cite[Thm.~2.1]{ry}), we have that $R_u \leq C_\eta  u^{1-\eta}$ for all  $u \in [0, 1]$, where  $C_\eta>0$ is a square integrable positive random variable.
So for $\delta \in (0,1]$, 
\begin{equation}
\label{truc2}
\nn^B_+ \bigg [ \sup_{h\in (0,1]}\Big (  \int_{\bbR_+} \bm{\mathrm{L}}_\ell^x \mm_h(\dr x) \Big )^2 \mathbf{1}_{\mathbf{A}_\delta } \bigg ] 
\leq \frac{1}{2}\bbE [C_\eta^2] \int_0^\delta \sup_{h \in (0,1]} \Big (\int_0^y u^{1-\eta} \mm_h(\dr u) \Big)^2 y^{-2}\dr y. 
\end{equation}
Moreover, since $ \int_0^y u^{1-\eta} \; \mm_h(\dr u) = y^{1-\eta}\mm_h( y) -(1-\eta)\int_0^y u^{-\eta}\mm_h(u) \dr u$, we have
\begin{equation*}
\Big (\int_0^y u^{1-\eta} \mm_h(\dr u) \Big)^2 \leq 4 \Big ( y^{2-2\eta} \mm_h(y)^2 +  \Big ( \int_0^y u^{-\eta}\vert \mm_h(u) \vert \dr u  \Big)^2 \Big ).
\end{equation*}
Take $0<2 \eta <1/\alpha -1/2$ (which is possible since  $\alpha<2$), by Lemma \ref{Potter'sbounds} we get that $|\mm_h(u)| \leq c_\eta u^{1/\alpha - 1 -\eta}$ uniformly in $h\in(0,1)$ and $u\in (0,1]$, for some constant $c_{\eta}$. Since $1/\alpha-1-2\eta >-1$ there is a constant $\hat{c}_{\eta}>0$ such that 
\[
\sup_{h\in (0,1]}\Big (\int_0^y u^{1-\eta} \mm_h(\dr u) \Big)^2 \leq \hat{c}_{\eta} y^{2/\alpha -4\eta}, \quad \forall y \in [0,1].
\]
Hence for $\eta>0$ small enough, the left-hand side of~\eqref{truc2} is bounded by a constant (that depends on $\eta$) times $\delta^{2/\alpha--1-4\eta}$, which goes to $0$ as $\delta\downarrow 0$.


\medskip\noindent
\textit{Item (iii).\ } We proceed in the same vein. Similarly as above, we have
\[
\nn_+^B \bigg [ \underset{h\in (0,1)}{\sup} \Big(  \int_0^{+\infty} \bm{\mathrm{L}}_\ell^x \mm_h(\dr x) \Big)   \mathbf{1}_{\mathbf{A}_\delta^c}\bigg ] \leq  \frac{1}{2} \int_\delta^\infty  \bbE \bigg [ \sup_{h \in (0,1)} \Big (  \int_0^y R_u \mm_h(\dr u) \Big ) \bigg ]y^{-2} \dr y \,.
\]
Fix $\eta>0$.
By time-inversion, $(\tilde{R}_u)_{u\geq 0}:=(u^2 R_{1/u})_{u\geq0}$ has the same law at $(R_u)_{u\geq 0}$, hence Kolmogorov's regularity theorem applied both to $(\tilde R_u)_{u\leq 1}$ and $(R_u)_{u\leq 1}$ gives that 
\[
 R_u \leq C_\eta \, (u^{1-\eta}\vee u^{1+\eta}), \quad \forall u\geq 0 \,,
\]
with $C_\eta$ some integrable random variable.
Thus we have, 
 \[
 \nn_+^B \bigg [ \underset{h\in (0,1]}{\sup} \Big(  \int_0^{+\infty} \bm{\mathrm{L}}_\ell^x \mm_h(\dr x) \Big)   \mathbf{1}_{\mathbf{A}_\delta^c}\bigg ] \leq \frac{1}{2}\bbE [C_\eta] \int_\delta^{+\infty} \sup_{h \in (0,1]} \Big ( \int_0^y u^{1-\eta}\vee u^{1+\eta}\mm_h(\dr u) \Big )y^{-2}\dr y
 \]
Then we write for all $y\geq\delta$,
\[
 \int_0^y u^{1-\eta}\vee u^{1+\eta}\mm_h(\dr u) = \int_0^{y\wedge 1} u^{1-\eta}\mm_h(\dr u) + \int_1^{y\vee 1} u^{1+\eta} \mm_h (\dr u).
\]
If $\eta$ is small enough so that $2\eta < 1-1/\alpha$ and $2\eta < 1/\alpha$, which is possible since $\alpha \in (1,2)$, it follows from Lemma ~\ref{Potter'sbounds} that $\vert \mm_h(y) \vert\leq C y^{1/\alpha-1-\eta}\vee y^{1/\alpha-1+\eta}$. 
Then, recalling that $\mm_h(0) = 0$, we see that
\[
 \int_0^{y\wedge 1} u^{1-\eta}\mm_h(\dr u) \leq (y\wedge1)^{1-\eta} \mm_h(y\wedge1) \leq C (y\wedge1)^{1/\alpha-2\eta} \leq C.
\]
Similarly we have
\[
 \int_1^{y\vee 1} u^{1+\eta} \mm_h (\dr u) \leq (y\vee1)^{1+\eta}( \mm_h(y\vee1) - \mm_h(1)) \leq C(y\vee1)^{1/\alpha+2\eta} \,.
\]
In the end, we have for any $y\geq\delta$ the bound $\int_0^y u^{1-\eta}\vee u^{1+\eta}\mm_h(\dr u) \leq C + C(y\vee1)^{1/\alpha+2\eta}$ which is integrable with respect to the measure $y^{-2}\dr y$ on $[\delta,\infty)$ since $2-1/\alpha - 2\eta > 1$. This concludes the proof item \textit{(iii)}.

\medskip\noindent
\textit{Item (iv).\ } This point is similar.  We write similarly as above
\[
\begin{split}
\nn_+^B \bigg [ \underset{h\in (0,1)}{\sup} \Big(  \int_0^{1} \bm{\mathrm{L}}_\ell^x \mm_h(\dr x) \Big)   \mathbf{1}_{\mathbf{A}_\delta^c}\bigg ] 
&= \frac{1}{2}\int_\delta^{\infty}  \bbE \Big [ \underset{h\in (0,1)}{\sup} \int_0^{1\wedge y} R_u \mm_h(\dr u) \Big ] y^{-2}\dr y \\
& \leq  \frac{1}{2}\bbE [C_\eta] \int_\delta^{+\infty} \sup_{h \in (0,1]} \Big ( \int_0^{y\wedge 1} u^{1-\eta}\mm_h(\dr u) \Big )y^{-2}\dr y\,,
\end{split}
\]
which is finite as in the case of the previous item.
\end{proof}

\subsection*{Acknowledgements}

We warmly thank Nicolas Fournier and Thomas Duquesne for fruitful discussions and remarks.

{\small
\bibliographystyle{abbrv}
\bibliography{biblio.bib}
}

{
  \bigskip
  \footnotesize

  Q.~Berger, \textsc{Sorbonne Universit\'e, Laboratoire de Probabilit\'es, Statistique et Modélisation, 75005 Paris, France} and 
  \textsc{DMA, École Normale Supérieure, Université PSL, 75005 Paris, France.}\par\nopagebreak
  \textit{E-mail address}: \texttt{quentin.berger@sorbonne-universite.fr}

  \medskip

  L.~B\'ethencourt, \textsc{Sorbonne Universit\'e, Laboratoire de Probabilit\'es, Statistique et Modélisation, 75005 Paris, France.}\par\nopagebreak
  \textit{E-mail address}: \texttt{loic.bethencourt@sorbonne-universite.fr}

  \medskip

  C.~Tardif \textsc{Sorbonne Universit\'e, Laboratoire de Probabilit\'es, Statistique et Modélisation, 75005 Paris, France.}\par\nopagebreak
  \textit{E-mail address}: \texttt{camille.tardif@sorbonne-universite.fr}
}

\end{document}